\documentclass[twoside,11pt,reqno]{amsart}
\usepackage{amsmath,amssymb,amscd,mathrsfs,epic,wasysym,latexsym,tikz,mathrsfs,cite,enumerate,enumitem,xspace,stmaryrd,needspace,genyoungtabtikz}
\usepackage[noBBpl]{mathpazo}
\usepackage[hidelinks,bookmarks=false,pagebackref,bookmarks=true]{hyperref}
\usepackage[capitalise,noabbrev,english]{cleveref}
\usepackage{pb-diagram}
\usepackage[all]{xy}
\usepackage{mathdots,mathtools}
\usepackage[a4paper,margin=2cm]{geometry}

\makeatletter

\def\author@andify{%
  \nxandlist {\unskip ,\penalty-1 \space\ignorespaces}%
    {\unskip {} \@@and~}%
    {\unskip \penalty-2 \space \@@and~}%
}

\@namedef{subjclassname@2020}{\textup{2020} Mathematics Subject Classification}
\makeatother

\numberwithin{equation}{section}

\allowdisplaybreaks

\crefname{thm}{Theorem}{Theorems}
\crefname{propn}{Proposition}{Propositions}
\crefname{lemma}{Lemma}{Lemmas}
\crefname{cory}{Corollary}{Corollaries}
\crefname{conj}{Conjecture}{Conjectures}
\Crefname{thm}{Theorem}{Theorems}
\Crefname{propn}{Proposition}{Propositions}
\Crefname{lemma}{Lemma}{Lemmas}
\Crefname{cory}{Corollary}{Corollaries}
\Crefname{conj}{Conjecture}{Conjectures}

\newtheoremstyle{mattthm}{}{}{\slshape}{}{\bfseries}{.}{ }{}

\theoremstyle{mattthm}
\newtheorem{lemma}{Lemma}[section]
\newtheorem{propn}[lemma]{Proposition}
\newtheorem{thm}[lemma]{Theorem}
\newtheorem{cory}[lemma]{Corollary}

\newtheorem*{mainthm}{Main Theorem}

\newtheoremstyle{mattdef}{}{}{}{}{\bfseries}{.}{ }{}

\theoremstyle{mattdef}
\newtheorem{rmk}[lemma]{Remark}

\newtheorem{Example}[lemma]{Example}
\newtheorem*{eg}{Example}

\let\<\langle
\let\>\rangle

\newcommand\Comment[2][\relax]{\space\par\medskip\noindent%
   \fbox{\begin{minipage}{\textwidth}\textbf{Comment\ifx\relax#1\else---#1\fi}\newline%
        #2\end{minipage}}\medskip
}

\def\um{\underline{m}}
\def\un{\underline{n}}

\def\bi{\text{\boldmath$i$}}
\def\bj{\text{\boldmath$j$}}

\def\bg{\text{\boldmath$g$}}
\def\bz{\text{\boldmath$z$}}

\def\bla{\text{\boldmath$\lambda$}}
\def\bmu{\text{\boldmath$\mu$}}
\def\bnu{\text{\boldmath$\nu$}}

\def\pmod#1{\text{ }(\text{\rm mod } #1)\,}

\newcommand\tens{\mathbin{\circ}}
\newcommand{\Hom}{\operatorname{Hom}}

\newcommand{\id}{\operatorname{id}}
\def\sgn{\mathtt{sgn}}

\newcommand{\Z}{\mathbb{Z}}

\newcommand{\N}{\mathbb{N}}
\newcommand{\0}{{\bar 0}}
\renewcommand{\1}{{\bar 1}}
\def\eps{{\varepsilon}}
\def\phi{{\varphi}}

\newcommand{\zz}{{\textsf{z}}}
\newcommand{\ze}{{\textsf{e}}}

\newcommand{\za}{{\textsf{a}}}

\newcommand{\zu}{{\mathsf{u}}}

\newcommand{\F}{{\mathbb F}}
\newcommand\bbf{\mathbb F}
\newcommand\bbc{\mathbb C}

\newcommand{\ga}{\gamma}

\newcommand{\la}{\lambda}

\newcommand{\al}{\alpha}
\newcommand{\be}{\beta}

\def\Si{\mathfrak{S}}
\newcommand{\si}{\sigma}
\newcommand{\om}{\omega}

\newcommand{\de}{\delta}

\newcommand{\ka}{\kappa}
\newcommand{\T}{\mathcal{T}}

\def\Mtype{\mathtt{M}}
\def\Qtype{\mathtt{Q}}

\def\id{\mathop{\mathrm {id}}\nolimits}

\newcommand{\Ind}{{\mathrm {Ind}}}

\newcommand{\Mor}{{\mathrm {Mor}}}
\newcommand{\sMor}{{\mathrm {sMor}}}

\newcommand{\Res}{{\mathrm {Res}}}

\newcommand{\C}{{\mathbb C}}
\newcommand{\Q}{{\mathbb Q}}
\newcommand{\CL}{{\mathcal C}}

\newcommand{\Zig}{{\sf A}}

\def\Par{{\mathscr P}}

\def\x{x}
\def\W{{\sf W}}
\def\y{y}

\def\iso{\stackrel{\sim}{\longrightarrow}}

{\catcode`\|=\active
  \gdef\set#1{\mathinner{\lbrace\,{\mathcode`\|"8000%
  \let|\midvert #1}\,\rbrace}}
}
\def\midvert{\egroup\mid\bgroup}

\colorlet{darkgreen}{green!50!black}
\tikzset{dots/.style={very thick,loosely dotted},
         greendot/.style={fill,circle,color=darkgreen,inner sep=1.5pt,outer sep=0},
         blackdot/.style={fill,circle,color=black,inner sep=1.5pt,outer sep=0},
         graydot/.style={fill,circle,color=gray,inner sep=1.1pt,outer sep=0}
}
\def\greendot(#1,#2){\node[greendot] at(#1,#2){}}
\def\blackdot(#1,#2){\node[blackdot] at(#1,#2){}}
\def\graydot(#1,#2){\node[graydot] at(#1,#2){}}

\newenvironment{braid}{
  \begin{tikzpicture}[baseline=6mm,black,line width=1pt, scale=0.5,
                      draw/.append style={rounded corners},
                      every node/.append style={font=\fontsize{5}{5}\selectfont}]%
  }{\end{tikzpicture}
}

\Yboxdim{20pt}

\newcommand\hf{1.42}\newcommand\xhf{.42}

\newcommand\Sp[1]{\mathrm{S}(#1)} 
\newcommand\SpC[1]{\mathrm{S}_\C(#1)} 
\newcommand\Dm[1]{\mathrm{D}(#1)} 
\newcommand\Em[1]{\mathrm{E}(#1)}
\newcommand\sPar{{\mathscr P_0}}
\newcommand\rPar{{\mathscr{RP}_p}}
\newcommand\ppmod[1]{\ (\operatorname{mod}\,#1)}

\newcommand\pfr{$p'$}
\newcommand\chm[2]{[#1:#2]}
\newcommand\ztz{\mathbb Z/2\mathbb Z}
\newcommand\fg[2]{F(#1,#2)}
\newcommand\li{linearly independent\xspace}
\newcommand\triv[1]{\mathbf{1}_{#1}}

\newcommand\gth[1]{\underline{g}(#1)}

\newcommand\ket[1]{|#1\rangle}
\newcommand\ch[1]{\chi^#1}
\newcommand\lset[2]{\left\{\left.#1\ \right|\ \smash{#2}\right\}}
\newcommand\con{cited~on~p.~}
\newcommand\cons{cited~on~pp.~}

\newcommand\dd{D^\ast}

\renewcommand*{\backref}[1]{}
\renewcommand*{\backrefalt}[4]{%
\ifcase #1 %
\hspace*{\fill}{\small [no~citations.]}%
\or
\hspace*{\fill}{\small [\con#2]}%
\else
\hspace*{\fill}{\small [\cons#2]}%
\fi
}

\renewcommand\iff{if and only if\xspace}
\newcommand\domby{\trianglelefteqslant}
\newcommand\domsby{\vartriangleleft}
\newcommand\dom{\trianglerighteqslant}
\newcommand\doms{\vartriangleright}
\newcommand\ndom{\ntrianglerighteqslant}
\newcommand\ndomby{\ntrianglelefteqslant}
\renewcommand\unlhd\domby

\makeatletter
\newcommand*{\bktq}{%
  \mathrel{%
    \mathpalette\@blackleq{}%
  }%
}
\newcommand*{\@blackleq}[1]{%
  \vcenter{%
    \m@th
    \setbox0=\hbox{$#1\mkern3mu$}%
    \setbox2=\hbox{$#1\vcenter{}$}%
    \setbox4=\hbox{\raisebox{1.1pt}[.2pt][.2pt]{$#1\leqslant$}}%
    \hbox{$#1\blacktriangleleft$}%
    \nointerlineskip
    \kern\wd0 %
    \copy4 %
  }%
}
\makeatother

\newcommand\loccit{[\textit{loc.~cit.}]\xspace}

\newcommand\sucq\succcurlyeq
\newcommand\preq\preccurlyeq

\newcommand\bkp{\backrefprint\renewcommand\con{}\renewcommand\cons{}}

\newcommand\bbq{\mathbb Q}
\newcommand\compn{\mathscr C}

\newcommand\calf{\mathscr F}
\newcommand\calp{\mathscr P}

\newcommand\blk[2]{\mathcal B^{#1,#2}}

\newcommand\gs\geqslant
\newcommand\ls\leqslant
\renewcommand\geq\geqslant
\renewcommand\leq\leqslant

\newcommand{\hsym}{\hat{\mathfrak S}}
\newcommand{\sym}{\mathfrak S}
\newcommand{\halt}{\hat{\mathfrak A}}

\newcommand\bkt\blacktriangleleft
\newcommand\reg{{\operatorname{reg}}}
\newcommand\pbqx{$p$-bar-quotient}
\newcommand\pbq{\pbqx\xspace}
\newcommand\pbqs{\pbqx s\xspace}
\newcommand\lr[2]{\operatorname{c}(#1;#2)}
\newcommand\bak[1]{#1_\circ}
\newcommand\fd[1]{#1^\circ}
\renewcommand\a[1]{a(#1)}
\newcommand\prj[1]{\phi^{{#1}}}
\newcommand\mprj[1]{\hat{\phi}^{{#1}}}
\newcommand\lprj[1]{\tilde\phi^{{#1}}}

\newcommand\prm{P_{\um}}
\newcommand\mprm{\hat P_{\um}}
\newcommand\lprm{\tilde P_{\um}}
\newcommand\rpstp[2]{\mathscr{RP}_p^{#1,#2}}
\newcommand\pstp[2]{\calp_p^{#1,#2}}
\newcommand\stp[2]{\calp_0^{#1,#2}}
\newcommand\pfstp[2]{\calp_{p'}^{#1,#2}}
\newcommand\rp[1]{\mathscr{RP}_#1}

\setitemize[1]{label={$\diamond$}}

\newenvironment{pf}{\begin{proof}}{\end{proof}}

\newcommand\vn\varnothing
\newcommand\spe[1]{\mathcal{S}^{#1}}
\newcommand\ypm[1]{\mathcal{M}^{#1}}

\newcommand\ki[2]{K^{-1}_{#1#2}(-1)}

\begin{document}

\title[Decomposition numbers for RoCK blocks of double covers]{{\bf Decomposition numbers for abelian defect RoCK blocks of double covers of symmetric groups}}

\author{\sc Matthew Fayers}
\address
{Queen Mary University of London, Mile End Road, London E1 4NS, U.K.} 
\email{m.fayers@qmul.ac.uk}

\author{\sc Alexander Kleshchev}
\address{Department of Mathematics\\ University of Oregon\\Eugene\\ OR 97403, USA}
\email{klesh@uoregon.edu}

\author{\sc Lucia Morotti}
\address
{Leibniz Universit\"{a}t Hannover\\ Institut f\"{u}r Algebra, Zahlentheorie und Diskrete Mathematik\\ 30167 Hannover\\ Germany} 
\email{morotti@math.uni-hannover.de}


\thanks{The authors would like to thank the Isaac Newton Institute for Mathematical Sciences, Cambridge, for support and hospitality during the programme `Groups, representations and applications: new perspectives' where work on this paper was undertaken. The first author was supported by EPSRC grant no and EP/W005751/1. The second author was supported by the NSF grant DMS-2101791, Charles Simonyi Endowment at the Institute for Advanced Study, and the Simons Foundation. While finishing writing the paper the third author was working at the Mathematisches Institut of the Heinrich-Heine-Universit\"at D\"usseldorf.
}

\begin{abstract}
We calculate the (super)decomposition matrix for a RoCK block of a double cover of the symmetric group with abelian defect, verifying a conjecture of the first author. To do this, we exploit a theorem of the second author and Livesey that a RoCK block $\blk\rho d$ is Morita superequivalent to a wreath superproduct of a certain quiver (super)algebra with the symmetric group $\sym_d$. We develop the representation theory of this wreath superproduct to compute its Cartan invariants. We then directly construct projective characters for $\blk\rho d$ to calculate its decomposition matrix up to a triangular adjustment, and show that this adjustment is trivial by comparing Cartan invariants.
\end{abstract}

\maketitle

\tableofcontents

\section{Introduction}

In the modular representation theory of the symmetric groups and their double covers, the central outstanding question is the \emph{decomposition number problem}: determining the composition factors of the $p$-modular reductions of ordinary irreducible representations. Even for the symmetric groups a solution to this problem seems far out of reach, but there is a remarkable family of blocks for which the problem has been solved. These are called \emph{RoCK blocks}. They are defined in a combinatorial way using the abacus, and were identified by Rouquier \cite{Ro} as being of particular importance. RoCK blocks have been pivotal in the proofs of several results, most importantly in the proof of Brou\'e's abelian defect group conjecture for symmetric groups \cite{CR}. This hinges on the proof by Chuang and Kessar \cite{CK} that a RoCK block of defect $d<p$ is Morita equivalent to the principal block of the wreath product $\sym_p\wr\sym_d$. A consequence of this is the formula due to Chuang--Tan \cite{CT1} for the decomposition numbers for RoCK blocks. The same formula appears in a computation of certain canonical basis coefficients, due independently to Leclerc--Miyachi and Chuang--Tan \cite{ctq,LM}.

In recent years, the representation theory of double covers of symmetric groups (or equivalently, the study of projective representations of symmetric groups) has been studied extensively. Let $p=2\ell+1$ be an odd prime (see \cite{mfspinalt2} for corresponding results in characteristic $2$), and $\F$ an algebraically closed field of characteristic $p$. Let $\hsym_n$ denote one of the proper double covers of the symmetric group $\sym_n$, for $n\gs4$, and let $z\in\hsym_n$ denote the central element of order $2$. An irreducible $\F\hsym_n$-module $M$ is a \emph{spin module} if $z$ acts as $-1$ on $M$, and a block of $\F\hsym_n$ is a spin block if it contains spin modules. In fact, for studying spin modules it is more natural to consider $\F\hsym_n$ as a superalgebra (i.e.\ a $\Z/2\Z$-graded algebra), and study spin supermodules and spin superblocks. The modular spin representation theory of $\hsym_n$ has been developed by Brundan and the second author in \cite{BKHeCl,bk1} (using two different approaches which were later unified by the second author and Shchigolev \cite{ks}). 
The combinatorial part of this theory revolves around the combinatorics of $p$-strict partitions. 

The definition of spin RoCK blocks for $\hsym_n$ was given by the second author and Livesey \cite{kl}, who proved an analogue of Chuang and Kessar's Morita equivalence result, and used this to show that Brou\'e's conjecture holds for spin RoCK blocks. Our purpose in this paper is to give a formula for the (super)decomposition numbers for spin RoCK blocks of abelian defect; in particular, we prove a formula conjectured by the first author in \cite{mattconj} based on calculations of canonical basis coefficients.

To state our main theorem we briefly introduce some notation. For a strict partition $\la$, we let $\Sp\la$ denote a $p$-modular reduction of the irreducible spin supermodule for $\bbc\hsym_n$ labelled by $\la$, and for a restricted $p$-strict partition $\mu$, we let $\Dm\mu$ denote the irreducible spin supermodule for $\bbf\hsym_n$ labelled by $\mu$; see Section~ \ref{SRSNHat} for details on these.

If $\la$ is any partition, we write $h(\la)$ for the number of positive parts of $\la$, and $\a\la=0$ or $1$ as $\la$ has an even or odd number of positive even parts. Finally, $\lr{\al}{\si,\tau}$ denotes the Littlewood--Richardson coefficient corresponding to partitions $\al,\si,\tau$, and $K^{-1}_{\tau\si}(q)$ the inverse Kostka polynomial corresponding to $\si,\tau$; see \S\ref{SSLR} and \S\ref{SSKostkaPol} for details on these. 

Rouquier $p$-bar-cores are discussed in Section \ref{SRouquier} -- these correspond to spin RoCK blocks of double covers of symmetric groups. Now our main theorem can be stated as follows.

\begin{mainthm}
Suppose $p=2\ell+1$ is an odd prime and $1\ls d<p$, and that $\rho$ is a $d$-Rouquier $p$-bar-core. Suppose $\la$ is a strict partition and $\mu$ a restricted $p$-strict partition, both with $p$-bar-core $\rho$ and $p$-bar-weight $d$. Let $(\la^{(0)},\dots,\la^{(\ell)})$ and $(\mu^{(0)},\dots,\mu^{(\ell)})$ be the \pbqs of $\la,\mu$. Then the decomposition number $[\Sp\la:\Dm\mu]$ equals
\[
2^{\lfloor\frac12(h(\la^{(0)})+\a\la)\rfloor}\sum\ki{\la^{(0)}}{\si^{(0)}}\prod_{i=1}^\ell\lr{\la^{(i)}}{\si^{(i)},\tau^{(i)}}\lr{\mu^{(i-1)}}{\si^{(i-1)},{\tau^{(i)}}'}
\]
where the sum is over all partitions $\si^{(0)},\dots,\si^{(\ell-1)},\tau^{(1)},\dots,\tau^{(\ell)}$, and we read $\si^{(\ell)}$ as $\vn$.
\end{mainthm}

We note that the assumption $d\gs 1$ made in the theorem is harmless -- it simply means that we are dealing with blocks of non-trivial defect; on the other hand, the assumption $d<p$ is equivalent to the assumption that the blocks we are dealing with have abelian defect groups. 

The proof of our main theorem involves two parts. 

First, we use the Morita equivalence result of Kleshchev--Livesey which shows that a RoCK block $\blk\rho d$ with $p$-bar-weight $d<p$ is Morita superequivalent to a wreath superproduct $\W_d=\Zig_\ell\wr\sym_d$, where $\Zig_\ell$ is an explicitly defined quiver superalgebra. In \cref{wreathsec} we develop superalgebra analogues of results of Chuang and Tan describing the representation theory of wreath products. In particular, by explicitly constructing indecomposable projective supermodules we are able to determine the (super)Cartan matrix of $\W_d$ when $d<p$, and hence of $\blk\rho d$ (but without any information on the labels of rows and columns). 

For the second part of the proof (in \cref{projcharsec}) we explicitly consider projective characters for $\blk\rho d$. The results of Leclerc--Thibon \cite{lt} comparing decomposition numbers with canonical basis coefficients, together with the first author's formula for canonical basis coefficients corresponding to spin RoCK blocks, show that our main theorem is true `up to column operations', i.e.\ that the decomposition matrix of $\blk\rho d$ is obtained from the matrix claimed in our main theorem by post-multiplying by a square matrix $A$. By explicitly constructing projective characters by induction and comparing with known general results on decomposition numbers, we are able to show that $A$ is triangular with non-negative integer entries. By then calculating the Cartan matrix entries predicted by our main theorem and showing that they agree with those of $\W_d$ when $d<p$, we can deduce that $A$ is the identity matrix, which gives us our main theorem.

\section{Combinatorial preliminaries}
We denote $\N:=\Z_{\geq 1}$ and $\N_0:=\Z_{\geq 0}$. 
Throughout the paper, we work over an algebraically closed field $\F$ of characteristic $p>2$. We write
\begin{itemize}
\item $\ell:=(p-1)/2$,
\item $I:=\{0,1,\dots,\ell\}$,
\item $J:=\{0,1,\dots,\ell-1\}$.
\end{itemize}
For $n\in\N_0$, we write $I^n$ for the set of words $i_1\dots i_n$ with $i_1,\dots,i_n\in I$.

\subsection{Compositions and partitions}
\label{SSPar}

A \emph{composition} is an infinite sequence $\la=(\la_1,\la_2,\dots)$ of non-negative integers which are eventually zero. Any composition $\la$ has finite sum $|\la|$, and we say that $\la$ is a composition of $|\la|$. We write $\compn$ for the set of all compositions, and for each $d\in\N_0$ we write $\compn(d)$ for the set of all compositions of $d$. When writing compositions, we may collect consecutive equal parts together with a superscript, and omit an infinite tail of $0$s. We write $\vn$ for the composition $(0,0,\dots)$.

A \emph{partition} is a composition whose parts are weakly decreasing. We write $\Par$ for the set of all partitions, and $\Par(d)$ for the set of partitions of $d$.

A partition is \emph{strict} if it has no repeated positive parts. We write $\sPar(d)$ for the set of all strict partitions of $d$. Say that a strict partition $\la$ is \emph{even} if $\la$ has an even number of positive even parts, and \emph{odd} otherwise. Now write
\begin{equation}
\label{EA}
\a\la:=
\left\{
\begin{array}{ll}
0 &\hbox{if $\la$ is even,}\\
1 &\hbox{if $\la$ is odd.}
\end{array}
\right.
\end{equation}

For a set $S$, let $\Par^S(d)$ denote the set of all \emph{$S$-multipartitions} of $d$. So the elements of $\Par^S(d)$ are tuples $\bla=(\la^{(s)})_{s\in S}$ of partitions satisfying $\sum_{s\in S}|\la^{(s)}|=d$. 
In the special case $S=I$, we write the elements of 
$\Par^I(d)$ as tuples $\bla=(\la^{(0)},\dots,\la^{(\ell)})$, and similarly for $\Par^J(d)$. We refer to $\la^{(i)}$ as the $i$th component of $\bla$. 
We identify $\Par^J(d)$ with the subset of $\Par^I(d)$ consisting of those  $\bla\in \Par^I(d)$ with $\la^{(\ell)}=\vn$. 

We use the following binary operations on partitions: if $\la,\mu\in\Par$, then we write $\la+\mu$ for the partition $(\la_1+\mu_1,\la_2+\mu_2,\dots)$, and $\la\sqcup\mu$ for the partition obtained by combining the parts of $\la$ and $\mu$ and putting them in weakly decreasing order. For example, $(3,1)+(4,1^2)=(7,2,1)$, while $(3,1)\sqcup(4,1^2)=(4,3,1^3)$.

The \emph{Young diagram} of a partition $\la$ is the set 
$
\lset{(r,c)\in\N^2}{c\ls\la_r},
$ 
whose elements are called the \emph{nodes} of $\la$. We draw the Young diagram as an array of boxes using the English convention, in which $r$ increases down the page and $c$ increases from left to right. We often identify partitions with their Young diagrams; for example, we may write $\la\subseteq\mu$ to mean that $\la_r\ls\mu_r$ for all $r$.

If $\la$ is a partition, the \emph{conjugate} partition $\la'$ is obtained by reflecting the Young diagram of $\la$ on the main diagonal.

The \emph{dominance order} is a partial order $\dom$ defined on $\calp$. We set $\la\dom\mu$ (and say that $\la$ \emph{dominates} $\mu$) if $|\la|=|\mu|$ and $\la_1+\dots+\la_r\gs\mu_1+\dots+\mu_r$ for all $r\gs1$. This can be interpreted in terms of Young diagrams in the following way: $\la\dom\mu$ \iff the Young diagram of $\mu$ can be obtained from the Young diagram of $\la$ by moving some nodes further to the left, see \cite[1.4.10]{JK}. 
By \cite[1.4.11]{JK}, the dominance order is reversed by conjugation: $\la\dom\mu$ \iff $\mu'\dom\la'$.

\medskip

Now we introduce the prime $p$ into the combinatorics. Say that a partition is \emph{$p$-strict} if its repeated parts are all divisible by $p$. A $p$-strict partition $\la$ is \emph{restricted} if for all $r$ either $\la_r-\la_{r+1}<p$ or $\la_r-\la_{r+1}=p$ and $p\nmid\la_r$. We write $\Par_p(n)$ for the set of $p$-strict partitions of $n$, and $\rPar(n)$ for the set of restricted $p$-strict partitions of $n$.

We also introduce some new terminology: say that a $p$-strict partition $\la$ is a {\em \emph{\pfr}-partition} (or simply that $\la$ is $p'$) if it has no positive parts divisible by $p$.

Suppose $\la$ is a $p$-strict partition. \emph{Removing a $p$-bar} from $\la$ means either:
\begin{itemize}
\item
replacing a part $\la_r\gs p$ with $\la_r-p$, and rearranging the parts into decreasing order, or
\item
deleting two parts summing to $p$.
\end{itemize}
In the first case we assume that either $p\mid\la_r$ or $\la_r-p$ is not a part of $\la$, so that the resulting partition is $p$-strict.

The \emph{$p$-bar-core} of $\la$ is the partition obtained by repeatedly removing $p$-bars until it is not possible to remove any more  -- this is well defined thanks to \cite[Theorem~1]{my}. The {\em $p$-bar-weight} of $\la$ is the number of $p$-bars removed to reach its $p$-bar-core.

If $\rho$ is a $p$-bar-core and $d\gs1$, we write:
\begin{itemize}
\item
$\pstp\rho d$ for the set of $p$-strict partitions with $p$-bar-core $\rho$ and $p$-bar-weight $d$;
\item
$\rpstp\rho d$ for the set of restricted partitions in $\pstp\rho d$;
\item
$\stp\rho d$ for the set of strict partitions in $\pstp\rho d$;
\item
$\pfstp\rho d$ for the set of \pfr-partitions in $\pstp\rho d$.
\end{itemize}
Note that $\pfstp\rho d\subseteq \stp\rho d$.

Now we look at individual nodes. The \emph{residue} of a node in column $c$ is the smaller of the residues of $c-1$ and $-c$ modulo $p$. So the residues of nodes follow the repeating pattern
\[
0,1,\dots,\ell-1,\ell,\ell-1,\dots,1,0,0,1,\dots,\ell-1,\ell,\ell-1,\dots,1,0,\dots
\]
from left to right in every row of a Young diagram.  Note that the residue of a node is always interpreted as an element of $I$. For $i\in I$,  
an \emph{$i$-node} means a node of residue $i$.

\subsection{Littlewood--Richardson coefficients, Specht modules and permutation modules}
\label{SSLR}

For partitions $\la,\mu^1,\dots,\mu^r$ we denote by $\lr\la{\mu^1,\dots,\mu^r}$ the corresponding Littlewood--Richardson coefficient, which is zero unless $|\la|=|\mu^1|+\dots+|\mu^r|$. 
In fact, $\lr\la{\mu^1,\dots,\mu^r}$ does not depend on the order of the partitions $\mu^1,\dots,\mu^r$ and depends only on the multiset $\{\mu^1,\dots,\mu^r\}$. So we will also use the notation $\lr\la M$ for any multiset $M$ of partitions. If $M=\{\mu^1,\dots,\mu^r\}$ and $N=\{\nu^1,\dots,\nu^s\}$ are two multisets of partitions, we can also consider 
\[
\lr\la{M\sqcup N}:=\lr\la{\mu^1,\dots,\mu^r,\nu^1,\dots,\nu^s}.
\]

Below we will use various standard results on Littlewood--Richardson coefficients which can be found for example in \cite[I.9]{macdbook} or \cite[Section 5]{fulton}.


We will often use calculations involving representations of the symmetric group in characteristic zero. For any group $G$, let $\triv G$ denote the trivial $G$-module. For the group algebra $\C\sym_d$, the irreducible modules are the \emph{Specht modules} $\spe\la$, for $\la\in\calp(n)$. In particular, $\spe{(n)}$ is the trivial $\sym_n$-module, and $\spe{(1^n)}$ is the sign module, which we also denote $\sgn$. It is well-known that 
\begin{equation}\label{ESpechtSign}
\spe\la\otimes\sgn\cong\spe{\la'}
\end{equation}
 for all $\la$, see \cite[6.7]{JamesBook}. Given a $\C\sym_n$-module $M$ and any partition $\la$, we write $[M:\spe\la]$ for the multiplicity of $\spe\la$ as a composition factor of $M$ if $|\la|=n$, and $0$ otherwise.

We often induce and restrict modules between $\sym_n$ and its Young subgroups. If $\al=(\al_1,\dots,\al_r)\in\compn(n)$, then the Young subgroup $\sym_\al$ is the naturally embedded subgroup $\sym_{\al_1}\times\dots\times\sym_{\al_r}$ of $\sym_n$. Now given modules $M_1,\dots,M_r$ for $\sym_{\al_1},\dots,\sym_{\al_r}$ respectively, we obtain a module $M_1\boxtimes\dots\boxtimes M_r$ for $\sym_\al$ 
and the induced module 
\[
M_1\tens\cdots\tens M_r:=\Ind^{\sym_n}_{\sym_\al}M_1\boxtimes\cdots\boxtimes M_r.
\] 
For example, if $\la\in\compn(n)$, then $\spe{(\la_1)}\tens\spe{(\la_2)}\tens\cdots$ is the permutation module $\ypm\la$ defined in \cite[4.1]{JamesBook}, nowadays called the \emph{Young permutation module}. In general, given partitions $\al^1,\dots,\al^r$ and $\la$, the multiplicity $[\spe{\al^1}\tens\dots\tens\spe{\al^r}:\spe\la]$ is the Littlewood--Richardson coefficient $\lr\la{\al^1,\dots,\al^r}$. By Frobenius reciprocity, this can also be written 
as $[\Res_{\sym_{(|\al^1|,\dots,|\al^r|)}}\spe\la:\spe{\al^1}\boxtimes\dots\boxtimes\spe{\al^r}]$. 

Later we will need the following results.

\begin{lemma}\label{msgnm}
Suppose $\al\in\calp$ and $\be,\ga\in\compn$. Then
\[
\lr\al{(1^{\be_1}),(1^{\be_2}),\dots,(\ga_1),(\ga_2),\dots}=\left[(\ypm\be\otimes\sgn)\tens\ypm\ga:\spe\al\right].
\]
\end{lemma}

\begin{pf}
The left-hand side equals
\begin{align*}
&\left[(\spe{(1^{\be_1})}\tens\spe{(1^{\be_2})}\tens\dotsb)\tens(\spe{(\ga_1)}\tens\spe{(\ga_2)}\tens\dotsb):\spe\al\right]
\\
=&\left[\big((\spe{(\be_1)}\otimes\sgn)\tens(\spe{(\be_2)}\otimes\sgn)\tens\dotsb\big)\tens(\spe{(\ga_1)}\tens\spe{(\ga_2)}\tens\dotsb):\spe\al\right]
\\
=&\left[\big((\spe{(\be_1)}\tens\spe{(\be_2)}\tens\dotsb)\otimes\sgn\big)\tens(\spe{(\ga_1)}\tens\spe{(\ga_2)}\tens\dotsb):\spe\al\right]
\\
=&\left[(\ypm\be\otimes\sgn)\tens\ypm\ga:\spe\al\right].\qedhere
\end{align*}
\end{pf}

\vspace{1mm}
\begin{lemma}\label{mackey}
Suppose $\la\in\compn$ and $\tau,\si\in\calp$. Then
\[
\sum_{\mu\in\calp}[\ypm\la:\spe\mu][\spe\tau\tens\spe\si:\spe\mu]=\sum_{\substack{\be,\ga\in\compn\\\be+\ga=\la}}[\ypm\be:\spe\tau][\ypm\ga:\spe\si].
\]
\end{lemma}

\begin{pf}
Let $n=|\la|$. We may assume that $|\tau|+|\si|=n$ as well (since otherwise both sides are obviously zero) and we may restrict the range of summation on the left-hand side to $\mu\in\calp(n)$. The definition of $\ypm\la$ gives $[\ypm\la:\spe\mu]=[\spe{(\la_1)}\tens\spe{(\la_2)}\tens\dotsb:\spe\mu]$. On the other hand, if we define $K$ to be the Young subgroup $\sym_{(|\tau|,|\si|)}$, then Frobenius reciprocity gives $[\spe\tau\tens\spe\si:\spe\mu]=[\Res_K\spe\mu:\spe\tau\boxtimes\spe\si]$. Since the irreducible $\C\sym_n$-modules are precisely the modules $\spe\mu$ for $\mu\in\calp(n)$, the left-hand side gives the multiplicity
\[
[\Res_K\Ind^{\sym_n}\triv{\sym_\la}:\spe\tau\boxtimes\spe\si].
\]
By Mackey's Theorem, this is the same as
\[
\sum_H[\Ind^K\triv H:\spe\tau\boxtimes\spe\si],
\]
summing over $K$-conjugacy class representatives of subgroups $H\ls K$ of the form $(\sym_\la)^x\cap K$ for $x\in\sym_n$. We can take these representatives to be the groups $\sym_\be\times\sym_\ga$ as $\be,\ga$ range over compositions satisfying $|\be|=|\tau|$, $|\ga|=|\si|$ and $\be_r+\ga_r=\la_r$ for each $r$. Now the definition of the modules $\ypm\be$ and $\ypm\ga$ gives the result.
\end{pf}

We have the following `Mackey formula' for Littlewood--Richardson coefficients.

\begin{lemma}\label{smackey}
Suppose $\al,\be,\ga,\de\in\calp$. Then
\[
\sum_{\la\in\calp}\lr\la{\al,\be}\lr\la{\ga,\de}=\sum_{\phi,\chi,\psi,\om\in\calp}\lr\al{\phi,\chi}\lr\be{\psi,\om}\lr\ga{\phi,\psi}\lr\de{\chi,\om}.
\]
\end{lemma}

\begin{pf}
The special case where $\al=(r)$ is proved by Chuang and Tan \cite[Lemma 2.2(3)]{ctq}, but their proof works in the general case.
\end{pf}

\subsection{Kostka polynomials}
\label{SSKostkaPol}
Given $\la,\si\in\calp$, we write $K^{-1}_{\la\si}(t)$ for the \emph{inverse Kostka polynomial} indexed by $\la,\si$; this polynomial arises in the theory of symmetric functions: it is the coefficient of the Schur function $s_\si$ when the Hall--Littlewood symmetric function $P_\la$ is expressed in terms of Schur functions. We refer to \cite[III.6]{macdbook} for more information on Kostka polynomials, but we note in particular that $K^{-1}_{\la\si}(t)$ is zero unless $\la\dom\si$ and that $K^{-1}_{\la\la}(t)=1$; see \cite[Lemma 3.4]{mattconj}.

Of special importance for us will be the evaluation of $K^{-1}_{\la\si}(t)$ at $t=-1$. So $K^{-1}_{\la\si}(-1)$ is the coefficient of $s_\si$ in the Schur P-function $P_\la$.

We note two lemmas that we will need later.

\begin{lemma}\label{kostlem}
Suppose $\si\in\calp(n)$. Then $\ki\la\si\in\N_0$ for all $\la\in\calp_0(n)$, and there is at least one $\la\in\calp_0(n)$ for which $\ki\la\si>0$.
\end{lemma}

\begin{pf}
Stembridge \cite[Theorem 9.3(b)]{stem} shows that $\ki\la\si$ equals the number of tableaux of a certain type, which means in particular that $\ki\la\si\in\N_0$. Stembridge's formula shows in particular that $\ki\la\si>0$ when $\la$ is the strict partition whose parts are the diagonal hook lengths of $\si$.
\end{pf}

\begin{lemma}\label{klem}
Suppose $\xi,\pi\in\calp$. Then
\[
\sum_{\la\in\calp_0}2^{h(\la)}\ki\la\xi\ki\la\pi=\sum_{\be,\ga\in\calp}\lr\xi{\be,\ga'}\lr\pi{\be,\ga}.
\]
\end{lemma}

\begin{pf}
We consider symmetric functions in an infinite set of variables $X$. Let $s_\pi$ denote the Schur function indexed by $\pi\in\calp$. Since the Schur functions are linearly independent, it suffices to show the following equality of symmetric functions, for each $\xi$:
\[
\sum_{\substack{\la\in\calp_0\\\pi\in\calp}}2^{h(\la)}\ki\la\xi\ki\la\pi s_\pi=\sum_{\be,\ga,\pi\in\calp}\lr\xi{\be,\ga'}\lr\pi{\be,\ga}s_\pi.
\]
Working with an indeterminate $t$, consider the symmetric function
\[
\sum_{\la,\pi\in\calp}b_\la(t)K^{-1}_{\la\xi}(t)K^{-1}_{\la\pi}(t)s_\pi,
\]
where $b_\la(t)$ is the polynomial defined in \cite[(2.12) on p.210]{macdbook}. According to the transition matrix in \cite[p.241]{macdbook}, this coincides with the `dual Schur function' $S_\xi(t)$. Now specialize $t$ to $-1$. It is immediate from the definition of $b_\la(t)$ that
\[
b_\la(-1)=
\begin{cases}
2^{h(\la)}&\text{if }\la\in\calp_0
\\
0&\text{otherwise},
\end{cases}
\]
so we find that
\[
S_\xi(-1)=\sum_{\substack{\la\in\calp_0\\\pi\in\calp}}2^{h(\la)}\ki\la\xi\ki\la\pi s_\pi.
\]
Let us write $S_\xi(-1)$ as $\bar S_\xi$. According to \cite[III.8, Example 7(a)]{macdbook}, $\bar S_\xi$ equals the function $s_\xi(X/{-}X)$ defined in \cite[I.5, Example 23]{macdbook}. From equation (1) in \loccit we obtain
\[
s_\xi(X/{-}X)=\sum_{\be\in\calp}s_\be s_{\xi'/\be'},
\]
where the skew Schur function $s_{\xi'/\be'}$ equals $\sum_{\ga\in\calp}\lr{\xi'}{\be',\ga}s_\ga$. In addition $s_\be s_\ga=\sum_{\pi\in\calp}\lr\pi{\be,\ga}s_\pi$ (indeed, this is the most usual definition of the Littlewood--Richardson coefficients), so that
\[
\sum_{\substack{\la\in\calp_0\\\pi\in\calp}}2^{h(\la)}\ki\la\xi\ki\la\pi s_\pi=\sum_{\be.\ga\in\calp}\lr{\xi'}{\be',\ga}\lr\pi{\be,\ga}s_\pi.
\]
Now the standard result that $\lr{\xi'}{\be',\ga}=\lr\xi{\be,\ga'}$ gives the required equality.
\end{pf}

\section{Rouquier bar-cores}\label{rouqsec}
\label{SRouquier}
\subsection{Definition and first properties}
For any $p$-strict partition $\rho$, define
\[
r_i(\rho):=\left|\lset{r\in\N}{\rho_r\equiv i\ppmod p}\right|
\]
for $i\in\{1,\dots,p-1\}$. If $\rho$ is a $p$-bar-core, then $\rho$ is determined by the integers $r_i(\rho)$. Following \cite{kl}, given $d\gs1$, we say that a $p$-bar-core $\rho$ is \emph{$d$-Rouquier} if
\begin{itemize}
\item
$r_1(\rho)\gs d$, and
\item
$r_i(\rho)\gs r_{i-1}(\rho)+d-1$ for $2\ls i\ls\ell$.
\end{itemize}
(This automatically implies that $r_i(\rho)=0$ for $i>\ell$, since a $p$-bar-core cannot have two parts whose sum is divisible by $p$.)

Assume that $\rho$ is a $d$-Rouquier $p$-bar-core, and $\la\in\pstp\rho d$. We want to define the \emph{\pbq} of $\la$. First note that $r_i(\la)=r_i(\rho)$ for each $1\ls i\ls\ell$, since $r_i(\rho)\geq d$, cf.\ \cite[Lemma 4.1.1.(i)]{kl}. Now define $\la^{(0)}$ to be the partition obtained by taking all the parts of $\la$ divisible by $p$ and dividing them by $p$. For $1\ls i\ls\ell$, let $r:=r_i(\la)$, let $\la_{k_1}>\dots>\la_{k_r}$ be the parts of $\la$ congruent to $i$ modulo $p$, and define the partition
\[
\la^{(i)}:=\left(\frac{\la_{k_1}-(r-1)p-i}p,\frac{\la_{k_2}-(r-2)p-i}p,\dots,\frac{\la_{k_r}-i}p\right).
\]
The \pbq of $\la$ is the multipartition $(\la^{(0)},\dots,\la^{(\ell)})\in\Par^I(d)$.

\begin{eg}
Suppose $p=5$ and $\rho=(32,27,22,17,16,12,11,7,6,2,1)$. Then $\rho$ is $4$-Rouquier, with $(r_1(\rho),r_2(\rho),r_3(\rho),r_4(\rho))=(4,7,0,0)$. The partition $\la=(37,32,22,17,16,12,11,10,7,6,2,1)$ lies in $\stp\rho4$, and has $5$-bar-quotient $(\la^{(0)},\la^{(1)},\la^{(2)})=((2),\vn,(1^2))$.
\end{eg}

\begin{lemma}\label{pbqlem}
Suppose $\rho$ is a $d$-Rouquier $p$-bar-core, and $\la\in\pstp\rho d$, with \pbq $(\la^{(0)},\dots,\la^{(\ell)})$. Then:
\begin{enumerate}
\item $\la$ is strict \iff $\la^{(0)}$ is strict;
\item $\la$ is \pfr\, \iff $\la^{(0)}=\vn$;
\item $\la$ is restricted \iff $\la^{(\ell)}=\vn$.
\end{enumerate}
\end{lemma}

\begin{proof}
The first two statements follow directly from the definition, so we need only prove the third. Note that by the given properties of the integers $r_i(\rho)$, the $d$ largest parts of $\rho$ are all congruent to $\ell$ modulo $p$, and $\rho_k<\rho_1-(d-1)p$ for any $k$ with $\la_k\not\equiv\ell\ppmod p$.

We obtain $\la$ from $\rho$ by adding $d$ $p$-bars. So any part $\la_k$ for which $\la_k\not\equiv\ell\ppmod p$ satisfies $\la_k<\rho_1+p$. If $\la^{(\ell)}=\vn$ then $\la_1<\rho_1+p$, while $\la$ contains all the integers $\ell,\ell+p,\dots,\rho_1$, so $\la$ is restricted.

If instead $\la^{(\ell)}_1\neq\vn$, choose $a$ such that $\la^{(\ell)}_a>\la^{(\ell)}_{a+1}$. Then $|\la^{(\ell)}|\gs a$, so that $|\la^{(i)}|\ls d-a$ for any $i\neq\ell$. This means that any part $\la_k\not\equiv\ell\ppmod p$ satisfies $\la_k<\rho_1-(a-1)p=\rho_a$. 
So $\la$ contains the part $\la_a=\rho_a+\la^{(\ell)}_ap$, but does not contain any parts $\la_k$ with $\rho_a+(\la^{(\ell)}_a-1)p\ls\la_k<\rho_a+\la^{(\ell)}_ap$, so is not restricted.
\end{proof}

Clearly $\la\in\pstp\rho d$ is determined by $\rho$ and the \pbq $(\la^{(0)},\dots,\la^{(\ell)})$; conversely, given a multipartition $(\la^{(0)},\dots,\la^{(\ell)})\in\Par^I(d)$, there is a partition $\la\in\pstp\rho d$ with \pbq $(\la^{(0)},\dots,\la^{(\ell)})$. In view of this and \cref{pbqlem}, we see that
\begin{equation}\label{EPartId}
|\pstp\rho d|=|\Par^I(d)|\qquad\text{and}\qquad
|\pfstp\rho d|=|\rpstp\rho d|=|\Par^J(d)|.
\end{equation}
\subsection{Rouquier bar-cores and dominance order}
For our calculations in RoCK blocks, it will be helpful to introduce a partial order on $\calp^I(d)$: given two multipartitions $(\la^{(0)},\dots,\la^{(\ell)})$ and $(\mu^{(0)},\dots,\mu^{(\ell)})$ in $\calp^I(d)$, we write $(\la^{(0)},\dots,\la^{(\ell)})\sucq(\mu^{(0)},\dots,\mu^{(\ell)})$ if
\[
|\la^{(0)}|+\dots+|\la^{(k-1)}|+{\la^{(k)}}'_1+\dots+{\la^{(k)}}'_c\gs|\mu^{(0)}|+\dots+|\mu^{(k-1)}|+{\mu^{(k)}}'_1+\dots+{\mu^{(k)}}'_c
\]
for all $0\ls k\ls\ell$ and $c\gs1$. This order can be visualized by drawing the Young diagrams of $\la^{(0)},\dots,\la^{(\ell)}$ in a row from left to right; then $(\la^{(0)},\dots,\la^{(\ell)})\sucq(\mu^{(0)},\dots,\mu^{(\ell)})$ \iff $(\la^{(0)},\dots,\la^{(\ell)})$ can obtained from $(\mu^{(0)},\dots,\mu^{(\ell)})$ by moving nodes further to the left.

\begin{lemma}\label{succdom}
Let $\rho$ be a $d$-Rouquier $p$-bar-core. Suppose that the partitions $\la$ and $\mu$ in $\pstp\rho d$ have $p$-bar-quotients $(\la^{(0)},\dots,\la^{(\ell)})$ and $(\mu^{(0)},\dots,\mu^{(\ell)})$, respectively. Then $(\la^{(0)},\dots,\la^{(\ell)})\sucq(\mu^{(0)},\dots,\mu^{(\ell)})$ \iff $\la\domby\mu$.
\end{lemma}

\begin{pf}
For $i=0,\dots,\ell$ let $r_i$ be the largest part of $\rho$ congruent to $i$ modulo $p$. We also denote $\bla:=(\la^{(0)},\dots,\la^{(\ell)})$, $\bmu:=(\mu^{(0)},\dots,\mu^{(\ell)})$. 

We construct $\la$ from $\rho$ by successively adding $p$-bars. Correspondingly, the \pbq $\bla$ is obtained from $(\vn,\dots,\vn)$ by adding nodes; adding the node $(r,c)$ to $\la^{(i)}$ corresponds to adding nodes to $\la$ in columns
\begin{equation}\label{addcols}
\begin{cases}
r_i+(c-r)p+1,r_i+(c-r)p+2,\dots,r_i+(c-r+1)p&\text{if $i>0$}
\\
(c-1)p+1,(c-1)p+2,\ldots,cp&\text{if $i=0$.}
\end{cases}
\end{equation}

We now prove the `only-if' part of the lemma.  
It is easy to see that if $\bla\sucq\bmu$ then we can reach $\bla$ from $\bmu$ by a sequence of moves in which a single node is moved further to the left; so it suffices to consider a single such move, and show that this move corresponds to moving nodes to the left in $\mu$. So suppose $\bla$ is obtained from $\bmu$ by replacing the node $(s,c)$ in the $j$th component with the node $(r,b)$ in the $i$th component, where $i\ls j$.

If $0<i=j$, then $b-r<c-s$, so by (\ref{addcols}) $\la$ is obtained from $\mu$ by moving $p$ nodes further to the left. If $0=i=j$, then a similar argument applies using the inequality $b<c$. If $0<i<j$, then $b-r\ls c-s+d-1$, because $\mu^{(i)}_1+{\mu^{(j)}}'_1\ls d$. Now (\ref{addcols}) and the fact that $r_i<r_j+(d-1)p$ means that $\la$ is obtained from $\mu$ by moving $p$ nodes further to the left. If $0=i<j$, then we use a similar argument via the inequality $b\ls c-s+d$.

In any case, we obtain $\la\domsby\mu$, as required.

We now prove the `if' part of the lemma. 
Assume $\bla\not\sucq\bmu$; then we must show that $\la\ndomby\mu$.

{\em Case 1:}  there is $k\in I$ such that $|\la^{(0)}|+\dots+|\la^{(k)}|<|\mu^{(0)}|+\dots+|\mu^{(k)}|$. 

Note that in this case $k<\ell$. 
Let $a=|\la^{(0)}|+\dots+|\la^{(k)}|$ and $b=|\mu^{(0)}|+\dots+|\mu^{(k)}|$. Now let $\nu,\xi\in\pstp\rho d$ be given by
\[
\nu^{(i)}=
\begin{cases}
(1^a)&\text{if $i=0$}
\\
(1^{d-a})&\text{if $i=k+1$}
\\
\vn&\text{otherwise},
\end{cases}
\qquad
\xi^{(i)}=
\begin{cases}
(b)&\text{if $i=k$}
\\
(d-b)&\text{if $i=\ell$}
\\
\vn&\text{otherwise}.
\end{cases}
\]
Then $(\nu^{(0)},\dots,\nu^{(\ell)})\sucq\bla$ and $\bmu\sucq(\xi^{(0)},\dots,\xi^{(\ell)})$. So (from the `$\Rightarrow$' part of the \lcnamecref{succdom}) in order to show that $\la\ndomby\mu$ it suffices to show that $\nu\ndomby\xi$. To do this, we let $r$ be such that $\rho_r=r_{k+1}-(d-a-1)p$, and compare $\nu_1+\dots+\nu_r$ with $\xi_1+\dots+\xi_r$. We obtain
\begin{align*}
\nu_1+\dots+\nu_r&=\rho_1+\dots+\rho_r+(d-a)p,
\\
\xi_1+\dots+\xi_r&=\rho_1+\dots+\rho_r+(d-b)p+\max\{r_k-r_{k+1}+(d+b-a-1)p,0\},
\end{align*}
and now the assumptions $r_{k+1}>r_k+(d-1)p$ and $a<b$ give $\nu_1+\dots+\nu_r>\xi_1+\dots+\xi_r$, so that $\nu\ndomby\xi$.

{\em Case 2:}   $|\la^{(0)}|+\dots+|\la^{(k)}|\gs|\mu^{(0)}|+\dots+|\mu^{(k)}|$ for every $k\in I$. 

The assumption that $\bla\not\sucq\bmu$ means that we can find $k\in I$ and $c\gs1$ for which
\begin{equation}\label{notsucq}
\sum_{i=0}^{k-1}|\la^{(i)}|\,+\,{\la^{(k)}}'_1+\dots+{\la^{(k)}}'_c<\sum_{i=0}^{k-1}|\mu^{(i)}|\,+\,{\mu^{(k)}}'_1+\dots+{\mu^{(k)}}'_c.
\end{equation}

First we assume that $k>0$. 

Let $r={\la^{(k)}}'_c$, and $s=r_k+(c-r+1)p$; then we claim that $\la'_1+\dots+\la'_s<\mu'_1+\dots+\mu'_s$, so that $\la\ndomby\mu$.

We calculate $\la'_1+\dots+\la'_s-(\rho'_1+\dots+\rho'_s)$ using (\ref{addcols}). For $0\ls i<k$ each node of $\la^{(i)}$ contributes $p$ to this sum. In addition, each node $(t,b)$ of $\la^{(k)}$ for which $b-t\ls c-r$ contributes $p$ to the sum. (The nodes of $\la^{(i)}$ for $i>k$ do not contribute, because of the inequality $r_{k+1}-r_k>(d-1)p$.) Writing $T_{r,c}=\sum_{x=1}^{r-1}\min\{x,c\}$, we obtain
\begin{align*}
\sum_{i=1}^s\la'_i-\sum_{i=1}^s\rho'_i
&=\left(\sum_{i=0}^{k-1}|\la^{(i)}|+\sum_{t\gs\max\{1,r-c\}}\min\{\la^{(k)}_t,t+c-r\}\right)p
\\
&=\left(\sum_{i=0}^{k-1}|\la^{(i)}|+\sum_{d=1}^c{\la^{(k)}}'_d-T_{r,c}\right)p,
\end{align*}
with the second equality coming from the fact that ${\la^{(k)}}'_c=r$.

We calculate $\mu'_1+\dots+\mu'_s-(\rho'_1+\dots+\rho'_s)$ in the same way. The assumption that $|\mu^{(0)}|+\dots+|\mu^{(k-1)}|\ls|\la^{(0)}|+\dots+|\la^{(k-1)}|$ means that each node of $\mu^{(i)}$ for $i<k$ contributes $p$ to this sum, while the nodes of $\mu^{(i)}$ for $i>k$ do not contribute. So, as with $\la$, we obtain
\[
\sum_{i=1}^s\mu'_i-\sum_{i=1}^s\rho'_i
=\left(\sum_{i=0}^{k-1}|\mu^{(i)}|+\sum_{t\gs\max\{1,r-c\}}\min\{\mu^{(k)}_t,t+c-r\}\right)p.
\]
It follows that
\[
\sum_{t\gs\max\{1,r-c\}}\min\{\mu^{(k)}_t,t+c-r\}\geq\sum_{d=1}^c{\mu^{(k)}}'_d-T_{r,c}
\]
and then
\[
\sum_{i=1}^s\mu'_i-\sum_{i=1}^s\rho'_i\geq\left(\sum_{i=0}^{k-1}|\mu^{(i)}|+\sum_{d=1}^c{\mu^{(k)}}'_d-T_{r,c}\right)p.
\]
We obtain $\la'_1+\dots+\la'_s<\mu'_1+\dots+\mu'_s$, as required.

Now assume instead that $k=0$. Then we claim that $\la'_1+\dots+\la_{cp}'<\mu'_1+\dots+\mu_{cp}'$. As for the case above, we calculate $\la'_1+\dots+\la'_s-(\rho'_1+\dots+\rho'_s)$ using (\ref{addcols}). Each node $(t,b)$ of $\la^{(0)}$ with  $b\ls c$ contributes $p$ to this sum, and the nodes of $\la^{(i)}$ for $i\gs1$ do not contribute, because of the inequality $r_1>dp$ and the fact that $|\la^{(0)}|\gs c$. So we obtain
\[
\sum_{i=1}^{cp}\la'_i-\sum_{i=1}^{cp} \rho'_i=\sum_{i=1}^c{\la^{(0)}}'_i.
\]
The same formula with $\mu$ in place of $\la$ gives the result.
\qedhere
\end{pf}

\subsection{Rouquier bar-cores and containment of partitions}
We will need the following generalization of \cite[Lemma 4.1.2]{kl}.

\begin{propn}\label{addbars}
Suppose $\rho$ is a $d$-Rouquier $p$-bar-core. Suppose $\la\in\pstp\rho a$ and $\al\in\pstp\rho b$, where $a,b\ls d$, and let $(\la^{(0)},\dots,\la^{(\ell)})$ and $(\al^{(0)},\dots,\al^{(\ell)})$ be the \pbqs of $\la$ and $\al$. Then the following are equivalent:
\begin{enumerate}
\item\label{lainal}
$\la\subseteq\al$;
\item\label{lakinalk}
$\la^{(j)}\subseteq\al^{(j)}$ for all $j\in I$;
\item\label{hookadd}
$\al$ can be obtained from $\la$ by successively adding $p$-bars.
\end{enumerate}
\end{propn}

\begin{pf}
It is trivial that (\ref{hookadd})$\Rightarrow$(\ref{lainal}). It is also very easy to see that (\ref{lakinalk})$\Rightarrow$(\ref{hookadd}): adding a node to a component of the \pbq corresponds to increasing one of the parts of the partition by $p$, which is a way of adding a $p$-bar.

So it remains to show that (\ref{lainal})$\Rightarrow$(\ref{lakinalk}). (We remark that the case where $b-a=1$ is proved in \cite[Lemma 4.1.2]{kl}.)

We use induction on $a$. The case $a=0$ is trivial, so we assume $a>0$, and that the result is true with $a$ replaced by any smaller value. Assume $\la\subseteq\al$. 

Suppose $\mu\in\pstp\rho{a-1}$ and that the \pbq $(\mu^{(0)},\dots,\mu^{(\ell)})$ of $\mu$ is obtained from the \pbq of $\la$ by removing a single node. Then $\mu\subset\la$ (from the fact that (\ref{lakinalk})$\Rightarrow$(\ref{hookadd})$\Rightarrow$(\ref{lainal})), so $\mu\subset\al$, and the inductive hypothesis gives $\mu^{(j)}\subseteq\al^{(j)}$ for all $j$. So the only node of $(\la^{(0)},\dots,\la^{(\ell)})$ which can fail to be a node of $(\al^{(0)},\dots,\al^{(\ell)})$ is the node removed to obtain $\mu$. In particular, if there are at least two such partitions $\mu$ (that is, if $(\la^{(0)},\dots,\la^{(\ell)})$ has at least two removable nodes), then $\la^{(j)}\subseteq\al^{(j)}$ for all $j$ as required.

So we can assume that $(\la^{(0)},\dots,\la^{(\ell)})$ has only one removable node. This means there is $k\in I$ such that $\la^{(k)}$ is a rectangular partition $(x^y)$ with $x,y\gs1$, while $\la^{(j)}=\vn$ for $j\neq k$. From the argument in the previous paragraph, we can assume that $\al^{(k)}$ contains the partition $(x^{y-1},x-1)$. If we suppose for a contradiction that $\la^{(k)}\nsubseteq\al^{(k)}$, then $\al^{(k)}$ has fewer than $\min\{x,y\}$ nodes $(r,c)$ for which $c-r=x-y$.

For each $j$ we define $r_j$ to be the largest part of $\rho$ congruent to $j$ modulo $p$. As observed in \cref{succdom}, adding the node $(r,c)$ to the $j$th component of $\la^{(j)}$ corresponds to adding nodes to $\la$ in columns
\[
r_j+(c-\hat r)p+1,\ r_j+(c-\hat r)p+2,\ \dots,\ r_j+(c-\hat r+1)p,
\]
where we write $\hat r=1$ if $j=0$, and $\hat r=r$ otherwise.

Assume first that $k\gs1$. Then the assumption $\la\subseteq\al$ and the paragraph above give
\[
\al'_{r_k+(x-y)p+1}-\rho'_{r_k+(x-y)p+1}\gs\la'_{r_k+(x-y)p+1}-\rho'_{r_k+(x-y)p+1}=\min\{x,y\}.
\]
Since $\al^{(k)}$ has fewer than $\min\{x,y\}$ nodes $(r,c)$ for which $c-r=x-y$, there must be some $j\neq k$ such that $\al^{(j)}$ has a node $(r,c)$ for which
\[
r_k+(x-y)p+1=
\begin{cases}
r_j+(c-\hat r)p+k-j+1
&\text{if }j<k,
\\
r_j+(c-r)p+p+k-j+1
&\text{if }j>k.
\end{cases}
\]
In fact this is impossible for $j>k$, since it gives
\[
(r-c-1+x-y)p+j-k=r_j-r_k\gs(d-1)p+j-k,
\]
and therefore
\[
r-c+x-y\gs d.
\]
But $|\al^{(j)}|\gs r-c+1$ and $|\al^{(k)}|\gs x+y-2$, and we obtain $|\al^{(k)}|+|\al^{(j)}|\gs d+1$, which contradicts the assumption that $\al\in\pstp\rho b$ with $b\ls d$.

So instead $j<k$. Now we obtain
\[
(c-\hat r+y-x)p+k-j=r_k-r_j\gs(d-1)p+k-j,
\]
so that
\[
c-\hat r+y-x\gs d-1.
\]
But $|\al^{(j)}|\gs c-\hat r+1$ and $|\al^{(k)}|\gs x+y-2$ and $|\al^{(j)}|+|\al^{(k)}|\ls d$, so we have equality everywhere, and in particular $|\al^{(j)}|+|\al^{(k)}|=d$.

Now we perform a similar calculation using the fact that $\al'_{r_k+(x-y+1)p}-\rho'_{r_k+(x-y+1)p}\gs\min\{x,y\}$. Now there is $j'\neq k$ such that (writing $\check r=1$ if $j'=0$ and $\check r=r$ otherwise)
\[
r_k+(x-y+1)p=
\begin{cases}
r_{j'}+(c-\check r)p+k-j'
&\text{if }j'<k,
\\
r_{j'}+(c-r)p+p+k-j'
&\text{if }j'>k.
\end{cases}
\]
Now the case $j'<k$ leads to an impossibility (in a similar way to the case $j>k$ above), so $j'$ must be greater than $k$. But now we have indices $j<k<j'$ with $|\al^{(j)}|+|\al^{(k)}|+|\al^{(j')}|\gs d+1$, which again contradicts the assumption $\al\in\pstp\rho d$. The result follows in the case $k\gs1$.

The case $k=0$ is similar but simpler. In this case
\[
\al'_{(x-1)p+1}-\rho'_{(x-1)p+1}\gs\la'_{(x-1)p+1}-\rho'_{(x-1)p+1}=y,
\]
but $\al^{(0)}$ has fewer than $y$ nodes in column $x$, so there is $j>0$ such that $\al^{(j)}$ has a node $(r,c)$ with
\[
(x-1)p+1=r_j+(c-r)p+p+1-j
\]
and therefore
\[
(r-c+x-2)p+j=r_j\gs (d-1)p+j
\]
so that
\[
r-c+x-2\gs d-1.
\]
But now the fact that $|\al^{(0)}|\gs x-1$ and $|\al^{(j)}|\gs r-c+1$ gives a contradiction. So the result follows in the case $k=0$ as well.
\end{pf}

\section{Superalgebras, supermodules and wreath superproducts}\label{wreathsec}

The representation theory of double covers of symmetric groups is best approached via superalgebras. In this section we recall the general theory and then study representations of some special wreath superproducts $\Zig_\ell\wr \sym_d$ which play a crucial role for RoCK  (super)blocks of double covers of symmetric groups, cf.\ \cref{mainkl}. Our aim is to compute the Cartan invariants for $\Zig_\ell\wr \sym_d$ in the case where $d<p$ in terms of Littlewood-Richardson coefficients, cf.\ \cref{wrcartan}. 

\subsection{Superspaces}
\label{SSSASM}

We write $\ztz=\{\0,\1\}$. If $V$ is a vector space over $\F$, a \emph{$\ztz$-grading} on $V$ is a direct sum decomposition $V=V_\0\oplus V_\1$. A vector \emph{superspace} is a vector space with a chosen $\ztz$-grading. For $\eps\in\ztz$, if $v\in V_{\eps}$ we write $|v|=\eps$ and say that $v$ is \emph{homogeneous} of parity~$\eps$. 

If $V$ and $W$ are superspaces and $\eps\in\Z/2\Z$ then a linear map $f:V\to W$ is called a \emph{homogeneous superspace homomorphism of parity $\eps$} if $f(V_\de)\subseteq W_{\de+\eps}$ for all $\de\in \Z/2\Z$. A \emph{superspace homomorphism $f:V\to W$}  means a map $f=f_\0+f_\1$, where for $\eps=\0,\1$ the map $f_\eps:V\to W$ is a homogeneous superspace homomorphism of parity~$\eps$. We will use the term `even homomorphism' to mean `homogeneous homomorphism of parity $\0$', and similarly for odd homomorphisms. 

We write $\Pi$ for the parity change functor, see e.g.\ \cite[\S12.1]{KBook}. Thus, for a superspace $V$, the superspace $\Pi V$ equals $V$ as a vector space, but with parities swapped. 
We define an odd isomorphism of superspaces
\begin{align*}
\si_V:V\longrightarrow \Pi V,
\ 
v\longmapsto (-1)^{|v|}v.
\end{align*}

If $V_1,\dots,V_d$ are superspaces then $V_1\otimes\dots\otimes V_d$ is a superspace  with $|v_1\otimes\dots\otimes v_d|=|v_1|+\dots|v_d|$. (Here and below in similar situations we assume that the elements $v_k$ are homogeneous and extend by linearity where necessary.) If $f_i:V_i\to W_i$ is a superspace  homomorphism for $i=1,\dots,d$, then 
\[
f_1\otimes \dots\otimes f_d: V_1\otimes\dots\otimes V_d\to W_1\otimes\dots\otimes W_d
\]
is a superspace homomorphism defined from
\[
(f_1\otimes \dots\otimes f_d)(v_1\otimes\dots\otimes v_d)=(-1)^{\sum_{1\leq r<s\leq d}|f_s||v_r|}f_1(v_1)\otimes\dots\otimes f_d(v_d).
\]

Let $V=V_\0\oplus V_\1$ be a superspace, and $d\in\N$. The symmetric group $\Si_d$ acts on $V^{\otimes d}$ via
\begin{equation*}\label{EWSignAction}
{}^{w}(v_1\otimes\dots\otimes v_d):=(-1)^{[w;v_1,\dots,v_d]}v_{w^{-1}(1)}\otimes\dots\otimes v_{w^{-1}(d)},
\end{equation*}
where for $w\in\Si_d$ and $v_1,\dots,v_d\in V$, we have 
\[
[w;v_1,\dots,v_d]:=\sum_{\substack{1\leq a<c\leq d\\w(a)>w(c)}}|v_a||v_c|.
\]
It is now easy to check that  
\begin{equation}\label{ESignW}
\si_V^{\otimes d}({}^w(v_1\otimes\dots\otimes v_d))=\sgn(w)\Big({}^w\big(\si_V^{\otimes d}(v_1\otimes\dots\otimes v_d)\big)\Big).
\end{equation}

\subsection{Superalgebras}

An $\F$-\emph{superalgebra} is an $\F$-algebra $A$ with a chosen $\ztz$-grading $A=A_\0\oplus A_\1$ such that $ab\in A_{|a|+|b|}$ (whenever $a,b\in A$ are both homogeneous). If $A$ and $B$ are  $\F$-superalgebras, a \emph{superalgebra homomorphism} $f:A\to B$ is an even unital algebra homomorphism. 

If $A_1,\dots,A_d$ are superalgebras then the superspace $A_1\otimes\dots\otimes A_d$ is a superalgebra with multiplication
\[
(a_1\otimes\dots\otimes a_d)(b_1\otimes\dots\otimes b_d)=(-1)^{\sum_{1\leq r<s\leq d}|a_s||b_r|}a_1b_1\otimes\dots\otimes a_db_d.
\]

\begin{Example}\label{ExAA}
We consider the quiver 
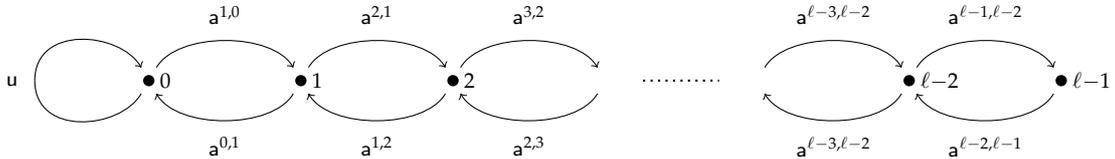
\begin{align*}
\begin{braid}\tikzset{baseline=3mm}
\coordinate (0) at (-3,0);
\coordinate (1) at (0,0);
\coordinate (2) at (4,0);
\coordinate (3) at (8,0);
\coordinate (4) at (12,0);
\coordinate (6) at (16,0);
\coordinate (L1) at (20,0);
\coordinate (L) at (24,0);
\draw [thin, black,->, shorten >= 0.2cm]   (0) to[distance=1.5cm,out=90, in=120] (1);
\draw [thin, black, shorten <= 0.2cm]   (1) to[distance=1.5cm,out=-120, in=-90] (0);
\draw [thin, black,->,shorten <= 0.2cm, shorten >= 0.2cm]   (1) to[distance=1.5cm,out=60, in=120] (2);
\draw [thin,black,->,shorten <= 0.2cm, shorten >= 0.2cm]   (2) to[distance=1.5cm,out=-120, in=-60] (1);
\draw [thin,black,->,shorten <= 0.2cm, shorten >= 0.2cm]   (2) to[distance=1.5cm,out=60, in=120] (3);
\draw [thin,black,->,shorten <= 0.2cm, shorten >= 0.2cm]   (3) to[distance=1.5cm,out=-120, in=-60] (2);
\draw [thin,black,->,shorten <= 0.2cm, shorten >= 0.2cm]   (3) to[distance=1.5cm,out=60, in=120] (4);
\draw [thin,black,->,shorten <= 0.2cm, shorten >= 0.2cm]   (4) to[distance=1.5cm,out=-120, in=-60] (3);
\draw [thin,black,->,shorten <= 0.2cm, shorten >= 0.2cm]   (6) to[distance=1.5cm,out=60, in=120] (L1);
\draw [thin,black,->,shorten <= 0.2cm, shorten >= 0.2cm]   (L1) to[distance=1.5cm,out=-120, in=-60] (6);
\draw [thin,black,->,shorten <= 0.2cm, shorten >= 0.2cm]   (L1) to[distance=1.5cm,out=60, in=120] (L);
\draw [thin,black,->,shorten <= 0.2cm, shorten >= 0.2cm]   (L) to[distance=1.5cm,out=-120, in=-60] (L1);
\blackdot(0,0);
\blackdot(4,0);
\blackdot(8,0);
\blackdot(20,0);
\blackdot(24,0);
\draw(0,0) node[right]{\scriptsize$0$};
\draw(4,0) node[right]{\scriptsize$1$};
\draw(8,0) node[right]{\scriptsize$2$};
\draw[thick,dotted](13,0)--(15,0);
\draw(20,0) node[right]{\scriptsize$\ell{-}2$};
\draw(24,0) node[right]{\scriptsize$\ell{-}1$};
 \draw(-3.6,0) node{\scriptsize$\zu$};
\draw(2,1.2) node[above]{\scriptsize$\za^{1,0}$};
\draw(6,1.2) node[above]{\scriptsize$\za^{2,1}$};
\draw(10,1.2) node[above]{\scriptsize$\za^{3,2}$};
\draw(18,1.2) node[above]{\scriptsize$\za^{\ell-3,\ell-2}$};
\draw(22,1.2) node[above]{\scriptsize$\za^{\ell-1,\ell-2}$};
\draw(2,-1.2) node[below]{\scriptsize$\za^{0,1}$};
\draw(6,-1.2) node[below]{\scriptsize$\za^{1,2}$};
\draw(10,-1.2) node[below]{\scriptsize$\za^{2,3}$};
\draw(18,-1.2) node[below]{\scriptsize$\za^{\ell-3,\ell-2}$};
\draw(22,-1.2) node[below]{\scriptsize$\za^{\ell-2,\ell-1}$};
\end{braid}
\end{align*}
and define the {\em Brauer tree algebra $\Zig_\ell$} to be the path algebra of this quiver generated by length $0$ paths $\{\ze^j\mid j\in J\}$, and length $1$ paths $\zu$ and $\{\za^{k,k+1},\za^{k+1,k}\mid 0\leq k\leq \ell-2\}$, modulo the following relations:
\begin{enumerate}
\item all paths of length three or greater are zero;
\item all paths of length two that are not cycles are zero;
\item the length-two cycles based at the vertex $i\in\{1,\dots,\ell-2\}$ are equal;
\item $\zu^2=\za^{0,1}\za^{1,0}$ if $l\gs2$. 
\end{enumerate}
For example, if $\ell=1$ then the algebra $\Zig_\ell$ is the truncated polynomial algebra $\F[\zu]/(\zu^3)$. The algebra $\Zig_\ell$ is considered as a superalgebra by declaring that $\zu$ is odd and all other generators are even.
\end{Example}

\begin{Example}\label{ExWr}
For $d\in\N$, we consider the {\em wreath superproduct}
$
\W_d:=\Zig_\ell\wr \sym_d.
$
As a vector superspace this is just $\Zig_\ell^{\otimes d}\otimes \F \sym_d$, with $\F \sym_d$ concentrated in degree $\0$. The multiplication is determined by the following requirements:
\begin{enumerate}
\item[{\rm (1)}] $\bz\mapsto \bz\otimes 1$ defines a superalgebra embedding $\Zig_\ell^{\otimes d}\to\Zig_\ell^{\otimes d}\otimes \F \sym_d$; we identify $\Zig_\ell^{\otimes d}$ with a subsuperalgebra of $\W_d$ via this embedding.
\item[{\rm (2)}] $x\mapsto 1\otimes x$ defines a superalgebra embedding $\F \sym_d\to\Zig_\ell^{\otimes d}\otimes \F \sym_d$; we identify $\F \sym_d$ with a subsuperalgebra of $\W_d$ via this embedding.
\item[{\rm (3)}] $w(\zz_1\otimes\dots\otimes \zz_d)={}^{w}(\zz_1\otimes\dots\otimes \zz_d)w$ for all $w\in \sym_d$ and all $\zz_1,\dots,\zz_d\in\Zig_\ell$. 
\end{enumerate}
\end{Example}

\subsection{Supermodules}\label{SSSM}
Let $A$ be a superalgebra. An $A$-\emph{supermodule} means an $A$-module $V$ with a chosen $\ztz$-grading $V=V_\0\oplus V_\1$ such that $av\in V_{|a|+|v|}$ for all (homogeneous) $a\in A$ and $v\in V$. 

If $V$ and $W$ are $A$-supermodules then a homomorphism 
$f:V\to W$ of superspaces is a homomorphism of $A$-supermodules if $f(av)=(-1)^{|f||a|}af(v)$ for $a\in A$ and $v\in V$. 

For an $A$-supermodule $V$, the superspace $\Pi V$  is considered as an $A$-supermodule via the new action  $a\cdot v=(-1)^{|a|}av$ for $a\in A$ and $v\in\Pi V=V$. The map $\si_V:V\to\Pi V$ is then an odd isomorphism of supermodules; in particular, $\si_V(av)=(-1)^{|a|}a\cdot\si_V(v)$ for $a\in A$ and $v\in V$. 


We write `$\simeq$' for an even isomorphism of $A$-supermodules, and `$\cong$' for an arbitrary isomorphism of $A$-supermodules, cf.\ \cite[Chapter 12]{KBook}. 

A \emph{subsupermodule} of an $A$-supermodule $V$ is an $A$-submodule $W\subseteq V$ such that $W=(W\cap V_\0)\oplus (W\cap V_\1)$. An  $A$-supermodule is \emph{irreducible} if it has exactly  two subsupermodules. 

Irreducible supermodules come in two different types: an irreducible supermodule is of \emph{type $\Mtype$} if it is irreducible as a module, and of \emph{type $\Qtype$} otherwise (in which case as a module it is the direct sum of two non-isomorphic irreducible modules, see for example \cite[Section~12.2]{KBook}). Every irreducible module arises in one of these ways from an irreducible supermodule (see for example \cite[Corollary 12.2.10]{KBook}), so understanding the irreducible supermodules (together with their types) 
is essentially equivalent to understanding irreducible modules.

If $L$ is a finite-dimensional irreducible $A$-supermodule, then $L$ is  of type $\Mtype$ if and only if $L\not\simeq \Pi L$, see \cite[Lemma 12.2.8]{KBook}. 

If $A_1,\dots,A_d$ are superalgebras and $V_1,\dots,V_d$ are supermodules over $A_1,\dots,A_d$ respectively, we have a supermodule $V_1\boxtimes\dots\boxtimes V_d$ over $A_1\otimes\dots\otimes A_d$, which is $V_1\otimes\dots\otimes V_d$ as a superspace with the action defined by
\[
(a_1\otimes\dots\otimes a_d)(v_1\otimes\dots\otimes v_d)=(-1)^{\sum_{1\leq r<s\leq d}|a_s||v_r|}a_1v_1\otimes\dots\otimes a_dv_d.
\]
If $f_i:V_i\to W_i$ is an $A$-supermodule homomorphism for $i=1,\dots,d$, then 
\[
f_1\otimes \dots\otimes f_d: V_1\boxtimes\dots\boxtimes V_d\to W_1\boxtimes\dots\boxtimes W_d
\]
is a homomorphism of supermodules over $A_1\otimes\dots\otimes A_d$. In particular, \begin{equation}\label{ESiSi}
\si_{V_1}\otimes\dots\otimes \si_{V_d}:V_1\boxtimes\dots\boxtimes V_d\to (\Pi V_1)\boxtimes\dots\boxtimes (\Pi V_d)
\end{equation}
is an isomorphism of ($A_1\otimes\dots\otimes A_d$)-modules (of parity $d\pmod{2}$).

If $V$ is a finite-dimensional $A$-supermodule, a {\em composition series} of $V$ is a sequence of subsupermodules\, $0=V_0\subset V_1\subset\dots\subset V_n=V$ such that $V_k/V_{k-1}$ is an irreducible supermodule for all $k=1,\dots,n$. If $L_1,\dots,L_n$ are irreducible $A$-modules (not necessarily distinct) such that $V_k/V_{k-1}\simeq L_k$ for $k=1,\dots,n$, we say that $V$ {\em has composition factors} $L_1,\dots,L_k$. These are well defined up to even isomorphisms and permutation. So if $L$ is an irreducible $A$-supermodule we have a well-defined  {\em composition multiplicity}
\[
[V:L]:=\{k\mid L_k\cong L\}.
\]
(If $L$ is of type $\Mtype$ so that $L\not\simeq \Pi L$, we could consider the more delicate graded composition multiplicity $[V:L]_\pi=m+n\pi$ where $m=\{k\mid L_k\simeq L\}$ and $n=\{k\mid L_k\simeq \Pi L\}$, so that $[V:L]=m+n$, but this will not be needed.) 

If $A$ is a finite-dimensional superalgebra and $L$ is an irreducible $A$-supermodule, we denote the projective cover of $L$ by $P_L$. This is a direct summand of the regular supermodule with head $L$, see \cite[Proposition 12.2.12]{KBook}. The composition factors of the principal indecomposable supermodules $P_L$ will be of central importance in this paper. In particular, the \emph{super-Cartan invariants of $A$} are defined as the multiplicities 
\[
c_{L,L'}:=[P_L:L']
\]
for all irreducible $A$-supermodules  $L,L'$. The {\em super-Cartan matrix of $A$} is then the matrix $(c_{L,L'})$ of all super-Cartan invariants of $A$. 

For the superalgebra $\Zig_\ell$ of \cref{ExAA}, up to even isomorphisms and parity shifts $\Pi$, a complete set of irreducible $\Zig_\ell$-supermodules is 
\begin{equation}\label{EAIrr}
\{L_j\mid j\in J\}
\end{equation}
where $L_j$ is spanned by an even vector $v_j$ such that $\ze_jv_j=v_j$ and all other standard generators of $\Zig_\ell$ act on $v_j$ as zero. Now, note that for each $j$, the supermodule $L_j$ is of type $\Mtype$, and $P_j:=\Zig_\ell\ze_j$ is a projective cover of $L_j$. We can easily write down a basis for each $P_j$:
\begin{alignat*}2
P_0&=\<\ze_0,\zu\ze_0,\za^{1,0}\ze_0,\zu^2\ze_0\>\qquad&&\text{(omitting $\za^{1,0}\ze_0$ if $\ell=1$)},
\\
P_j&=\<\ze_j,\za^{j-1,j}\ze_j,\za^{j+1,j}\ze_j,\za^{j,j+1}\za^{j+1,j}\ze_j\>\qquad&&\text{for }1\ls j\ls\ell-2,
\\
P_{\ell-1}&=\<\ze_{\ell-1},\za^{\ell-2,\ell-1}\ze_{\ell-1},\za^{\ell-1,\ell-2}\za^{\ell-2,\ell-1}\ze_{\ell-1}\>\qquad&&\text{if $\ell\gs2$}.
\end{alignat*}
From this, we can immediately read off the composition factors of each $P_j$:

\begin{lemma}\label{pjcomp}
$P_0$ has composition factors $L_0,\Pi L_0,L_1,L_0$ (omitting $L_1$ if $\ell=1$), $P_i$ has composition factors $L_i,L_{i-1},L_{i+1},L_i$ for $1\leq i\leq \ell-2$, and $P_{\ell-1}$ has composition factors $L_{\ell-1},L_{\ell-2},L_{\ell-1}$ if $\ell\gs2$.
\end{lemma}

\subsection{Representations of wreath superproducts $\W_d$}
\label{Wd}
We suppose from now until the end of \cref{wreathsec} that $d<p$ or $p=0$. 
Our aim is to develop the representation theory of the wreath superproduct algebra  $\W_d$ from Example~\ref{ExWr}, and ultimately to compute the super-Cartan matrix for $\W_d$. We take inspiration from the paper \cite{CT1} by Chuang and Tan; many of our results are straightforward adaptations of their results to supermodules.

Given $\bj=j_1\dots j_d\in J^d$, we define the idempotent
\[
\ze_\bj:=\ze_{j_1}\otimes\dots\otimes \ze_{j_d}\in\Zig_\ell^{\otimes d}\subseteq \W_d.
\]
Then we have the orthogonal idempotent decomposition in $\W_d$:
\begin{equation}\label{EId}
1=\sum_{\bj\in J^d}e_\bj.
\end{equation}

For a composition $\de=(\de_1,\dots,\de_k)$ of $d$, we have a Young subgroup $\sym_\de=\sym_{\de_1}\times\dots\times \sym_{\de_k}\leq  \sym_d$ and the corresponding parabolic subalgebra 
\[
\W_\de:=\Zig_\ell^{\otimes d}\otimes \F \sym_{\de}\subseteq \W_d.
\]
Note that $\W_\de\cong \W_{\de_1}\otimes\dots\otimes \W_{\de_k}$ (tensor product of superalgebras). If $V_1,\dots,V_k$ are supermodules for $\W_{\de_1},\dots,\W_{\de_k}$ respectively, then we have the supermodule $V_1\boxtimes\dots\boxtimes V_k$ over  $\W_{\de_1}\otimes\dots\otimes \W_{\de_k}=\W_\de$, so we can form the $\W_d$-supermodule 
\[
V_1\circ\dots\circ V_k:=\Ind^{\W_d}_{\W_\de}(V_1\boxtimes\dots\boxtimes V_k).
\]
Note that the operation `$\circ$' is commutative in the sense that $V\circ V'\simeq V'\circ V$.

Recall that if $\la\in\Par(d)$, we write $\spe\la$ for the corresponding Specht module for $\F \sym_d$. Our assumptions on $p$ mean that $\spe\la$ is irreducible, and we can fix a primitive idempotent $f^\la\in \F \sym_d$ such that $\F \sym_d f^\la\cong \spe\la$.

Now given $\bla\in\Par^J(d)$, define $\de:=(\de_0,\dots,\de_{\ell-1})=(|\la^{(0)}|,\dots,|\la^{(\ell-1)}|)$. Then we have a primitive idempotent
\[
f^\bla:=f^{\la^{(0)}}\otimes \dots \otimes f^{\la^{(\ell-1)}}\in \F \sym_{\de_0}\otimes\dots\otimes \F \sym_{\de_{\ell-1}}=\F \sym_\de,
\]
from which we define an idempotent 
\begin{equation}\label{EELa}
e(\bla):=e_0^{\otimes \de_0}\otimes\dots\otimes e_{\ell-1}^{\otimes \de_{\ell-1}}\otimes f^\bla\in \W_d.
\end{equation}

Let $V$ be a finite-dimensional $\Zig_\ell$-supermodule and $\la\in \Par(d)$. Denote:
\[
V(\la):=V^{\otimes d}\otimes \spe\la
\]
considered as a $\W_d$-supermodule via 
\begin{align*}
\bz(v_1\otimes\dots\otimes v_d\otimes y)&=\bz(v_1\otimes\dots\otimes v_d)\otimes y,
\\
w(v_1\otimes\dots\otimes v_d\otimes y)&={}^w(v_1\otimes\dots\otimes v_d)\otimes wy
\end{align*}
for all $\bz\in \Zig_\ell^{\otimes x}$, $w\in \sym_d$, $v_1,\dots,v_d\in V$, $y\in \spe\la$. Important special cases of this construction where $V=L_j$ and $V=P_j$ are the simple $\Zig_\ell$-module and its projective cover constructed in \cref{SSSM}, yield the $\W_d$-supermodules $L_j(\la)$ and $P_j(\la)$. For a general $V$ we have the following two results.

\begin{lemma} \label{LSign} 
Let $V$ be a finite-dimensional $\Zig_\ell$-supermodule and $\la\in \Par(d)$. Then $(\Pi V)(\la)\cong V(\la')$.
\end{lemma}
\begin{proof}
By (\ref{ESpechtSign}), we have $\spe{\la'}\cong \spe\la\otimes \sgn$,\, so we can identify $\spe{\la'}$ with $\spe\la$ as vector spaces but with the new action $w\cdot d=\sgn(w)wd$. Now, we consider the linear isomorphism 
\[
\phi:=\si_V^{\otimes d}\otimes\id:V(\la')=V^{\otimes d}\otimes \spe{\la'}\iso
(\Pi V)^{\otimes d}\otimes \spe{\la}=(\Pi V)(\la).
\]
As pointed out in (\ref{ESiSi}), $\phi$ is an isomorphism of $\Zig_\ell^{\otimes d}$-supermodules. On the other hand, for $v_1,\dots,v_d\in V$, $y\in \spe{\la'}$ and $w\in \sym_d$, we have 
\begin{align*}
\phi\big(w(v_1\otimes\dots\otimes v_d\otimes y)\big)&=\phi\big({}^w(v_1\otimes \dots\otimes v_d)\otimes (\sgn(w)wy)\big)
\\
&=\sgn(w)\si_V^{\otimes d}\big({}^w(v_1\otimes \dots\otimes v_d)\big)\otimes wy
\\
&={}^w\big(\si_V^{\otimes d}(v_1\otimes \dots\otimes v_d)\big)\otimes wy
\\
&=w\phi(v_1\otimes \dots\otimes v_d\otimes y)
\end{align*}
where we use (\ref{ESignW}) for the penultimate equality. So $\phi$ is also an isomorphism of $\F \sym_d$-modules. It follows that $\phi$ is an isomorphism of $\W_d$-supermodules. 
\end{proof}

\begin{lemma} \label{LResInd} 
Let $V$ be a finite-dimensional $\Zig_\ell$-supermodule, and $\de=(\de_1,\dots,\de_k)$ be a composition of~$d$. 
\begin{enumerate}
\item For $\la\in \Par(d)$, we have 
\[
\Res^{\W_d}_{\W_\de}V(\la)\cong\bigoplus_{\mu^1\in\Par(\de_1),\dots,\mu^k\in\Par(\de_k)}\big(V(\mu^1)\boxtimes\dots\boxtimes V(\mu^k)\big)^{\oplus \lr\la{\mu^1,\dots,\mu^k}}.
\]
\item For $\mu^1\in\Par(\de_1),\dots,\mu^k\in\Par(\de_k)$, we have 
\[
V(\mu^1)\circ\dots\circ V(\mu^k)\cong \bigoplus_{\la\in \Par(d)}V(\la)^{\oplus \lr\la{\mu^1,\dots,\mu^k}}.
\]
\end{enumerate} 
\end{lemma}
\begin{proof}
The proof is identical to that of \cite[Lemma 3.3]{CT1} paying attention to the superalgebra signs.  
\end{proof}

Given $\bla=(\la^{(0)},\dots,\la^{(\ell-1)})\in\Par^J(d)$, we now define the $\W_d$-supermodules
\begin{align*}
L(\bla)&:=L_0(\la^{(0)})\circ\dots\circ L_{\ell-1}(\la^{(\ell-1)}),
\\
P(\bla)&:=P_0(\la^{(0)})\circ\dots\circ P_{\ell-1}(\la^{(\ell-1)}).
\end{align*}

\begin{propn}
The set $\{L(\bla)\mid \bla\in\Par^J(d)\}$ is a complete irredundant set of irreducible $\W_d$-supermodules up to even isomorphism and parity shift. Moreover, $P(\bla)$ is a projective cover of $L(\bla)$ for each $\bla\in\Par^J(d)$.
\end{propn}
\begin{proof}
The first statement is easy to see and is well-known, see e.g. \cite[Theorem A.5]{MacD}. For the second statement, note using Frobenius reciprocity that  
$
P(\bla)\simeq \W_de(\bla)
$
for the idempotent $e(\bla)\in \W_d$ defined in (\ref{EELa}). 
We now also deduce that $\dim\Hom_{\W_d}(P(\bla),L(\bmu))=\de_{\bla,\bmu}$ completing the proof. 
\end{proof}

Now we  determine the composition factors of the modules $L(\bla^1)\circ\cdots\circ L(\bla^k)$.

\begin{lemma} \label{LLCircL} 
Let $\bmu\in\Par^J(d)$ and let $(\de_1,\dots,\de_k)$ be a composition of $d$. 
For $r=1,\dots,k$, suppose $\bla^r=(\la^{(r,0)},\dots,\la^{(r,\ell-1)})\in \Par^J(\de_r)$. Then 
\[
[L(\bla^1)\circ\cdots\circ L(\bla^k):L(\bmu)]=
\prod_{j\in J}\lr{\mu^{(j)}}{\la^{(1,j)},\dots,\la^{(k,j)}}.
\]
\end{lemma}
\begin{proof}
This follows from \cref{LResInd}(ii) using commutativity of `$\circ$'. 
\end{proof}
 
\subsection{The super-Cartan matrix for $\W_d$}
In this subsection we continue to assume that $d<p$ or $p=0$.  Having explicitly constructed the irreducible and projective indecomposable supermodules for $\W_d$, we now proceed to  compute its super-Cartan invariants. 

\begin{lemma} \label{LTwoFactors} 
Let $V,W$ be finite-dimensional $\Zig_\ell$-supermodules and $U$ be a subsupermodule of $V$ such that $V/U\simeq W$. 
Then for $\la\in\Par(d)$, the $\W_d$-supermodule $V(\la)$ has a filtration with subfactors $U(\mu)\circ W(\nu)$ each appearing exactly $\lr\la{\mu,\nu}$ times. 
\end{lemma}
\begin{proof}
For $0\leq c\leq d$, we denote by $V_c$ the subsupermodule of $V(\la)=
V^{\otimes d}\otimes \spe\la$ spanned by the vectors of the form $v_1\otimes\dots\otimes v_d\otimes x$ such that at least $c$ of the vectors $v_1,\dots,v_d\in V$ belong to $U$ and $x\in \spe\la$. 
This gives a filtration $V(\la)=V_0\supseteq V_{1}\supseteq\dots\supseteq V_d\supseteq V_{d+1}=0$ with 
\[
\frac{V_c}{V_{c+1}}\cong \bigoplus_{\substack{\mu\in\Par(c)\\\nu\in\Par(d-c)}}(U(\mu)\circ W(\nu))^{\oplus\lr\la{\mu,\nu}},
\]
cf.\ the proof of \cite[Lemma 4.2]{CT1}.
\end{proof}

\begin{lemma} \label{LOneColor} 
Let $\la\in\Par(d)$, and $V$ be an $\Zig_\ell$-supermodule with composition series 
\[
V=V_{0}\supset V_{1}\supset\dots\supset V_{m+1}=0.
\]
Set $K:=\{0,\dots,m\}$. For $\bnu=(\nu^{(0)},\dots,\nu^{(m)})\in\Par^K(d)$ and $j\in J$, define multisets
\begin{align*}
M(j,\bnu)&:=\{\nu^{(k)}\mid k\in K,\ V_{k}/V_{k+1}\simeq L_j\}   
\\
M'(j,\bnu)&:=\{(\nu^{(k)})'\mid k\in K,\ V_{k}/V_{k+1}\simeq \Pi L_j\}.
\end{align*}
Then for any $\bmu=(\mu^{(0)},\dots,\mu^{(\ell-1)})\in\Par^J(d)$, we have  
\[
[V(\la):L(\bmu)]=\sum_{\bnu\in\Par^K(d)}
\lr\la{\nu^{(0)},\dots,\nu^{(m)}}\,\prod_{j\in J}\lr{\mu^{(j)}}{M(j,\bnu)\sqcup M'(j,\bnu)}.
\]
\end{lemma}
\begin{proof}
This follows by induction from \cref{LTwoFactors}, using \cref{LResInd,LSign}. 
\end{proof}

The following result is a `superversion' of \cite[Proposition 4.4]{CT1}.

\begin{propn} \label{PCTMain} 
Let $V_0,\dots, V_{\ell-1}$ be finite-dimensional $\Zig_\ell$-supermodules and $\bla=(\la^{(0)},\dots,\la^{(\ell-1)})\in\Par^J(d)$. Set
\[
V(\bla):=V_0(\la^{(0)})\circ\dots\circ V_{\ell-1}(\la^{(\ell-1)}).
\]
Let $V_j=V_{j,0}\supset V_{j,1}\supset\dots\supset V_{j,m_j+1}=0$ be a composition series of $V_j$ for each $j\in J$. Set 
\[
K :=\{(j,s)\in J\times\N_0\mid s\leq m_j\}.
\]
For $i\in J$ and $\bnu\in\Par^K(d)$, let
\begin{align*}
M(i,\bnu)&:=\{\nu^{(j,s)}\mid (j,s)\in K\ \text{and}\ V_{j,s}/V_{j,s+1}\simeq L_i\}   
\\
M'(i,\bnu)&:=\{(\nu^{(j,s)})'\mid  (j,s)\in K\ \text{and}\ V_{j,s}/V_{j,s+1}\simeq \Pi L_i\}.
\end{align*}
Then for any $\bmu=(\mu^{(0)},\dots,\mu^{(\ell-1)})\in\Par^J(d)$, we have 
\[
[V(\bla):L(\bmu)]=\sum_{\bnu\in\Par^K(d)}
\prod_{j\in J}\lr{\la^{(j)}}{\nu^{(j,0)},\dots,\nu^{(j,m_j)}}\,\lr{\mu^{(j)}}{M(j,\bnu)\sqcup M'(j,\bnu)}.
\]
\end{propn}
\begin{proof}
For $j\in J$, we set 
\begin{align*}
M(i,\bnu,j)&:=\{\nu^{(j,s)}\mid 0\leq s\leq m_j\ \text{and}\ V_{j,s}/V_{j,s+1}\simeq L_i\}   
\\
M'(i,\bnu,j)&:=\{(\nu^{(j,s)})'\mid  0\leq s\leq m_j\ \text{and}\ V_{j,s}/V_{j,s+1}\simeq \Pi L_i\},
\end{align*}
so that $M(i,\bnu)=\bigsqcup_{j\in J}M(i,\bnu,j)$ and $M'(i,\bnu)=\bigsqcup_{j\in J}M'(i,\bnu,j)$. 

By \cref{LOneColor} , for each $j\in J$, setting $\de_j:=|\la^{(j)}|$, we have in the Grothendieck group
\[
[V(\la^{(j)})]=\sum_{\bmu^j\in\Par^J(\de_j)}q_{\bmu^j} [L(\bmu^j)],
\]
where
\[
q_{\bmu^j}=\sum_{\nu^{(j,0)},\dots,\nu^{(j,m_j)}}
\lr{\la^{(j)}}{\nu^{(j,0)},\dots,\nu^{(j,m_j)}}\,\prod_{i\in J}\lr{\mu^{(j,i)}}{M(i,\bnu,j)\sqcup M'(i,\bnu,j)}.
\]
Now, 
\begin{align*}
[V(\bla)]&=[V(\la^{(0)})\circ\dots\circ V_{\ell-1}(\la^{(\ell-1)})]
\\
&=
\sum_{\bmu^0\in\Par^J(\de_0),\dots,\bmu^{\ell-1}\in\Par^J(\de_{\ell-1})}q_{\bmu^0}\dots q_{\bmu^{\ell-1}}[L_{\bmu^0}\circ\cdots\circ L_{\bmu^{\ell-1}}].
\end{align*}
It remains to apply \cref{LLCircL} and use the following  identity involving Littlewood--Richardson coefficients: 
\[
\lr{\mu^{(i)}}{M(i,\bnu)\sqcup M'(i,\bnu)}=\lr{\mu^{(i)}}{\mu^{(0,i)},\dots,\mu^{(\ell-1,i)}}\prod_{j\in J}\lr{\mu^{(j,i)}}{M(i,\bnu,j)\sqcup M'(i,\bnu,j)},
\]
which in turn follows from the description of the Littlewood--Richardson coefficient in terms of induction for symmetric groups  using the transitivity of induction. 
\end{proof}

\begin{cory} \label{wrcartan}
Let $\bla,\bmu\in\Par^J(d)$. Then
\[
[P(\bla):L(\bmu)]=\sum 
\prod_{j\in J}\lr{\mu^{(j)}}{\al^{(j)},\be^{(j+1)},\ga^{(j-1)},\de^{(j)}}\,\lr{\la^{(j)}}{\al^{(j)},\be^{(j)},\ga^{(j)},\de^{(j)}},
\]
where the summation is over all partitions $\al^{(i)},\be^{(i)},\ga^{(i)},\de^{(i)}$ with $i\in J$, reading $\ga^{(-1)}=(\be^{(0)})'$ and $\be^{(\ell)}=\ga^{(\ell-1)}=\vn$. (If $\ell=1$ this formula is interpreted as 
$
c_{\la,\mu}=\sum_{\al,\be,\de}\lr\mu{\al,\be',\de}\lr\la{\al,\be,\de}
$.) 
\end{cory}
\begin{proof}
Apply \cref{PCTMain} to the case $V(\bla)=P(\bla)$, using \cref{pjcomp}.
\end{proof} 

\section{Representations of double covers of symmetric groups}
\label{SRSNHat}

\subsection{The double cover of the symmetric group}

Let $\hsym_n$ denote a proper double cover of the symmetric group $\sym_n$. Then $\hsym_n$ contains a central element $z$ of order $2$, with $\hsym_n/\<z\>\cong\sym_n$.

The central involution $z$ yields a central idempotent $e_z=\frac12(1-z)$, and direct sum decomposition
\[
\F\hsym_n=e_z\F\hsym_n\oplus(1-e_z)\F\hsym_n.
\]
The algebra $(1-e_z)\F\hsym_n$ is isomorphic to $\F\sym_n$, so we concentrate here on representations of $e_z\F\hsym_n$, often called the \emph{spin representations} of $\sym_n$. We identify $e_z\F\hsym_n$ with the twisted group algebra $\T_n$, see \cite[Section 13.1]{KBook}, where a superalgebra structure is defined on $\T_n$ by letting $e_z\si$ being even or odd depending on whether the image of $\si$ in $\sym_n$ is even or odd.

The classification of irreducible spin supermodules in characteristic $0$ goes back to Schur (though Schur worked with modules rather than supermodules, and only constructed characters; the corresponding modules were constructed much later, by Nazarov \cite{Na}). For each strict partition $\la$ of $n$ there is an irreducible spin supermodule $\SpC\la$ for $\C\hsym_n$, and $\lset{\SpC\la}{\la\in\sPar(n)}$ is a complete irredundant set of irreducible spin supermodules. 
Moreover, recalling (\ref{EA}), the supermodule $\SpC\la$ is of type $\Mtype$ if $\a\la=0$, and of type $\Qtype$ if $\a\la=1$.

The classification of irreducible supermodules in characteristic $p$ is due to Brundan and the second author \cite{bk1}. 
(Another classification is obtained in \cite{BKHeCl}, and \cite[Theorem B]{ks} shows that the two classifications agree.) 
For each restricted $p$-strict partition $\mu$ of $n$, there is an irreducible $\T_n$-supermodule $\Dm\mu$, and $\lset{\Dm\mu}{\mu\in\rPar(n)}$ is a complete irredundant set of irreducible $\T_n$-supermodules. Moreover, $\Dm\mu$ is of type $\Mtype$ if $\mu$ has an even number of nodes of non-zero residue, and of type $\Qtype$ otherwise.

Since we shall be interested exclusively in representations in characteristic $p$, we use the notation $\Sp\la$ for a $p$-modular reduction of $\SpC\la$, viewed as a $\T_n$-supermodule. Note that $\Sp\la$ is not well-defined as a supermodule, but its composition factors are. The (super) \emph{decomposition number problem} then asks for the composition multiplicities $[\Sp\la:\Dm\mu]$ for $\la\in\sPar(n)$ and $\mu\in\rPar(n)$.

The block classification for spin modules is due to Humph\-reys~\cite{humph}. Here we prefer to deal with spin \emph{superblocks}, i.e.\ indecomposable direct summands of $\T_n$ as a superalgebra; in fact blocks and superblocks coincide except in the trivial case of simple blocks, so we ignore this distinction, and say `block' to mean `superblock', see \cite[\S5.2b]{kl} for more details on this. With this convention, each $\Sp\la$ belongs to a single block, and the $\T_n$-supermodules $\Sp\la$ and $\Dm\mu$ lie in the same block \iff $\la$ and $\mu$ have the same $p$-bar-core. This automatically means that they have the same $p$-bar-weight, so blocks are labelled by pairs $(\rho,d)$, where $\rho$ is a $p$-bar-core and $d\in\N_0$ with $|\rho|+pd=n$. We write $\blk\rho d$ for the block corresponding to the pair $(\rho,d)$.

An alternative statement of the block classification can be given using residues: in view of \cite[Theorem 5]{my}, two $p$-strict partitions of $n$ have the same $p$-bar-core \iff they have the same number of $i$-nodes for each $i\in I$. 
So we may alternatively label a block of $\T_n$ with a multiset consisting of $n$ elements of $I$, corresponding to the residues of the nodes of any partition labelling an irreducible module in the block. We write $\mathcal{B}_S$ for the block labelled by the multiset $S$. An important consequence of this is that all the irreducible supermodules in a block have the same type; so we say that a block has type $\Mtype$ or $\Qtype$ accordingly.

We also have a double cover $\halt_n\subseteq \hsym_n$ of the alternating group whose twisted group algebra $e_z\F\halt_n$ can be identified with the even component $(\T_n)_\0$. 
Moreover, by \cite[Proposition~3.16]{Kessar}, the even component $\blk \rho d_\0$ of $\blk \rho d$ is a single block of $\F \halt_n$, unless $d=0$ and $\blk \rho d$ is of type $\Mtype$. We refer the reader to \cite[\S5.2b]{kl} for more on this.

\subsection{Branching rules and weights}\label{branchsec}

The block classification using multisets of residues allows us to define restriction and induction functors $E_i$ and $F_i$. Suppose $M$ is a $\T_n$-supermodule lying in the block $\mathcal{B}_S$. Given $i\in I$, we define a $\T_{n+1}$-module $F_iM$ by inducing $M$ to $\T_{n+1}$ and then taking the block component lying in the block $\mathcal{B}_{S\sqcup\{i\}}$ (if there is such a block; otherwise we set $F_iM:=0$). The restriction functor $E_i$ is defined in a similar way by restricting to $\T_{n-1}$ and removing a copy of $i$ from $S$. The functors $E_i,F_i$ (which are called $\operatorname{res}_i$ and $\operatorname{ind}_i$ in \cite[(22.17),(22.18)]{KBook}) are defined for all $n$, so we can consider powers $E_i^r,F_i^r$ for $r\gs0$. 

Given $\la\in\sPar(n)$, let $M(\la,i)$ be the set of strict partitions of $n+1$ which can be obtained by adding an $i$-node to $\la$. Then, in view of \cite[Theorem 3]{my2}, in the Grothendieck group of $\T_{n+1}$ we have
\begin{equation}\label{EFi}
[F_i\Sp\la]=\sum_{\mu\in M(\la,i)}a_{\la\mu}[\Sp\mu],
\end{equation}
where $a_{\la\mu}$ equals $2$ if $\la$ is odd and $\mu$ is even, and $1$ otherwise. Frobenius reciprocity yields a corresponding result for $E_i\Sp\la$. (This description of $[E_i\Sp\la]$ and $[F_i\Sp\la]$ can also be deduced by considering the $p>n$ case of \cite[Theorems~22.3.4,~22.3.5]{KBook}.) 

We can now apply  the operators $E_i$ and $F_i$ to characters of supermodules (either ordinary characters or $p$-modular Brauer characters) as well as to supermodules. For example, if $\chi^\la$ denotes the character of an irreducible supermodule $S_\C(\la)$, we define $F_i\chi^\la=\sum_{\mu\in M(\la,i)}a_{\la\mu}\chi^\mu$. We define $E_i\chi^\la$ similarly.

The modular branching rules of Brundan--Kleshchev and Kleshchev--Shchigolev give information on the modules $E_i\Dm\mu$. We just need one result, and to state this we need some more combinatorics. Suppose $\mu$ is a $p$-strict partition and $i\in I$. Let $\mu^-$ denote the smallest $p$-strict partition such that $\mu^-\subseteq\mu$ and $\mu\setminus\mu^-$ consists of $i$-nodes. These nodes are called the \emph{removable} $i$-nodes of $\mu$. Similarly, let $\mu^+$ denote the largest $p$-strict partition such that $\mu^+\supseteq\mu$ and $\mu^+\setminus\mu$ consists of $i$-nodes. These nodes are called the \emph{addable} $i$-nodes of $\mu$.

The \emph{$i$-signature} of $\mu$ is the sequence of signs obtained by listing the addable and removable $i$-nodes of $\mu$ from left to right, writing a $+$ for each addable $i$-node and a $-$ for each removable $i$-node. The \emph{reduced $i$-signature} is the subsequence obtained by successively deleting adjacent pairs $+-$. The removable nodes corresponding to the $-$ signs in the reduced $i$-signature are called the \emph{normal} $i$-nodes of $\mu$.

The result we will need below is the following (see \cite[Theorem A(ii)]{ks}).

\begin{lemma}\label{normal}
Suppose $\mu\in\rPar(n)$ and $\nu\in\rPar(n-1)$, and that $\nu$ is obtained from $\mu$ by removing a normal $i$-node. Then $\Dm\nu$ is a composition factor of $E_i\Dm\mu$.
\end{lemma}

Now given a $\T_n$-supermodule and a word $\bi=i_1\dots i_n\in I^n$, we say that $\bi$ is a \emph{weight} of $M$ if $E_{i_1}\dots E_{i_n}M\neq0$. The fact that the functors $E_i$ are exact, together with the results above, yields the following.

\begin{propn}\label{wtprop}
Suppose $i\in I$ and $i_1\dots i_{n-1}\in I^{n-1}$.
\begin{enumerate}
\item
Suppose $\la\in\sPar(n)$ and $\mu\in\sPar(n-1)$ is obtained from $\la$ by removing an $i$-node. If $i_1\dots i_{n-1}$ is a weight of $\Sp\mu$, then $i_1\dots i_{n-1}i$ is a weight of $\Sp\la$.
\item
Suppose $\mu\in\rPar(n)$ and $\nu\in\rPar(n-1)$ is obtained from $\mu$ by removing a normal $i$-node. If $i_1\dots i_{n-1}$ is a weight of $\Dm\nu$, then $i_1\dots i_{n-1}i$ is a weight of $\Dm\mu$.
\end{enumerate}
\end{propn}

For (much) more information on branching rules for $\T_n$, see \cite[Part II]{KBook} and\cite{ks}.

\subsection{Virtual projective characters}
Given $\la\in\stp\rho d$, we write $\ch\la$ for the character of the irreducible supermodule $\SpC\la$, and we denote by $\operatorname{Ch}^{\rho,d}$ the $\Q$-span of the set $\{\ch\la\mid  \la\in\stp\rho d\}$ of class functions on $\hat\Si_{|\rho|+dp}$. 

For each $\mu\in\rpstp\rho d$ we have an indecomposable projective supermodule $P(\mu)$ with simple head $D(\mu)$. Lifting the idempotents as in the classical theory we deduce that $P(\mu)$ lifts to characteristic zero, yielding the character  
$\prj\mu\in \operatorname{Ch}^{\rho,d}$. We denote by $\operatorname{PCh}^{\rho,d}$ the $\Q$-span of the set $\{\prj\mu\mid  \mu\in\rpstp\rho d\}$ and refer to the elements of $\operatorname{PCh}^{\rho,d}$ as \emph{virtual projective characters}. 

Note that $\{\ch\la\mid \la\in\stp\rho d\}$ is a basis for $\operatorname{Ch}^{\rho,d}$ since each $\ch\la$ is either an irreducible character or a sum of two irreducible characters $
\ch{{\la,+}}+\ch{{\la,-}}$, and all the irreducible characters $\ch{{\la,\pm}},\ch\mu$ are distinct (cf.\ \cite[Corollary 12.2.10]{KBook}). 
Moreover, $\{\prj\mu\mid  \mu\in\rpstp\rho d\}$ is a basis for $\operatorname{PCh}^{\rho,d}$. This is proved as for the $\ch\la$. First, note that each $\prj\mu$ is either an indecomposable projective character or a sum of two indecomposable projective characters $
\prj{{\la,+}}+\prj{{\la,-}}$, and all the indecomposable projective  characters $\prj{{\la,\pm}},\prj\mu$ are distinct in view of \cite[Proposition~12.2.12 and Lemma~12.2.16]{KBook}. Then 
use linear independence of the indecomposable projective characters \cite[Theorem 18.26(iii)]{CR}.

Given $\phi=\sum_{\la\in\stp\rho d}a_\la \ch\la\in \operatorname{Ch}^{\rho,d}$, we write the coefficient $a_\la$ as $\chm{\phi}{\ch\la}$. We say that $\ch\la$ \emph{occurs} in $\phi$ if $\chm{\phi}{\ch\la}$ is non-zero.  Below we will use a superversion of Brauer reciprocity to compute decomposition numbers for $\blk\rho d$ in terms of the multiplicities $\chm{\prj\mu}{\ch\la}$:
\begin{equation}\label{EBrRec}
[\Sp\la:\Dm\mu]=
\begin{cases}
2[\prj\mu:\ch\la]&\text{if $\SpC\la$ is of type $\Qtype$ and $\Dm\mu$ is of type $\Mtype$},
\\
\tfrac12[\prj\mu:\ch\la]&\text{if $\SpC\la$ is of type $\Mtype$ and $\Dm\mu$ is of type $\Qtype$},
\\
[\prj\mu:\ch\la]&\text{otherwise}.
\end{cases}
\end{equation}
This follows from the classical Brauer reciprocity taking into account that when $\SpC\la$ is of type $\Qtype$ we have $\ch\la=\ch{{\la,+}}+\ch{{\la,-}}$, and when $\Dm\mu$ is of type $\Qtype$ we have $\prj\la=\prj{{\la,+}}+\prj{{\la,-}}$, and moreover, $\Dm\mu\cong \Dm{\mu,+}\oplus\Dm{\mu,-}$ for non-isomorphic irreducible modules $\Dm{\mu,+}$ and $\Dm{\mu,-}$ obtained from each other by tensoring with sign. 

\subsection{Projective characters from the $q$-deformed Fock space}\label{mprjsec}

Leclerc and Thibon \cite{lt} show how one can use canonical basis vectors to obtain another basis for the space 
$\operatorname{PCh}^{\rho,d}$;   
we briefly outline the background. Let $q$ be an indeterminate. The \emph{level-$1$ Fock space} of type $A^{(2)}_{2l}$ is a $\bbq(q)$-vector space $\calf$ with a \emph{standard basis}
\[
\lset{\ket\la}{\la\text{ a $p$-strict partition}}.
\]
This space is naturally a module for the quantum group $U_q(A^{(2)}_{2l})$. 
We note that the conventions for residues (and for simple roots in type $A^{(2)}_{2l})$ used here are as in \cite{kl},\cite{mattconj} and differ from those in  \cite{lt}.    
The submodule of $\calf$ generated by the vector $\ket\vn$ possesses a \emph{canonical basis}
\[
\lset{G(\mu)}{\mu\text{ a restricted $p$-strict partition}}.
\]
Expanding the canonical basis vectors in terms of the standard basis, one obtains the \emph{$q$-decomposition numbers} $d_{\la\mu}(q)$, indexed by pairs of $p$-strict partitions $\la,\mu$ with $\mu$ restricted:
\[
G(\mu)=\sum_{\la\text{ $p$-strict}}d_{\la\mu}(q)\ket\la.
\]
In fact \cite[Theorem 4.1]{lt} implies that $d_{\la\mu}(q)$ is zero unless $\la$ and $\mu$ have the same $p$-bar-core and the same size, so for $\mu\in\rpstp\rho d$ we actually have 
\begin{equation}\label{EDSameBlock}
G(\mu)=\sum_{\la\in\pstp\rho d }d_{\la\mu}(q)\ket\la.
\end{equation}
By \cite[Theorem 4.1(i)]{lt} each $d_{\la\mu}(q)$ is a polynomial in $q$ with integer coefficients. So given a strict partition $\la$ and a restricted $p$-strict partition $\mu$, recalling (\ref{EA}), we can define the integers
\[
D_{\la\mu}=2^{\lfloor\frac12(h_p(\la)+1-\a\la)\rfloor}d_{\la\mu}(1),
\]
where $h_p(\la)$ denotes the number of positive parts of $\la$ that are divisible by $p$. Then the discussion in \cite[Section 6]{lt} shows the following.

\begin{propn}\label{fromcanon}
Suppose $\mu$ is a restricted $p$-strict partition of $n$. Then the character
\begin{equation}
\label{EHatPhi}
\mprj\mu:=\sum_{\la\text{ strict}}D_{\la\mu}\ch\la
\end{equation}
is a virtual projective character of $\hsym_n$. Moreover, 
$\{\mprj\mu\mid  \mu\in\rpstp\rho d\}$ is a basis for $\operatorname{PCh}^{\rho,d}$. 
\end{propn}

In fact, the character $\mprj\mu$ coincides with $\prj\mu$ quite often, and our main aim in this paper is to show that $\mprj\mu=\prj\mu$ when $\mu\in\rpstp\rho d$ and $\blk{\rho}{d}$ is an abelian defect RoCK block. 

\subsection{RoCK blocks for double covers and the Kleshchev--Livesey Morita equivalence}
\label{moritasec}

Now, following \cite{kl}, we can define {\em RoCK blocks}: given a $p$-bar-core $\rho$ and $d\gs0$, we say that $\blk\rho d$ is a \emph{RoCK block} if $\rho$ is $d$-Rouquier. The term `RoCK' is borrowed from the corresponding theory for (non-spin) representations of symmetric groups, and stands for `Rouquier or Chuang--Kessar'. 

The definition of spin RoCK blocks is a natural analogue of the non-spin situation, and we expect that RoCK blocks will play a similarly important role. This has already begun with the use of RoCK blocks in proving Brou\'e's conjecture for double covers \cite{kl,bk3,elv}. Our purpose in this paper is to emulate the work of Chuang and Tan in the non-spin case and find the decomposition numbers for RoCK blocks.

Recall the material of \cref{wreathsec}, in particular, the wreath superproduct $\W_d=\Zig_\ell\wr\sym_d$. One of the main results of \cite{kl} is a Morita superequivalence relating a RoCK (super)block $\blk\rho d$ with $d<p$ and $\W_d$. This easily implies the following theorem:

\begin{thm}\label{mainkl}
Suppose\, $1\ls d<p$, and $\rho$ is a $d$-Rouquier $p$-bar-core. Then we have a Morita equivalence 
\[
\W_d\sim_{{\Mor}}
\begin{cases}
\blk\rho d&\text{if $\blk\rho d$ is of type $\Mtype$,}
\\
\blk\rho d_\0&\text{if $\blk\rho d$ is of type $\Qtype$.}
\end{cases}
\]
\end{thm}
\begin{proof}
By \cite[Proposition 5.4.10(i)]{kl}, we have a Morita {\em superequivalence} 
\[
\blk\rho d\sim_{{\sMor}}
\begin{cases}
\W_d&\text{if $\blk\rho d$ is of type $\Mtype$,}
\\
\W_d\otimes\CL_1&\text{if $\blk\rho d$ is of type $\Qtype$.}
\end{cases}
\]
where $\CL_1$ is the Clifford superalgebra of rank $1$. 
If $\blk\rho d$ is of type $\Mtype$ the result follows immediately since Morita superequivalence implies Morita equivalence, see \cite[\S2.2c]{kl}. If $\blk\rho d$ is of type $\Qtype$, then we obtain $\blk\rho d\otimes\CL_1\sim_{{\sMor}}\W_d\otimes\CL_1\otimes\CL_1\simeq\W_d\otimes\CL_2$, and we apply \cite[Lemmas 2.2.19 and 2.2.20]{kl}. 
\end{proof}

\subsection{The regularization theorem}

One of the early general results concerning decomposition numbers for symmetric groups is James's regularization theorem \cite{j1}. Later we will need the analogue for spin modules, which was proved by Brundan and the second author \cite[Theorem 1.2]{bk2}. They define (in a combinatorial way) a function $\la\mapsto\la^\reg$ from $\sPar(n)$ to $\rPar(n)$ and prove the following statement.

%
%

\begin{thm}\label{regthm}
Suppose $\la$ is a strict partition. Then $\Dm{\la^\reg}$ occurs as a composition factor of $\Sp\la$, and $\Dm{\nu}$ is a composition factor of $\Sp\la$ only if $\la^\reg\dom\nu$. 
\end{thm}

We will not need the exact definition of regularization, since we use an alternative description of regularization in RoCK blocks, as follows.

\begin{lemma}\label{rockreg}
Suppose $\rho$ is a $d$-Rouquier $p$-bar-core, and $\la\in\stp\rho d$ with \pbq $(\la^{(0)},\dots,\la^{(\ell)})$. Then $\la^\reg$ is the partition in $\rpstp\rho d$ with \pbq
\[
(\la^{(0)},\dots,\la^{(\ell-2)},\la^{(\ell-1)}+{\la^{(\ell)}}',\vn).
\]
\end{lemma}

\cref{rockreg} is not very hard to prove directly from the combinatorial definition of $\la^\reg$, but we will give a proof using canonical basis coefficients in \cref{mprjsec}.

\section{Projective characters}\label{projcharsec}

Having summarized all the background we need, we now work towards our main result. Throughout this section we fix an integer $d\gs1$ and a $d$-Rouquier $p$-bar-core $\rho$. Our aim is to work with projective characters in $\blk\rho d$; our main result in this section is to find the decomposition matrix for $\blk\rho d$ up to multiplying by a non-negative unitriangular matrix. Note that the results of this section do not require $d<p$.

\subsection{Projective characters $\mprj\mu$ in RoCK blocks}\label{mprjsecRoCK}
Recall the virtual projective characters $\mprj\mu$ defined in  (\ref{EHatPhi}). 
One of the main results of the first author's paper \cite{mattconj} is an explicit determination of the canonical basis vectors $G(\mu)$ for partitions in RoCK blocks. As a result of this, we can give the characters $\mprj\mu$ in $\blk\rho d$ explicitly. 

First, we give the formula for the canonical basis coefficients in a weight space corresponding to a RoCK block. Recall the notation of \cref{SSKostkaPol}.

\begin{thm}[{\!\cite[Theorem 8.2]{mattconj}}]\label{fockrock}
Suppose $\rho$  is a $d$-Rouquier $p$-bar-core, $\la\in\pstp\rho d$ and $\mu\in\rpstp\rho d$. Then
\[
d_{\la\mu}(q)=\sum K^{-1}_{\la^{(0)}\si^{(0)}}(-q^2)\prod_{i=1}^{\ell}\lr{\la^{(i)}}{\si^{(i)},\tau^{(i)}}\lr{\mu^{(i-1)}}{\si^{(i-1)},{\tau^{(i)}}'}q^{2\sum_{i\in I}i(|\la^{(i)}|-|\mu^{(i)}|)},
\]
where the sum is over all partitions $\si^{(0)},\dots,\si^{(\ell-1)},\tau^{(1)},\dots,\tau^{(\ell)}$, and we read $\si^{(\ell)}$ as $\vn$.
\end{thm}

As a consequence we can write down the characters $\mprj\mu$ in RoCK blocks; this follows from \cref{fockrock}, (\ref{EDSameBlock}) and the definition (\ref{EHatPhi}).

\begin{cory}\label{mattproj}
Suppose $\rho$ is a $d$-Rouquier $p$-bar-core and $\mu\in\rpstp\rho d$. 
Then
\[
\mprj\mu=\sum_{\la\in\stp\rho d}2^{\lfloor\frac12(h(\la^{(0)})+1-\a\la)\rfloor}
\sum
\ki{\la^{(0)}}{\si^{(0)}}\prod_{i=1}^{\ell}\lr{\la^{(i)}}{\si^{(i)},\tau^{(i)}}\lr{\mu^{(i-1)}}{\si^{(i-1)},{\tau^{(i)}}'}\ch\la,
\]
where the second sum is over all partitions $\si^{(0)},\dots,\si^{(\ell-1)},\tau^{(1)},\dots,\tau^{(\ell)}$, and we read $\si^{(\ell)}$ as $\vn$.
\end{cory}

\begin{cory}\label{mattprojPosInt}
Suppose $\rho$ is a $d$-Rouquier $p$-bar-core and $\mu\in\rpstp\rho d$. 
Then $\mprj\mu$ is a non-negative integral linear combination of irreducible characters $\ch\la$ with $\la\in\stp\rho d$.
\end{cory}
\begin{pf}
Follows from  \cref{mattproj} and \cref{kostlem}.
\end{pf}


We now use \cref{fockrock} to give the deferred proof of \cref{rockreg}. 
This relies on the following regularization theorem for canonical basis coefficients.

\begin{thm}[{\!\cite[Theorem 3.2]{mfreg}}]\label{fockreg}
 If $\la\in\Par_p(n)$ and $\mu\in\rPar(n)$ then $d_{\la\la^\reg}(q)\neq0$, and $d_{\la\mu}(q)=0$ unless $\la^\reg\dom\mu$.
\end{thm}

\begin{pf}[Proof of \cref{rockreg}]
The \lcnamecref{rockreg} asserts that $\la^\reg$ is the partition $\nu\in\pstp\rho d$ defined by
\[
\nu^{(i)}=
\begin{cases}
\la^{(i)}&\text{if }0\ls i\ls\ell-2,
\\
\la^{(\ell-1)}+{\la^{(\ell)}}'&\text{if }i=\ell-1,
\\
\vn&\text{if }i=\ell.
\end{cases}
\]
\cref{fockreg} shows that $\la^\reg$ is the most dominant $p$-strict partition $\mu$ for which $d_{\la\mu}(q)\neq0$. So to show that $\la^\reg=\nu$ we must show that $d_{\la\nu}(q)\neq0$, and that if $\mu\dom\nu$ with $d_{\la\mu}(q)\neq0$ then $\mu=\nu$.

Showing that $d_{\la\nu}(q)\neq0$ is straightforward: in order to obtain a non-zero summand in the formula in \cref{fockrock}, we must take $\si^{(i)}=\la^{(i)}$ for $0\ls i\ls\ell-1$, $\tau^{(i)}=\vn$ for $1\ls i\ls\ell-1$, and $\tau^{(\ell)}=\la^{(\ell)}$, giving $d_{\la\nu}(q)=q^{2|\la^{(\ell)}|}$.

Now take a $p$-strict partition such that $\mu\dom\nu$ and $d_{\la\mu}(q)\neq0$. From (\ref{EDSameBlock}), $\mu$ must lie in $\rpstp\rho d$. Choose partitions $\si^{(i)},\tau^{(i)}$ for which the summand in \cref{fockrock} is non-zero. We assume for the rest of the proof that $p\gs5$; a minor modification is needed when $p=3$, which we leave to the reader.

In view of \cref{succdom}, the assumption that $\mu\dom\nu$ means that
\[
|\mu^{(0)}|+\dots+|\mu^{(r)}|\ls|\la^{(0)}|+\dots+|\la^{(r)}|
\]
for $0\ls r\ls\ell-2$. On the other hand, the non-vanishing of the polynomial $K^{-1}_{\la^{(0)}\si^{(0)}}(-q^2)$ and of the Littlewood--Richardson coefficients $\lr{\la^{(i)}}{\si^{(i)},\tau^{(i)}}$ and $\lr{\mu^{(i-1)}}{\si^{(i-1)},{\tau^{(i)}}'}$ implies that
\[
|\mu^{(0)}|+\dots+|\mu^{(r)}|=|\la^{(0)}|+\dots+|\la^{(r)}|+|\tau^{(r+1)}|
\]
for $0\ls r\ls\ell-1$. So $|\tau^{(1)}|=\dots=|\tau^{(\ell-1)}|=0$ and $|\tau^{(\ell)}|=|\la^{(\ell)}|$. Again by the non-vanishing of the Littlewood--Richardson coefficients it then follows that $\tau^{(1)}=\dots=\tau^{(\ell-1)}=\vn$, while $\tau^{(\ell)}=\la^{(\ell)}$. This in turn gives $\si^{(i)}=\mu^{(i)}$ for $0\ls i\ls\ell-2$, and $\si^{(i)}=\la^{(i)}$ for $1\ls i\ls\ell-1$, so that
\begin{itemize}
\item
$K^{-1}_{\la^{(0)}\mu^{(0)}}(t)\neq0$,
\item
$\mu^{(i)}=\la^{(i)}$ for $1\ls i\ls\ell-2$,
\item
$\lr{\mu^{(\ell-1)}}{\la^{(\ell-1)},{\la^{(\ell)}}'}\neq0$.
\end{itemize}
In particular, $|\mu^{(i)}|=|\nu^{(i)}|$ for all $i$, so (again using \cref{succdom}) the assumption $\mu\dom\nu$ amounts to the statement that $\mu^{(i)}\dom\nu^{(i)}$ for all $i$. But now the only way that $K^{-1}_{\la^{(0)}\mu^{(0)}}(t)=K^{-1}_{\nu^{(0)}\mu^{(0)}}(t)$ can be non-zero is if $\mu^{(0)}=\nu^{(0)}=\la^{(0)}$. A standard result about Littlewood--Richardson coefficients is that the most dominant partition $\xi$ for which $\lr\xi{\la^{(\ell-1)},{\la^{(\ell)}}'}\neq0$ is $\la^{(\ell-1)}+{\la^{(\ell)}}'$, so we also obtain $\mu^{(\ell-1)}=\nu^{(\ell-1)}$, and therefore $\mu=\nu$.
\end{pf}

\subsection{Gelfand--Graev induction}

Our aim is to explore the relationship between the characters $\prj\mu$ and $\mprj\mu$ by considering a third set of projective characters obtained by inducing the projective character $\ch\rho$ along special words which we call thick Gelfand--Graev words. Recall the induction operators $F_i$ from \cref{branchsec}.

Given $i\in J$ and $k\gs1$, we define the corresponding \emph{thick Gelfand--Graev word} (cf.\ \cite[(4.2.1)]{kl})
\begin{align}
\label{EGGW}
\bg^{i,k}&:=\ell^k(\ell-1)^{2k}\dots(i+1)^{2k}i^k\dots1^k0^{2k}1^k\dots i^k
\\
\intertext{and the corresponding induction operator}
\fg ik&:=F_i^k\dots F_1^kF_0^{2k}F_1^k\dots F_i^kF_{i+1}^{2k}\dots F_{\ell-1}^{2k}F_\ell^k.
\label{EFIK}
\end{align}
We want to know what these operators do to characters in a RoCK block.

\begin{rmk}\label{RDP}
We could define divided power induction operators $F_i^{(r)}:=\frac{F_i^r}{r!}$ and use them in place of the usual powers in the definition of $\fg ik$. This would produce slightly simpler formulas in \cref{ggop,lprjform} below but would not make things any easier, since a priori, $F_i^{(r)}$ is defined on the Grothendieck groups with scalars extended from $\Z$ to $\Q$ (although one can check, using \cite[Lemma 22.3.15]{KBook} for the case of large $p$, that in fact $F_i^{(r)}$ is always defined on the Grothendieck groups without extending scalars; we will not pursue this).
\end{rmk}

\begin{propn}\label{ggop}
Take $i\in J$, $\la\in\stp\rho c$ and $\al\in\stp\rho{c+k}$, where $k\geq 1$ and $c+k\ls d$. Then $\ch\al$ occurs in $\fg ik\ch\la$ \iff the \pbq $(\al^{(0)},\dots,\al^{(\ell)})$ is obtained from $(\la^{(0)},\dots,\la^{(\ell)})$ by adding $k$ nodes in components $i$ and $i+1$, with no two nodes added in the same column of component $i$ or in the same row of component $i+1$. If $\al$ satisfies this condition, define
\[
f(\la,\al)=\bigl|\lset{c\gs1}{\al^{(0)}\setminus\la^{(0)}\text{ contains a node in column $c$ but not in column $c+1$}}\bigr|.
\]
Then
%
\[
\chm{\fg ik\ch\la}{\ch\al}=2^{f(\la,\al)+\frac12(k(p-2)+h(\la^{(0)})-h(\al^{(0)})+\a\la-\a\al)}(2k)!^{\ell-i}k!^{2i+1}.
\]
\end{propn}

\begin{pf}
First we assume $i>0$.

For $j\in I$, we define a \emph{$j$-hook} to be a set of nodes of the form
\[
\{(r+\ell-j,c+j+1),(r+\ell-j-1,c+j+2),\dots,(r,c+\ell+1),(r,c+\ell+2),\dots,(r,c+j+p)\}
\]
for $r\gs1$ and $c\gs0$ with $p\mid c$. In other words, a $j$-hook is a set of $p$ nodes with residues in the configuration below.
\[
{\footnotesize
\gyoung(:::^\hf;\ell{{\ell{-}1}}^\hf\hdts1001^\hf\hdts{{j{-}2}}{{j{-}1}},::^\hf;{{\ell{-}1}},::'\hf\hf\edts,:;{{j{+}1}},j)
}
\]
In \cite[Section 4.1a]{kl}, Kleshchev and Livesey observe that if $\la\in\stp\rho c$ with $c<d$, then adding a node to the $j$th component of the \pbq of $\la$ corresponds to adding a $j$-hook to $\la$.

By \cref{addbars}, if $\la\in\stp\rho c$ and $\al\in\stp\rho{c+k}$ with $\al\supseteq\la$, then $\al$ can be obtained from $\la$  by adding some $p$-bars. Thus $\al^{(j)}\supseteq\la^{(j)}$ for all $j\in I$. In particular, if $\ch\al$ occurs in $\fg ik\ch\la$, then $\al$ is obtained from $\la$ by adding $j$-hooks (for various values of $j$). But by the branching rule $\al$ is also obtained from $\la$ by adding nodes one at a time, with a specific sequence of residues determined by the definition of $\fg ik$. In particular, the last $k$ nodes added must all have residue $i$, so there must be a strict partition $\be$ with $\la\subset\be\subset\al$ such that $\al\setminus\be$ comprises $k$ nodes of residue $i$.

In any of the individual $j$-hooks comprising $\al\setminus\la$, the last node added must either be the leftmost node of residue $j$, or the rightmost node of residue $j-1$. So the last node added can have residue $i$ only if $j=i$ or $i+1$. Moreover, the assumption that $i>0$ means that the last two nodes added in a given $j$-hook cannot both have residue $i$. So the only way the last $k$ nodes added in reaching $\al$ from $\la$ can all have residue $i$ is if all the added hooks are $i$-hooks or $(i+1)$-hooks, and each of these hooks contains exactly one node of $\al\setminus\be$. In particular, the \pbq of $\al$ is obtained from the \pbq of $\la$ by adding nodes in components $i$ and $i+1$.

If two nodes are added to the same column of $\la^{(i)}$, the corresponding $i$-hooks are diagonally adjacent, as in the following diagram.
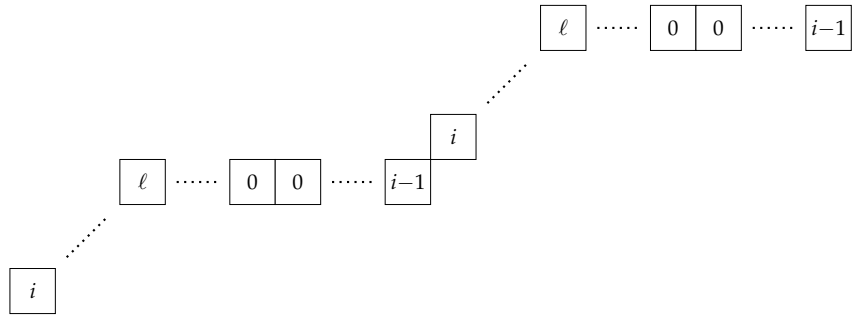
\Yboxdim{17pt}
\[
{\scriptsize
\gyoung(::::::^\hf^\hf^\hf^\hf;\ell^\hf\hdts00^\hf\hdts{{i{-}1}},::::::^\hf^\hf^\hf'\hf\hf\edts,:::::^\hf^\hf^\hf;i,:^\hf;\ell^\hf\hdts00^\hf\hdts{{i{-}1}},:'\hf\hf\edts,i)
}
\]
But now the $i$-hook on the right cannot contain a node of $\al\setminus\be$, because the $i$-node at the left of this hook must be added before the $(i-1)$-node at the right of the hook on the left. This is a contradiction. Similarly, if two nodes are added to the same row of $\la^{(i+1)}$, then the corresponding hooks are horizontally adjacent, and we reach a contradiction in the same way.
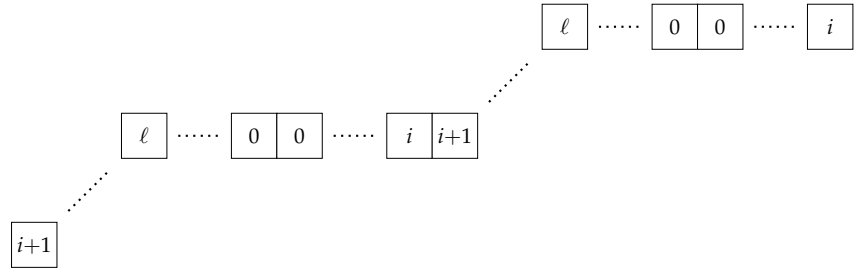
\[
{\scriptsize
\gyoung(::::::^\hf^\hf^\hf^\hf;\ell^\hf\hdts00^\hf\hdts i,::::::^\hf^\hf^\hf'\hf\hf\edts,:^\hf;\ell^\hf\hdts00^\hf\hdts i{{i{+}1}},:'\hf\hf\edts,{{i{+}1}})
}
\]
This is enough to prove the `only if' part of the \lcnamecref{ggop}. For the `if' part, suppose the \pbq of $\al$ is obtained form the \pbq of $\la$ by adding nodes in different columns of $\la^{(i)}$ and in different rows of $\la^{(i+1)}$. To show that $\ch\al$ occurs in $\fg ik\ch\la$, we show that we can get from $\la$ to $\al$ by adding nodes one at a time with the appropriate sequence of residues.

We begin by adding all the $\ell$-nodes in $\al\setminus\la$ (in an arbitrary order), then all the $(\ell-1)$-nodes, and so on, down to the $(i+1)$-nodes. Then we add an $i$-node in each hook, then an $(i-1)$-node in each hook, and so on, working along the arm of each hook, until we add a node of residue $1$ to each hook. Then we add all nodes of residue $0$ in $\al\setminus\la$, and then all remaining nodes of residues $1,\dots,i$ in turn. The assumptions on $\al$ mean that we obtain a strict partition at each stage, so $\ch\al$ does occur in $\fg ik\ch\la$.

The construction in the preceding paragraph enables us to compute the coefficient of $\ch\al$ in $\fg ik\ch\la$. To do this, we need to count possible orders in which the nodes of $\al\setminus\la$ can be added to $\la$ with the required sequence of residues, so that the partition obtained at each stage is strict. For each term $F_j^{ak}$ appearing in $\fg ik$, we need to add $ak$ nodes of residue $j$, and it clear that the choice made in the previous paragraph is the only possibility: in order to be able to add the nodes of residue $0$ in a given hook when applying $F_0^{2k}$, we must already have added the nodes of residues $i,i-1,\dots,1$ to the left of the nodes of residue $0$ in that hook. So our only choice is in which order to add the $j$-nodes for each factor $F_j^{ak}$. In each case we have a free choice, except for the factor $F_0^{2k}$: here in each hook the leftmost $0$-node must be added before the rightmost one. So the number of choices of order is
\[
k!\times\prod_{j=i+1}^{\ell-1}(2k)!\times\prod_{i=1}^ik!^2\times\frac{(2k)!}{2^k}=\frac{k!^{2i+1}(2k)!^{\ell-i}}{2^k}.
\]
It remains to consider the coefficients $a_{\la\mu}$ appearing in the branching rule. Because $i>0$, the assumptions on $\al$ give $\al^{(0)}=\la^{(0)}$, which in turn implies that $h(\la)=h(\al)$; therefore, as we go from $\la$ to $\al$ by adding nodes, the partitions obtained alternate between even and odd. So the number of times we pass from an odd partition to an even partition is $\frac12(kp+\a\la-\a\al)$. This yields
\[
\chm{\fg ik\ch\la}{\ch\al}=2^{\frac12(k(p-2)+\a\la-\a\al)}(2k)!^{\ell-i}k!^{2i+1},
\]
which agrees with the \lcnamecref{ggop} because $\la^{(0)}=\al^{(0)}$.

Now we consider the case where $i=0$. Now in order for for $\ch\al$ to appear in $\fg ik\ch\la$, it must be the case that $\al$ is obtained from $\la$ by adding $j$-hooks, and now there must exist a strict partition $\be$ with $\la\subset\be\subset\al$ such that $\al\setminus\be$ comprises $2k$ nodes of residue $0$. Arguing as in the previous case, this implies that $\al^{(j)}=\la^{(j)}$ for $j\gs2$, while $\al^{(1)}$ is obtained from $\la^{(1)}$ by adding nodes in distinct rows, and $\al^{(0)}\supseteq\la^{(0)}$. Now if two nodes are added in the same column of $\la^{(0)}$, then the corresponding $0$-hooks are vertically stacked, as in the following diagram.
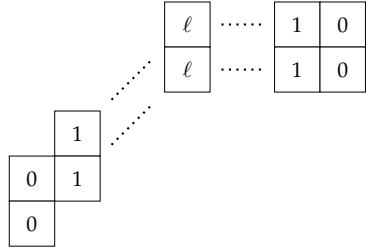
\[
{\scriptsize
\gyoung(::^\hf;\ell^\hf\hdts10,::'\hf\hf\edts\ell^\hf\hdts10,::'\hf\hf\edts/\xhf,:;1,01,0)
}
\]
But now the upper $0$-bar cannot contain any nodes of $\al\setminus\be$, giving a contradiction. So again we find that the nodes added to $\la^{(0)}$ to obtain $\al^{(0)}$ must be added in distinct columns.

Now suppose $\al$ satisfies the conditions, and consider how we can obtain $\al$ from $\la$ by applying $\fg0k:=F_0^{2k}F_1^{2k}\dots F_{\ell-1}^{2k}F_\ell^k$. For each of the residues $j=\ell,\ell-1,\dots,1$, we can add the $j$-nodes of $\al\setminus\la$. In each added $1$-hook, the two $0$-nodes must be added in order from left to right, but otherwise there are no restrictions on the $1$-hooks. The $0$-nodes occurring in the added $0$-hooks can be added in any order, except that when two added $0$-hooks correspond to nodes in consecutive columns of $\al^{(0)}$, then the rightmost $0$-node of the left hook is adjacent to the leftmost $0$-node of the right hook (as in one of the following diagrams) so that these two nodes must be added in a specific order.
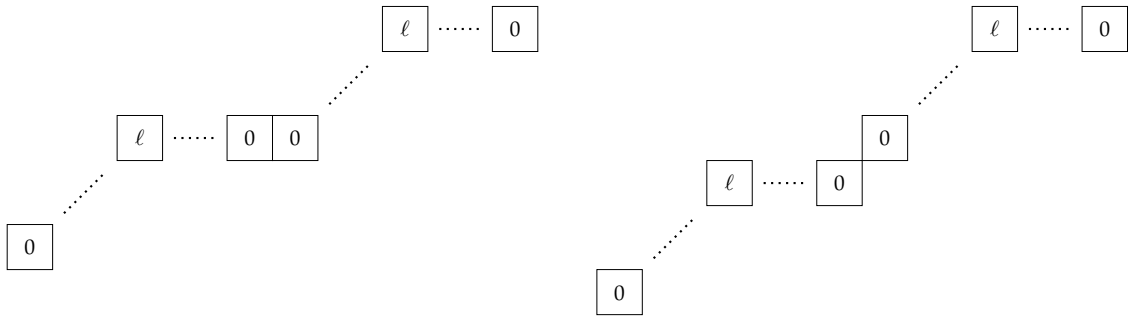
\[
{\scriptsize
\gyoung(::::^\hf^\hf^\hf;\ell^\hf\hdts0,::::^\hf^\hf'\hf\hf\edts,:^\hf;\ell^\hf\hdts00,:'\hf\hf\edts,0)
}\qquad
{\scriptsize
\gyoung(::::^\hf^\hf^\hf;\ell^\hf\hdts0,::::^\hf^\hf'\hf\hf\edts,:^\hf:^\hf:;0,:^\hf;\ell^\hf\hdts0,:'\hf\hf\edts,0)
}
\]
As a result, we obtain a coefficient $k!(2k)!^\ell/2^{k-f(\la,\al)}$. But we also need to take into account the coefficients coming from the branching rule: the partitions obtained as we add nodes alternate between even and odd, except when we add a node in column $1$. So we obtain a further factor $2^{\frac12(kp+\a\la-\a\al+h(\la)-h(\al))}$. Putting these coefficients together, we obtain
\[
\chm{\fg0k\ch\la}{\ch\al}=2^{f(\la,\al)+\frac12(k(p-2)+h(\la)-h(\al)+\a\la-\a\al)}(2k)!^{\ell}k!,
\]
in agreement with the \lcnamecref{ggop}.
\end{pf}

\subsection{Projective characters obtained by induction}

Our aim is to explore the relationship between the characters $\prj\mu$ and $\mprj\mu$, which we do by considering a third set of projective characters.

Recall from \cref{SSPar} the set $\pfstp\rho d\subseteq \stp\rho d$ of the \pfr-partitions in $\pstp\rho d$. By \cref{pbqlem}(ii), a partition $\la\in \pstp\rho d$ is $p'$ if and only if $\la^{(0)}=\vn$. 
Recall (\ref{EFIK}). 
Given $\la\in\pfstp\rho d$, we will define  a projective character $\lprj\la$ by inducing the projective character $\ch\rho$:
\begin{equation}\label{EPhiTilde}
\lprj\la=\prod_{i=1}^\ell\prod_{r=1}^{\la^{(i)}_1}\fg{i-1}{{\la^{(i)}}'_r}\,\ch\rho\in \operatorname{PCh}^{\rho,d},
\end{equation}
where the factors $\fg{i-1}{{\la^{(i)}}'_r}$ can be taken in any order. (It is not obvious at this stage that $\lprj\la$ is independent of the order of the factors, but we will see   in \cref{lspan}(ii)  that this is the case. For now, we define $\lprj\la$ by fixing an arbitrary order for each $\la$.)

For any strict partition $\pi$ and any composition $\ga$ let $\bar{c}(\pi;\ga)$ be the number of ways $\pi$ can be obtained from $\vn$ by adding at each step $\ga_i$ nodes all in different columns such that each step a strict partition is obtained. Now given $\al\in\stp\rho d$ and $\la\in\pfstp\rho d$, define
\begin{align*}
\widetilde D_\la&=2^{\frac12(d(p-2)+\a\la-\a\rho)}\prod_{i=1}^\ell\prod_{r\gs1}(2{\la^{(i)}}'_r)!^{l-i+1}{\la^{(i)}}'_r!^{2i-1},
\\
\widetilde D_{\la\al}&=\widetilde D_\la\sum_{\substack{\be^{(1)},\dots,\be^{(\ell)}\in\compn\\\ga^{(1)},\dots,\ga^{(\ell)}\in\compn\\\be^{(i)}+\ga^{(i)}={\la^{(i)}}'}}\bar{c}(\al^{(0)};\ga^{(1)})\prod_{i=1}^\ell\left[(\ypm{\be^{(i)}}\otimes\sgn)\tens\ypm{\ga^{(i+1)}}:\spe{\al^{(i)}}\right],
\end{align*}
where we read $\ga^{(\ell+1)}$ as $\vn$. Then we can deduce the following result from \cref{ggop}.

\begin{propn}\label{lprjform}
Suppose $\al\in\stp\rho d$ and $\la\in\pfstp\rho d$. Then $\ch\al$ occurs in $\lprj\la$ \iff $\widetilde D_{\la\al}\neq0$.
Furthermore, if $\al$ is a \pfr-partition, then $\chm{\lprj\la}{\ch\al}=\widetilde D_{\la\al}$.
\end{propn}

\begin{pf}
We construct $\lprj\la$ by starting from $\ch\rho$ and applying each of the operators $\fg{i-1}{{\la^{(i)}}'_r}$, for $1\ls i\ls\ell$ and $1\ls r\ls\la^{(i)}_1$. We start from the \pbq of $\rho$, i.e.\ $(\vn,\dots,\vn)$, and when we apply $\fg{i-1}{{\la^{(i)}}'_r}$, we add ${\la^{(i)}}'_r$ nodes in components $i-1$ and $i$ in accordance with \cref{ggop}, and we consider the possible choices of how to add these nodes. Let $\be^{(i)}_r$ be the number of nodes we add in component $i$, and $\ga^{(i)}_r$ the number of nodes we add in component $i-1$. This defines partitions $\be^{(i)},\ga^{(i)}$ for $1\ls i\ls\ell$ with $\be^{(i)}+\ga^{(i)}={\la^{(i)}}'$, and we need to consider all possible such choices of $\be^{(i)},\ga^{(i)}$. Take a particular choice of $\be^{(i)},\ga^{(i)}$, and consider the coefficient of $\ch\al$ obtained. 
Recall from \cref{ggop} that when we apply $\fg{i-1}{{\la^{(i)}}'_r}$, the nodes added in component $i-1$ must be in distinct columns, and the nodes added in component $i$ must be in distinct rows. So (by the Pieri rule) the number of ways of obtaining the \pbq $(\al^{(0)},\al^{(1)},\dots,\al^{(\ell)})$ is
\begin{align*}
&\bar{c}(\al^{(0)};\ga^{(1)})\prod_{i=1}^\ell\lr{\al^{(i)}}{(\ga^{(i+1)}_1),(\ga^{(i+1)}_2),\dots\ (1^{\be^{(i)}_1}),(1^{\be^{(i)}_2}),\dots}\\
&=\bar{c}(\al^{(0)};\ga^{(1)})\prod_{i=1}^\ell\left[(\ypm{\be^{(i)}}\otimes\sgn)\tens\ypm{\ga^{(i+1)}}:\spe{\al^{(i)}}\right]
\end{align*}
by \cref{msgnm}; here we read $\ga^{(\ell+1)}=\vn$.

We sum over all possible choices of $\be^{(i)},\ga^{(i)}$ to get $\widetilde D_{\la\al}/\widetilde D_\la$; so the coefficient of $\ch\al$ is non-zero \iff $\widetilde D_{\la\al}\neq0$. In the case where $\al$ is a \pfr-partition, the product of the coefficients arising from \cref{ggop} is $\widetilde D_\la$, so the coefficient of $\ch\al$ in $\lprj\la$ is $\widetilde D_{\la\al}$.
\end{pf}

Our next task is to show that the characters $\lprj\la$ are \li. First we use \cref{lprjform} to give more information about the structure of the characters $\lprj\la$. Recall the partial order $\sucq$ on multipartitions from \cref{rouqsec}.

\begin{propn}\label{ldom}
Suppose $\la\in\pfstp\rho d$. Then the character $\ch\la$ occurs in $\lprj\la$, while any character $\ch\al$ occurring in $\lprj\la$ satisfies
\[
(\la^{(0)},\dots,\la^{(\ell)})\preq(\al^{(0)},\dots,\al^{(\ell)})\preq({\la^{(1)}}',\dots,{\la^{(\ell)}}',\vn).
\]
\end{propn}

\begin{pf}
Certainly $\ch\la$ occurs in $\lprj\la$: in the sum in \cref{lprjform} we can take $\be^{(i)}={\la^{(i)}}'$ and $\ga^{(i)}=\vn$ for all $i$; the corresponding summand is then
\begin{align*}
\prod_{i=1}^{\ell-1}\left[\ypm{{\la^{(i)}}'}\otimes\sgn:\spe{\la^{(i)}}\right]&=\prod_{i=1}^{\ell-1}\left[\ypm{{\la^{(i)}}'}\otimes\sgn:\spe{{\la^{(i)}}'}\otimes\sgn\right]
\\
&=\prod_{i=1}^{\ell-1}\left[\ypm{{\la^{(i)}}'}:\spe{{\la^{(i)}}'}\right]
\end{align*}
which is well known to be non-zero (indeed, $\spe{{\la^{(i)}}'}$ is \emph{defined} to be a submodule of $\ypm{{\la^{(i)}}'}$).

Now suppose $\ch\al$ occurs in $\lprj\la$, and choose $\be^{(1)},\dots,\be^{(\ell)},\ga^{(1)},\dots,\ga^{(\ell)}$ such that the corresponding summand in $\widetilde D_{\la\al}$ is non-zero. Then in particular $|\al^{(i)}|=|\be^{(i)}|+|\ga^{(i+1)}|$ for $0\ls i\ls\ell$ (where we read $\be^{(0)}=\ga^{(\ell+1)}=\vn$). To show that $(\al^{(0)},\dots,\al^{(\ell)})\sucq(\la^{(0)},\dots,\la^{(\ell)})$, take $0\ls k\ls\ell$ and $c\gs1$. Then
\begin{align*}
&\left(\sum_{i=0}^{k-1}|\al^{(i)}|+\sum_{i=1}^c{\al^{(k)}}'_i\right)-\left(\sum_{i=0}^{k-1}|\la^{(i)}|+\sum_{i=1}^c{\la^{(k)}}'_i\right)
\\
&=|\ga^{(k)}|+\sum_{i=1}^c{\al^{(k)}}'_i-\sum_{i=1}^c{\la^{(k)}}'_i
\\
&\gs|\ga^{(k)}|+\sum_{i=1}^c(\be^{(k)}\sqcup{\ga^{(k+1)}}')_i-\sum_{i=1}^c\be^{(k)}_i-\sum_{i=1}^c\ga^{(k)}_i
\\
&\gs\sum_{i=1}^c(\be^{(k)}\sqcup{\ga^{(k+1)}}')_i-\sum_{i=1}^c\be^{(k)}_i
\\
&\gs0,
\end{align*}
as required.

To show that $(\al^{(0)},\dots,\al^{(\ell)})\preq({\la^{(1)}}',\dots,{\la^{(\ell)}}',\vn)$, take $0\ls k\ls\ell$ and $c\gs1$. Then
\begin{align*}
&\left(\sum_{i=0}^{k-1}|\la^{(i+1)}|+\sum_{i=1}^c\la^{(k+1)}_i\right)-\left(\sum_{i=0}^{k-1}|\al^{(i)}|+\sum_{i=1}^c{\al^{(k)}}'_i\right)
\\
&=|\be^{(k)}|+\sum_{i=1}^c\la^{(k+1)}_i-\sum_{i=1}^c{\al^{(k)}}'_i
\\
&\gs|\be^{(k)}|+\sum_{i=1}^c({\be^{(k+1)}}'\sqcup{{\ga^{(k+1)}}')}_i-\sum_{i=1}^c(\be^{(k)}+{\ga^{(k+1)}}')_i
\\
&\gs\sum_{i=1}^c((\be^{(k+1)})'\sqcup{\ga^{(k+1)}}')_i-\sum_{i=1}^c{\ga^{(k+1)}}'_i
\\
&\gs0,
\end{align*}
as required.
\end{pf}

As a consequence, we can show that the characters $\lprj\la$ span the space of virtual projective characters, and derive some information about the form of the indecomposable projective characters.

\needspace{4em}
\begin{cory}\label{lspan}\indent
\begin{enumerate}
\vspace*{-\topsep}
\item\label{lpspan}
The set $\lset{\lprj\la}{\la\in\pfstp\rho d}$ is a basis for the space of virtual projective characters in $\blk\rho d$.
\item
For each $\la\in\pfstp\rho d$, the  character $\lprj\la$ is independent of the order of the factors $\fg{i-1}{{\la^{(i)}}'_r}$.
\item
There is a bijection $\la\mapsto\bak\la$ from $\pfstp\rho d$ to $\rpstp\rho d$ such that $\ch\la$ occurs in $\prj{\bak\la}$, and any  character $\ch\al$ occurring  in $\prj{\bak\la}$  satisfies $\al\unlhd\la$.
\end{enumerate}
\end{cory}


\begin{pf}
(i) 
Since $|\pfstp\rho d|=|\rpstp\rho d|$  by (\ref{EPartId}), 
it suffices to show that the $\lprj\la$ are \li. But this follows from \cref{ldom} which shows that the matrix giving the multiplicities $\chm{\lprj\la}{\ch\al}$ for $\al,\la\in\pfstp\rho d$ is triangular with non-zero diagonal.

(ii)
Let $\lprj\la$ be defined using a particular choice of order of the factors $\fg{i-1}{({\la^{(i)}}'_r)}$, and let $\lprj{\la^\ast}$ be defined in the same way but using a different order. By \cref{lprjform}, $\lprj\la-\lprj{\la^\ast}$ is a linear combination of the characters $\ch\al$ with $\al$ \emph{not} being \pfr. By (\ref{lpspan}) we can write $\lprj\la-\lprj{\la^\ast}$ as a linear combination of the characters $\lprj\xi$ with $\xi\in\pfstp\rho d$.  If this linear combination is non-zero, then take $\xi$ maximal in the dominance order such that $\lprj\xi$ appears with non-zero coefficient. Then by \cref{ldom} the character $\ch\xi$ occurs in $\lprj\la-\lprj{\la^\ast}$, a contradiction.

(iii)
Since $\lprj\la$ is a character (not just a virtual character), it can be written as a linear combination, with non-negative coefficients, of the indecomposable projective characters. 
Since $\ch\la$ occurs in $\lprj\la$, it must occur in some indecomposable constituent $\prj{\bak\la}$ of $\lprj\la$. Then if $\ch\al$ occurs in $\prj{\bak\la}$ it must occur in $\lprj\la$, giving $\al\domby\la$.

This defines a map $\pfstp\rho d\to \rpstp\rho d,\ \la\mapsto\bak\la$ with the required properties. This map is obviously injective, and hence bijective  since $|\pfstp\rho d|=|\rpstp\rho d|$  by (\ref{EPartId}).  
\end{pf}

\subsection{The bijection $\la\mapsto\bak\la$}
In \cref{lspan}(iii), we have defined the bijection 
\[
\pfstp\rho d\to \rpstp\rho d,\qquad \la\mapsto\bak\la
\]
such that $\ch\la$ occurs in $\prj{\bak\la}$, and any character $\ch\al$ occurring in $\prj{\bak\la}$ satisfies $\al\unlhd\la$.
The goal of this subsection is to prove \cref{expbij} which 
describes the bijection explicitly. 
To prove this proposition,  we consider weights of modules, as outlined in \cref{branchsec}. We fix a weight $\bi^\rho$ of~$\Dm\rho$.
Recalling (\ref{EGGW}), for any $\la\in\pfstp\rho d$ and $j\in J$, define the word $\bg^{j,\la}$ to be the concatenation
\[
  \bg^{j,\la}:=   \bg^{j,{\la^{(j+1)}}'_1}\,\bg^{j,{\la^{(j+1)}}'_2}\,\bg^{j,{\la^{(j+1)}}'_3}\,\dots .
\]
Now define $\bg^\la$ to be the concatenation
\[
  \bg^\la:=\bi^\rho\,\bg^{\ell-1,\la}\,\bg^{\ell-2,\la}\,\dots\,\bg^{0,\la}. 
\]

\begin{lemma}\label{LBrGG}
Let $\mu\in \stp\rho d$. Then $\bg^\la$ is a weight of $\Sp\mu$ \iff $\ch\mu$ occurs in $\lprj\la$.
\end{lemma}
\begin{pf}
For a word $\bi=i_1\dots i_n\in I^n$, we denote $E_\bi:=E_{i_1}\dots E_{i_n}$ and $F_\bi:=F_{i_n}\dots F_{i_1}$. Then by definition, $\bg^\la$ is a weight of $\Sp\mu$ if and only if $E_{\bg^\la}\Sp\mu\neq 0$ if and only if $E_{\bg^\la}\ch\mu\neq 0$. But $E_{\bg^\la}=E_{\bi^\rho}E_{\bg^{\ell-1,\la}}\dots E_{\bg^{0,\la}}$, so $E_{\bg^\la}\ch\mu\neq 0$ \iff $E_{\bg^{\ell-1,\la}}\dots E_{\bg^{0,\la}}\ch\mu=c\ch\rho$ for some non-zero scalar $c$. By Frobenius reciprocity, this is equivalent to the fact that $\ch\mu$ occurs in $F_{\bg^{0,\la}}\dots F_{\bg^{\ell-1,\la}}\ch\rho$. Recalling the definition (\ref{EPhiTilde}) of $\lprj\la$ and taking into account \cref{lspan}(ii), we deduce that $F_{\bg^{0,\la}}\dots F_{\bg^{\ell-1,\la}}\ch\rho=\lprj\la$, completing the proof of the lemma. 
\end{pf}

Given $\mu\in\rpstp\rho d$, define $\tilde\mu\in\rpstp\rho {d-|\mu^{(0)}|}$ to be the partition with $p$-bar-core $\rho$ and \pbq 
$(\vn,\mu^{(1)},\dots,\mu^{(\ell-1)},\vn);$ 
in other words, $\tilde\mu$ is defined by deleting from $\mu$ all the parts divisible by $p$ from $\mu$, cf.\ \cref{pbqlem}(iii). 

Given $\la\in\pfstp\rho d$, define $\hat\la\in\pfstp\rho {d-|\la^{(1)}|}$ to be the partition with $p$-bar-core $\rho$ and \pbq 
$(\vn,\vn,\la^{(2)},\dots,\la^{(\ell)}).$

\begin{lemma}\label{weightclaim}
Suppose $\mu\in\rpstp\rho d$ and $\la\in\pfstp\rho d$. If $\mu^{(0)}={\la^{(1)}}'$ and $\bg^{\hat\la}$ is a weight of $\Dm{\tilde\mu}$, then $\bg^\la$ is a weight of $\Dm\mu$.
\end{lemma}

\begin{pf}
%
By \cref{wtprop} it suffices to show that we can get from $\mu$ to $\tilde\mu$ by successively removing normal nodes, with the residues of these nodes giving the word $\bg^{0,\mu^{(0)}_1}\bg^{0,\mu^{(0)}_2}\dots$. We use induction on $|\mu^{(0)}|$, with the case $\mu^{(0)}=\vn$ being vacuous. For the inductive step, suppose $\mu^{(0)}\neq\vn$. Let $\mu^-$ be the partition obtained from $\mu$ by deleting the last positive part divisible by $p$; call this last part $k=\mu^{(0)}_{h(\mu^{(0)})}$. Similarly, define $\la^-$ by deleting the last non-zero column from $\la^{(1)}$. Then $\tilde\mu=\widetilde{\mu^-}$ and $\hat\la=\widehat{\la^-}$, so by induction if $\bg^{\hat\la}$ is a weight of $\Dm{\tilde\mu}$ then $\bg^{\la^-}$ is a weight of $\Dm{\mu^-}$. So we just need to show that we can get from $\mu$ to $\mu^-$ by removing $2k$ normal $0$-nodes, then $2k$ normal $1$-nodes, \dots, $2k$ normal $(\ell-1)$-nodes, and finally $k$ normal $\ell$-nodes. In fact, to do this it suffices to look at the first $kp$ columns of $\mu$. By assumption $\mu$ has at least one part equal to $kp$, so let $r$ be maximal such that $\mu_r=kp$, and let $h=h(\mu)$. Then (because $\rho$ is $d$-Rouquier) the integers $\mu_{r+1},\dots,\mu_h$ are simply the integers $a<kp$ which are congruent to $1,\dots,\ell$ modulo $p$. So $\mu$ has removable $0$-nodes in columns $1,p,p+1,2p,\dots,kp$. These are normal, and we define a smaller partition by removing them; specifically, if we remove them in order from right to left, then each node remains normal until it is removed. Now rows $r,\dots,h-1$ of the resulting partition are the integers $a<kp$ which are congruent to $2,\dots,\ell$ or $-1$ modulo $p$. This means there are removable $1$-nodes in columns $2,p-1,p+2,2p-1,\dots,kp-1$. These nodes are normal, and we remove them (again, in order from right to left). Now rows $r,\dots,h-1$ of the resulting partition are the integers $a<kp$ which are congruent to $1,3,\dots,\ell$ or $-2$ modulo $p$. We continue in this way, removing at the final step normal $\ell$-nodes in columns $l+1,2l+1,\dots,kp-l$. In the partition resulting after this final step, rows $r,\dots,h-1$ are the integers $a<kp$ which are congruent to $1,\dots,\ell$ modulo $p$, in other words, the integers $\mu_{r+1},\dots,\mu_h$. So the overall effect is just to have deleted the part $kp$, and we have the partition $\mu^-$, as required.
\end{pf}

\begin{propn}\label{expbij}
Suppose $\la\in\pfstp\rho d$. Then $\bak\la$ is the partition with \pbq $({\la^{(1)}}',\dots,{\la^{(\ell)}}',\vn)$.
\end{propn}
\begin{pf}
The defining properties of the bijection $\la\mapsto\bak\la$, together with Brauer reciprocity (\ref{EBrRec}), show that the composition factors of $\Sp\la$ lie among the irreducible supermodules $\Dm{\bak\ka}$ for which $\ka\dom\la$, and include $\Dm{\bak\la}$ at least once.

Now consider weights. By \cref{LBrGG}, $\bg^\la$ is a weight of $\Sp\ka$ \iff $\ch\ka$ occurs in $\lprj\la$, which, by \cref{ldom} and \cref{succdom},  happens only if $\ka\domby\la$. So if $\ka\doms\la$ then $\bg^\la$ is not a weight of $\Sp\ka$, and in particular is not a weight of $\Dm{\bak\ka}$. So $\bg^\la$ is not a weight of any composition factor of $\Sp\la$ except possibly $\Dm{\bak\la}$; but $\bg^\la$ is a weight of $\Sp\la$, so it must be a weight of~$\Dm{\bak\la}$.

So we can characterize the bijection $\mu\mapsto\bak\la$ recursively by the conditions
\begin{equation}\label{lcharact}
\bak\la\notin\lset{\bak\ka}{\ka\domsby\la},\qquad \bg^\la\text{ is a weight of }\Dm{\bak\la}.
\end{equation}

Now to prove the \lcnamecref{expbij} we use induction on $d$. For given $d$, we consider first the partitions $\la$ for which $\la^{(1)}\neq\vn$. For these partitions we use induction on the dominance order; so we assume that the \lcnamecref{expbij} is true if $\la$ is replaced with any partition $\ka\domsby\la$ (observe by \cref{succdom} that if $\la$ and $\ka$ are \pfr-partitions with $\la^{(1)}\neq\vn$ and $\ka\domsby\la$, then $\ka^{(1)}\neq\vn$ as well).

Given $\la$ with $\la^{(1)}\neq\vn$, let $\hat\la$ be the partition with \pbq $(\vn,\vn,\la^{(2)},\dots,\la^{(\ell)})$ as above. By induction on $d$ we know that $\fd{\hat\la}$ is the partition with \pbq $(\vn,{\la^{(2)}}',\dots,{\la^{(\ell)}}',\vn)$. In particular, $D(\vn,{\la^{(2)}}',\dots,{\la^{(\ell)}}',\vn)$ has $\bg^{\hat\la}$ as a weight. Now \cref{weightclaim} shows that $\bg^\la$ is a weight of $D({\la^{(1)}}',\dots,{\la^{(\ell)}}',\vn)$. By induction we know the partitions $\bak\ka$ for $\ka\domsby\la$, in particular we know that none of them has \pbq $({\la^{(1)}}',\dots,{\la^{(\ell)}}',\vn)$. So from the characterization (\ref{lcharact}), $\bak\la$ must be the partition with \pbq $({\la^{(1)}}',\dots,{\la^{(\ell)}}',\vn)$.

So (for our fixed $d$) we can assume the \lcnamecref{expbij} is true whenever $\la^{(1)}\neq\vn$. In particular, this means that
\[
\lset{\bak\ka}{\ka\in\pfstp\rho d,\ \ka^{(1)}\neq\vn}=\lset{\mu}{\mu\in\rpstp\rho d,\ \mu^{(0)}\neq\vn},
\]
and therefore
\begin{equation}\label{indstep}
\lset{\bak\ka}{\ka\in\pfstp\rho d,\ \ka^{(1)}=\vn}=\lset{\mu}{\mu\in\rpstp\rho d,\ \mu^{(0)}=\vn}.
\end{equation}

Now we deal with partitions $\la$ for which $\la^{(1)}=\vn$. For these partitions, we use induction with a different order: we write $\ka\bktq\la$ if $(\vn,\vn,{\ka^{(2)}}',\dots,{\ka^{(\ell)}}')\sucq(\vn,\vn,{\la^{(2)}}',\dots,{\la^{(\ell)}}')$, and we assume that 
the \lcnamecref{expbij} is true if $\la$ is replaced by any $\ka$ for which $\ka\bkt\la$.

Now (\ref{indstep}) shows that there is $\ka$ with $\ka^{(1)}=\vn$ such that $\bak\ka$ is the partition with \pbq $(\vn,{\la^{(2)}}',\dots,{\la^{(\ell)}}',\vn)$. By \cref{regthm} and Brauer reciprocity (\ref{EBrRec})  we know that $\ch{\bak\ka}$ occurs in $\prj{\bak\ka}$, and therefore occurs in $\lprj\ka$. Then \cref{ldom} gives
\[
(\vn,{\la^{(2)}}',\dots,{\la^{(\ell)}}',\vn)\preq(\vn,{\ka^{(2)}}',\dots,{\ka^{(\ell)}}',\vn),
\]
which is the same as saying $\ka\bktq\la$. But if $\ka\bkt\la$ then we know by induction that $\bak\ka$ is the partition with \pbq $(\vn,{\ka^{(2)}}',\dots,{\ka^{(\ell)}}',\vn)$, a contradiction. So $\ka=\la$, and we are done.
\end{pf}

\subsection{Adjustment matrix}\label{comparingsec}

Now we can return to the virtual projective characters $\mprj\mu$. First we express the characters $\lprj\la$ in terms of the characters $\mprj\mu$.

\begin{propn}\label{lfromm}
Suppose $\la\in\pfstp\rho d$. Then
\[
\lprj\la=\widetilde D_\la\sum_{\mu\in\rpstp\rho d}\prod_{i=1}^\ell\left[\ypm{{\la^{(i)}}'}:\spe{\mu^{(i-1)}}\right]\mprj\mu.
\]
\end{propn}

\begin{pf}
\cref{ldom} shows that $\lprj\la$ is determined among all virtual projective characters in $\blk\rho d$ by the coefficients $\chm{\lprj\la}{\ch\al}$ for $\al\in\pfstp\rho d$. So we fix $\al\in\pfstp\rho d$, and we just need to show that the coefficient of $\ch\al$ on each side of the equation is the same.

Using \cref{mattproj} together with the assumption that $\al^{(0)}=\vn$, we find that the coefficient of $\ch\al$ on the right-hand side of the equation can be written as $\widetilde D_\la X_{\la\al}$, where
\[
X_{\la\al}=\sum_{\mu\in\rpstp\rho d}\prod_{i=1}^\ell\left[\ypm{{\la^{(i)}}'}:\spe{\mu^{(i-1)}}\right]\sum_{\si^{(\bullet)},\tau^{(\bullet)}}\prod_{i=1}^\ell\lr{\al^{(i)}}{\si^{(i)},\tau^{(i)}}\lr{\mu^{(i-1)}}{\si^{(i-1)},{\tau^{(i)}}'}.
\]
Here and throughout this proof, $\sum_{\si^{(\bullet)},\tau^{(\bullet)}}$ means that we sum over all $\si^{(1)},\dots,\si^{(\ell-1)},\tau^{(1)},\dots,\tau^{(\ell)}\in\calp$, and we read $\si^{(0)}$ and $\si^{(\ell)}$ as $\vn$.

Summing over $\mu\in\rpstp\rho d$ is equivalent to summing over $\mu^{(0)},\dots,\mu^{(\ell-1)}\in\calp$ (because if $|\mu^{(0)}|+\dots+|\mu^{(\ell-1)}|\neq d$ then the summand is zero anyway). So we can write
\begin{align*}
X_{\la\al}&=\sum_{\si^{(\bullet)},\tau^{(\bullet)}}\prod_{i=1}^\ell\lr{\al^{(i)}}{\si^{(i)},\tau^{(i)}}\sum_{\mu^{(i-1)}\in\calp}\left[\ypm{{\la^{(i)}}'}:\spe{\mu^{(i-1)}}\right]\left[\spe{\si^{(i-1)}}\tens\spe{{\tau^{(i)}}'}:\spe{\mu^{(i-1)}}\right]
\\
&=\sum_{\si^{(\bullet)},\tau^{(\bullet)}}\prod_{i=1}^\ell\lr{\al^{(i)}}{\si^{(i)},\tau^{(i)}}\sum_{\substack{\be^{(i)},\ga^{(i)}\in\compn\\\be^{(i)}+\ga^{(i)}={\la^{(i)}}'}}\left[\ypm{\be^{(i)}}:\spe{{\tau^{(i)}}'}\right]\left[\ypm{\ga^{(i)}}:\spe{\si^{(i-1)}}\right]
\\
&=\sum_{\si^{(\bullet)},\tau^{(\bullet)}}\prod_{i=1}^\ell\sum_{\substack{\be^{(i)},\ga^{(i)}\in\compn\\\be^{(i)}+\ga^{(i)}={\la^{(i)}}'}}\left[\ypm{\be^{(i)}}\otimes\sgn:\spe{\tau^{(i)}}\right]\left[\ypm{\ga^{(i)}}:\spe{\si^{(i-1)}}\right]\left[\spe{\tau^{(i)}}\tens\spe{\si^{(i)}}:\spe{\al^{(i)}}\right],
\end{align*}
where for the second equality we use \cref{mackey}. 
Since we interpret $\si^{(0)}$ as $\vn$, the term $[\ypm{\ga^{(1)}}:\spe{\si^{(0)}}]$ equals $1$ if $\ga^{(1)}=\vn$, and $0$ otherwise. But if $\ga^{(1)}\neq\vn$, then (comparing the sizes of all the partitions involved) the product of the remaining terms in the summand is zero anyway. So we can simply omit the term $[\ypm{\ga^{(1)}}:\spe{\si^{(0)}}]$. Since we read $\si^{(\ell)}$ as $\vn$, we can also add a harmless factor $[\ypm{\ga^{(\ell+1)}}:\spe{\si^{(\ell)}}]$ in which we interpret $\ga^{(\ell+1)}$ as $\vn$. Now (with a shift of variable) $X_{\la\al}$ becomes
\begin{align*}
&\phantom=\sum_{\si^{(\bullet)},\tau^{(\bullet)}}\sum_{\substack{\be^{(1)},\dots,\be^{(\ell)}\in\compn\\\ga^{(1)},\dots,\ga^{(\ell)}\in\compn\\\be^{(i)}+\ga^{(i)}={\la^{(i)}}'}}\prod_{i=1}^\ell\left[\ypm{\be^{(i)}}\otimes\sgn:\spe{\tau^{(i)}}\right]\left[\ypm{\ga^{(i+1)}}:\spe{\si^{(i)}}\right]\left[\spe{\tau^{(i)}}\tens\spe{\si^{(i)}}:\spe{\al^{(i)}}\right],
\\
&=\sum_{\substack{\be^{(1)},\dots,\be^{(\ell)}\in\compn\\\ga^{(1)},\dots,\ga^{(\ell)}\in\compn\\\be^{(i)}+\ga^{(i)}={\la^{(i)}}'}}\sum_{\si^{(\bullet)},\tau^{(\bullet)}}\prod_{i=1}^\ell\left[(\ypm{\be^{(i)}}\otimes\sgn)\boxtimes\ypm{\ga^{(i+1)}}:\spe{\tau^{(i)}}\boxtimes\spe{\si^{(i)}}\right]\left[\spe{\tau^{(i)}}\tens\spe{\si^{(i)}}:\spe{\al^{(i)}}\right]
\\
&=\sum_{\substack{\be^{(1)},\dots,\be^{(\ell)}\in\compn\\\ga^{(1)},\dots,\ga^{(\ell)}\in\compn\\\be^{(i)}+\ga^{(i)}={\la^{(i)}}'}}\prod_{i=1}^\ell\left[(\ypm{\be^{(i)}}\otimes\sgn)\tens\ypm{\ga^{(i+1)}}:\spe{\al^{(i)}}\right]
\end{align*}
where in the final equality we use transitivity of induction and the fact that for any $t,s$ the irreducible $\C(\sym_t\times\sym_s)$-modules are precisely the modules $\spe\tau\boxtimes\spe\si$ for $\tau\in\calp(t)$ and $\si\in\calp(s)$. 

So $\widetilde D_\la X_{\la\al}$ coincides with the coefficient $\widetilde D_{\la\al}$ from \cref{lprjform} (bearing in mind that $\al^{(0)}=\vn$), and the proof is complete.
\end{pf}

We are now ready to prove the main result of this section. First we need some more notation. Recall from above the bijection
\begin{align*}
\pfstp\rho d\longmapsto\rpstp\rho d,
\ 
\la\longmapsto\bak\la
\end{align*}
where $\bak\la$ is the partition with \pbq $({\la^{(1)}}',\dots,{\la^{(\ell)}}',\vn)$. We write $\mu\mapsto\fd\mu$ for the inverse bijection.

%
%
%

Given $\la\in\pstp\rho d$, we define $\gth\la$ to be the composition $(|\la^{(0)}|,\dots,|\la^{(\ell)}|)\in\compn(d)$. Given a composition $\um=(m_0,\dots,m_{\ell-1},0)$ of $d$, we define $\fd \um$ to be the composition $(0,m_0,\dots,m_{\ell-1})$.

Now say that a virtual character $\psi$ is \emph{$\um$-bounded} if:
\begin{itemize}
\item
every $\ch\al$ occurring in $\psi$ satisfies $\um\dom\gth\al\dom\fd \um$,
\item
there is at least one $\ch\al$ occurring in $\psi$ with $\gth\al=\um$, and
\item
there is at least one $\ch\al$ occurring in $\psi$ with $\gth\al=\fd \um$.
\end{itemize}

Say that a virtual character $\psi$ is \emph{$\um$-semi-bounded} if:
\begin{itemize}
\item
every $\ch\al$ occurring in $\psi$ satisfies $\gth\al\dom\fd \um$, and
\item
there is at least one $\ch\al$ occurring in $\psi$ with $\gth\al=\fd \um$.
\end{itemize}


We make the following observations about the virtual characters we have defined. We begin with the virtual characters $\mprj\mu$.

\begin{lemma}\label{mattprop}
Suppose $\mu\in\rpstp\rho d$, and let $\um=\gth\mu$. Then $\mprj\mu$ is $\um$-bounded.
\end{lemma}

\begin{pf}
Suppose $\ch\al$ occurs in $\mprj\mu$. Then by \cref{mattproj}, there are partitions $\si^{(0)},\dots,\si^{(\ell-1)},\tau^{(1)},\dots,\tau^{(\ell)}$ such that
\[
\ki{\al^{(0)}}{\si^{(0)}}\prod_{i=1}^{\ell}\lr{\al^{(i)}}{\si^{(i)},\tau^{(i)}}\lr{\mu^{(i-1)}}{\si^{(i-1)},{\tau^{(i)}}'}\neq0
\]
(where as usual we read $\si^{(\ell)}=\vn$). Using the fact that $\ki\pi\si$ is non-zero only if $|\pi|=|\si|$, while $\lr\be{\ga,\de}$ is non-zero only if $|\be|=|\ga|+|\de|$, this gives
\begin{equation}\label{eq1}
|\mu^{(0)}|+\dots+|\mu^{(r-1)}|+|\si^{(r)}|=|\al^{(0)}|+\dots+|\al^{(r)}|=|\mu^{(0)}|+\dots+|\mu^{(r)}|-|\tau^{(r+1)}|
\end{equation}
for each $r=0,1,\dots,\ell$ (interpreting $|\tau^{(\ell+1)}|$ as $0$), so that
\[
(|\mu^{(0)}|,\dots,|\mu^{(\ell-1)}|,0)\dom(|\al^{(0)}|,\dots,|\al^{(\ell)}|)\dom(0,|\mu^{(0)}|,\dots,|\mu^{(\ell-1)}|)
\]
as required.

Now use \cref{kostlem} to choose a strict partition $\nu^{(0)}$ for which $\ki{\nu^{(0)}}{\mu^{(0)}}\neq0$, and let $\nu$ be the partition in $\stp\rho d$ with \pbq $(\nu^{(0)},\mu^{(1)},\dots,\mu^{(\ell-1)},\vn)$. Then $\gth\nu=m$ and $\ch\nu$ occurs in $\mprj\mu$ (the only non-zero summand is for $\tau^{(i)}=\vn$ and $\sigma^{(i)}=\mu^{(i)}$ for every $i$).

Finally, let $\la=\fd\mu$. Then $\gth\la=\um$ and $\ch\la$ occurs in $\mprj\mu$ (the only non-zero summand is for $\tau^{(i)}={\mu^{(i-1)}}'$ and $\si^{(i)}=\vn$ for every $i$).
\end{pf}

Next we look at the characters $\lprj\la$.

\begin{lemma}\label{luciaprop}
Suppose $\la\in\pfstp\rho d$ and let $\um=(|\la^{(1)}|,\dots,|\la^{(\ell)}|,0)$. Then $\lprj\la$ is $\um$-bounded.
\end{lemma}

\begin{pf}
The fact that $\um\dom\gth\al\dom\fd \um$ whenever $\ch\al$ occurs in $\lprj\la$ is just a cruder version of the second statement in \cref{ldom}. Furthermore, the first statement in \cref{ldom} says that $\ch\la$ occurs in $\lprj\la$, and by definition $\gth\la=\fd \um$. Finally, let $\xi$ be the partition in $\stp\rho d$ with \pbq $((|\la^{(1)}|),{\la^{(2)}}',\dots,{\la^{(\ell)}}',\vn)$. Then the coefficient $\widetilde D_{\la\xi}$ from \cref{lprjform} is non-zero: to see this, observe that the summand in which $\be^{(i)}=\vn$ and $\ga^{(i)}={\la^{(i)}}'$ for each $i$ equals $\bar c((|\la^{(1)}|),{\la^{(1)}}')=1$.
Hence $\ch\xi$ occurs in $\lprj\la$, and satisfies $\gth\xi=\um$.
\end{pf}

Finally we look at the indecomposable projective characters $\prj\mu$.

\begin{lemma}\label{pimprop}
Suppose $\mu\in\rpstp\rho d$, and let $\um=\gth\mu$. Then $\prj\mu$ is $\um$-semi-bounded.
\end{lemma}

\begin{pf}
Let $\la=\fd\mu$. Then $\gth\la=\fd \um$, and \cref{lspan}(iii)  says that $\prj\mu$ is $\um$-semi-bounded.
\end{pf}

We now have a lot of information about our three families of virtual characters. Take a composition $\um=(m_0,\dots,m_{\ell-1},0)$ of $d$, and let
\[
\rp{{\um}}=\lset{\mu\in\rpstp\rho d}{\gth\mu=\um}.
\]
Now define
\begin{align*}
\prm&=\lset{\prj\mu}{\mu\in\rp{{\um}}},
\\
\lprm&=\lset{\lprj\la}{\bak\la\in\rp{{\um}}},
\\
\mprm&=\lset{\mprj\mu}{\mu\in\rp{{\um}}}.
\end{align*}
Each of these sets is a linearly independent set of virtual characters, of size $|\rp{{\um}}|$. The virtual characters in $\lprm$ and $\mprm$ are $\um$-bounded, while the virtual characters in $\prm$ are $\um$-semi-bounded.

Now we can finally make the connection between the characters $\prj\mu$ and the virtual characters~$\mprj\mu$.

\begin{thm}\label{projthm}
Take a composition $\um=(m_0,\dots,m_{\ell-1},0)$ of $d$. Then for each $\mu,\bak\la\in\rp{{\um}}$ we can write
\[
\prj\mu=\sum_{\nu\in\rp{{\um}}}A_{\nu\mu}\mprj\nu\qquad\text{and}\qquad\lprj{\la}=\sum_{\nu\in\rp{{\um}}}B_{\nu\la}\prj\nu
\]
where:
\begin{itemize}
\item
$A_{\nu\mu},B_{\nu\la}\in\N_0$ for each $\nu$;
\item
$A_{\mu\mu}>0$ and $B_{\bak\la\la}>0$;
\item
if $A_{\nu\mu}>0$ resp. $B_{\nu\la}>0$, then $\nu\dom\mu$ resp. $\nu\dom\bak\la$.
\end{itemize}
Moreover, each character $\prj\mu\in\prm$ is $\um$-bounded.
\end{thm}

\begin{pf}
We use induction on $\um$ in decreasing dominance order. So assume the \lcnamecref{projthm} is true whenever $\um$ is replaced by a composition $\un\doms \um$.

Since $\bak\la\in\rp{{\um}}$, the character $\lprj\la$ is $\um$-bounded by \cref{luciaprop}. Because $\lprj\la$ is a projective character (not just a virtual projective character), it is a linear combination, with non-negative coefficients, of the characters $\prj\nu$ for $\nu\in\rpstp\rho d$. We claim that only the characters $\prj\nu$ for $\nu\in\rp{{\um}}$ can occur, that is that $\lprj{\la}=\sum_{\nu\in\rp{{\um}}}B_{\nu\la}\prj\nu$ for some non-negative integer coefficients $B_{\nu\la}$. By \cref{pimprop} any other character $\prj\psi$ that occurs is $\un$-semi-bounded for some $\un\neq \um$: if $\un\ndom \um$, then there is a character $\ch\al$ occurring in $\prj\psi$ for which $\gth\al=\fd{\un}\ndom \fd \um$, so $\prj\psi$ cannot be a constituent of $\lprj\la$ because $\lprj\la$ is $\um$-bounded. On the other hand, if $\un\doms \um$, then by induction $\prj\psi$ is $\un$-bounded, so includes a character $\ch\al$ with $\gth\al=\un\doms \um$, so again the fact that $\lprj\la$ is $\um$-bounded means that $\prj\psi$ does not appear in $\lprj\la$. This proves our claim. 

By the previous paragraph, the span of $\lprm$ equals the span of $\prm$. On the other hand, \cref{lfromm} shows that the span of $\lprm$ equals the span of $\mprm$. So the span of $\prm$ equals the span of $\mprm$. In particular each character $\prj\mu\in\prm$ is a linear combination of the virtual characters in $\mprm$, that is there are coefficients $A_{\nu\mu}$ such that
\[
\prj\mu=\sum_{\nu\in\rp{{\um}}}A_{\nu\mu}\mprj\nu
\]
for each $\mu\in\rp{{\um}}$. 

But now observe from \cref{mattproj} that each character $\mprj\nu\in\mprm$ includes exactly one character $\ch\al$ with $\gth\al=\fd \um$, namely the partition $\al=\fd\nu$, and that $\chm{\mprj\nu}{\ch{{\fd\nu}}}=1$. To see this note that by \eqref{eq1} we need to take $\si^{(i)}=\vn$ and then $\tau^{(i)}={\nu^{(i-1)}}'$ in order to have a non-zero summand in the formula for $\chm{\mprj\nu}{\ch{{\fd\nu}}}$. 

So for each $\mu,\nu$, the coefficient $A_{\mu\nu}$ is simply the coefficient $\chm{\prj\mu}{\ch{{\fd\nu}}}$. These coefficients are certainly non-negative integers because $\prj\mu$ is a character, and \cref{lspan}(iii) shows that if $A_{\nu\mu}\neq0$ then $\fd\nu\domby\fd\mu$. Since $\gth\mu=\gth\nu$, this condition is the same as saying $\nu\dom\mu$. This triangularity property also gives $A_{\mu\mu}\neq0$ for each $\mu$, because the characters $\prj\mu$ are \li. 

Furthermore, by \cref{lfromm},
\[\sum_{\psi\in\rp{{\um}}} A_{\nu\psi}B_{\psi\la}=\widetilde D_\la\prod_{i=0}^{\ell-1}\left[\ypm{{\bak\la}^{(i)}}:\spe{\nu^{(i)}}\right].\]
It then follows by \cite[Theorem 4.13]{JamesBook} that 
$\nu\dom\bak\la$ whenever $B_{\nu\la}>0$, and then also that $B_{\bak\la\la}>0$ since by the previous paragraph $A_{\bak\la\psi}>0$ only if $\psi
\trianglelefteqslant\bak\la$.

The final statement of the \lcnamecref{projthm} now follows for $\um$: each $\mprj\nu\in\mprm$ is $\um$-bounded, and is a non-negative linear combination of irreducible characters, which means that any non-zero non-negative linear combination of the $\mprj\nu$ will also be $m$-bounded; so in particular $\prj\mu$ is $\um$-bounded.
\end{pf}

We extend the definition of the integers $A_{\nu\mu}$ to all $\mu,\nu\in\rpstp\rho d$ by setting $A_{\nu\mu}=0$ when $\gth\mu\neq\gth\nu$. Then the matrix $A$ with entries $A_{\nu\mu}$ is a non-singular square matrix with non-negative integer entries, which is triangular with respect to the dominance order. We call $A$ the \emph{adjustment matrix} for $\blk\rho d$. \cref{projthm} shows that the (super)decomposition matrix for $\blk\rho d$ can be obtained from the matrix determined by the characters $\mprj\mu$ by post-multiplying by $A$. Our aim in the remainder of the paper is to show that $A$ is the identity matrix when $d<p$ (and $\blk\rho d$ is RoCK).

\section{Cartan matrices and proof of the main theorem}


\subsection{The super-Cartan matrix and the adjustment matrix}\label{cartansec}

In this subsection and the next we consider the entries of the super-Cartan matrix of $\blk\rho d$.

Recall from (\ref{EHatPhi}) and \cref{mattproj} that we have 
integers
\[
D_{\la\mu}=2^{\lfloor\frac12(h(\la^{(0)})+1-\a\la)\rfloor}\sum_{\substack{\si^{(0)},\dots,\si^{(\ell-1)}\in\calp\\\tau^{(1)},\dots,\tau^{(\ell)}\in\calp}}\ki{\la^{(0)}}{\si^{(0)}}\prod_{i=1}^{\ell}\lr{\la^{(i)}}{\si^{(i)},\tau^{(i)}}\lr{\mu^{(i-1)}}{\si^{(i-1)},{\tau^{(i)}}'}
\]
for $\la\in\stp\rho d$ and $\mu\in\rpstp\rho d$, such that $\mprj\mu=\sum_\la D_{\la\mu}\ch\la$ for each $\mu$. We also have from \cref{comparingsec} an adjustment matrix $A$ such that $\prj\mu=\sum_\nu A_{\nu\mu}\mprj\nu$ for each $\mu$. Hence
\[
[\prj\mu:\ch\la]=\sum_{\nu\in\rpstp\rho d}A_{\nu\mu}D_{\la\nu}
\]
for every $\la\in\stp\rho d,\mu\in\rpstp\rho d$. So by Brauer reciprocity (\ref{EBrRec}), 
if we define
\[
\dd_{\la\mu}=2^{\lfloor\frac12(h(\la^{(0)})+\a\la)\rfloor}\sum_{\substack{\si^{(0)},\dots,\si^{(\ell-1)}\in\calp\\\tau^{(1)},\dots,\tau^{(\ell)}\in\calp}}\ki{\la^{(0)}}{\si^{(0)}}\prod_{i=1}^\ell\lr{\la^{(i)}}{\si^{(i)},\tau^{(i)}}\lr{\mu^{(i-1)}}{\si^{(i-1)},{\tau^{(i)}}'}
\]
for all $\la,\mu$, then
\begin{equation}\label{adjd}
[\Sp\la:\Dm\mu]=\sum_{\nu\in\rpstp\rho d}A_{\nu\mu}\dd_{\la\nu}.
\end{equation}

Now consider the super-Cartan matrix entries
\[
C_{\nu\mu}=[P(\nu):\Dm\mu]
\]
for $\mu,\nu\in\rpstp\rho d$, where $P(\nu)$ denotes the projective cover of $\Dm\nu$. From above, we can write
\[
C_{\nu\mu}=\sum_{\la\in\stp\rho d}[\prj\nu:\ch\la][\Sp\la:\Dm\mu]=\sum_{\substack{\la\in\stp\rho d\\\xi,\pi\in\rpstp\rho d}}A_{\xi\mu}\dd_{\la\xi}A_{\pi\nu}D_{\la\pi}.
\]


\subsection{Entries in the unadjusted Cartan matrix}
Our objective in this subsection is to compute the `unadjusted super-Cartan matrix' entries 
\[
\mathring C_{\mu\nu}:=\sum_\la\dd_{\la\mu}D_{\la\nu}
\]
for $\mu,\nu\in\rpstp\rho d$. In \cref{SSMainTProof}, we will then use \cref{wrcartan} and \cref{mainkl} to see that these `unadjusted super-Cartan matrix' entries coincide with the actual super-Cartan matrix entries, which will imply our main result.

\begin{propn}\label{unadjc}
Suppose $\mu,\nu\in\rpstp\rho d$. Then
\[
\mathring C_{\mu\nu}
=\prod_{i=0}^{\ell-1}\lr{\mu^{(i)}}{\phi^{(i)},\chi^{(i)},{\psi^{(i+1)}}',\om^{(i)}}\lr{\nu^{(i)}}{\phi^{(i)},\psi^{(i)},{\chi^{(i+1)}}',\om^{(i)}},
\]
summing over all partitions $\phi^{(i)},\psi^{(i)},\chi^{(i)},\om^{(i)}$ for $0\ls i\ls\ell-1$ such that $\chi^{(0)}={\psi^{(0)}}'$, and reading $\chi^{(\ell)}=\psi^{(\ell)}=\vn$.
\end{propn}

\begin{pf}
From the definition of the integers $D_{\la\mu}$ and $\dd_{\la\mu}$, we obtain
\begin{alignat*}2
\mathring C_{\mu\nu}
&=\sum_{\la\in\stp\rho d}\dd_{\la\mu}D_{\la\nu}&&
\\
&=\sum\Big(2^{h(\la^{(0)})}\ki{\la^{(0)}}{\si^{(0)}}\ki{\la^{(0)}}{\bar\si^{(0)}}&&\prod_{i=1}^\ell\lr{\la^{(i)}}{\si^{(i)},\tau^{(i)}}\lr{\mu^{(i-1)}}{\si^{(i-1)},{\tau^{(i)}}'}
\\
&&\mathllap{\times}&\prod_{i=1}^\ell\lr{\la^{(i)}}{\bar\si^{(i)},\bar\tau^{(i)}}\lr{\nu^{(i-1)}}{\bar\si^{(i-1)},\bar\tau{{}^{(i)}}'}\Big)
\end{alignat*}
where the sum is over $\la\in\stp\rho d$, $\si^{(i)},\bar\si^{(i)}\in\calp$ for $0\ls i\ls\ell-1$ and $\tau^{(i)},\bar\tau^{(i)}\in\calp$ for $1\ls i\ls\ell$, and we read $\si^{(\ell)}=\bar\si^{(\ell)}=\vn$. Summing over $\la\in\stp\rho d$ is equivalent to summing over $\la^{(0)},\dots,\la^{(\ell)}\in\calp$ with $\la^{(0)}$ strict and $|\la^{(0)}|+\dots+|\la^{(\ell)}|=d$. But in fact the summand is zero when $|\la^{(0)}|+\dots+|\la^{(\ell)}|\neq d$, so we can safely replace the variable $\la\in\stp\rho d$ with variables $\la^{(0)}\in\calp_0$ and $\la^{(i)}\in\calp$ for $1\ls i\ls\ell$.

We apply \cref{smackey} to eliminate the variables $\la^{(1)},\dots,\la^{(\ell)}$. We obtain
\begin{align*}
\mathring C_{\mu\nu}=\sum\Big(2^{h(\la^{(0)})}\ki{\la^{(0)}}{\si^{(0)}}\ki{\la^{(0)}}{\bar\si^{(0)}}&\prod_{i=1}^\ell\lr{\si^{(i)}}{\phi^{(i)},\chi^{(i)}}\lr{\tau^{(i)}}{\psi^{(i)},\om^{(i-1)}}
\\
\times&\prod_{i=1}^\ell\lr{\bar\si^{(i)}}{\phi^{(i)},\psi^{(i)}}\lr{\bar\tau^{(i)}}{\chi^{(i)},\om^{(i-1)}}
\\
\times&\prod_{i=1}^\ell\lr{\mu^{(i-1)}}{\si^{(i-1)},{\tau^{(i)}}'}\lr{\nu^{(i-1)}}{\bar\si^{(i-1)},\bar\tau{{}^{(i)}}'}\Big),
\end{align*}
where we have eliminated the variables $\la^{(i)}$ from the summation, and introduced new variables $\phi^{(i)},\psi^{(i)},\chi^{(i)}\in\calp$ for $1\ls i\ls\ell$ and $\om^{(i)}$ for $0\ls i\ls\ell-1$. Now we use standard relations for Littlewood--Richardson coefficients to get
\begin{align*}
\mathring C_{\mu\nu}=\sum\Big(&2^{h(\la^{(0)})}\ki{\la^{(0)}}{\si^{(0)}}\ki{\la^{(0)}}{\bar\si^{(0)}}\lr{\mu^{(0)}}{\si^{(0)},{\psi^{(1)}}',{\om^{(0)}}'}\lr{\nu^{(0)}}{\bar\si^{(0)},{\chi^{(1)}}',{\om^{(0)}}'}
\\
&\times\prod_{i=2}^\ell\lr{\mu^{(i-1)}}{\phi^{(i-1)},\chi^{(i-1)},{\psi^{(i)}}',{\om^{(i-1)}}'}\lr{\nu^{(i-1)}}{\phi^{(i-1)},\psi^{(i-1)},{\chi^{(i)}}',{\om^{(i-1)}}'}\Big).
\end{align*}
Here we have eliminated the variables $\si^{(i)},\bar\si^{(i)}$ for $1\ls i\ls\ell-1$ and $\tau^{(i)},\bar\tau^{(i)}$ for $1\ls i\ls\ell$. We have also elided the terms $\lr{\si^{(\ell)}}{\phi^{(\ell)},\chi^{(\ell)}}\lr{\bar\si^{(\ell)}}{\phi^{(\ell)},\psi^{(\ell)}}$. This is harmless because we interpret $\si^{(\ell)}=\bar\si^{(\ell)}=\vn$, but it means that we now eliminate $\phi^{(\ell)},\chi^{(\ell)},\psi^{(\ell)}$ as variables, reading $\psi^{(\ell)}=\chi^{(\ell)}=\vn$ in the formula above.

Now we apply \cref{klem} to eliminate the terms $\ki{\la^{(0)}}{\si^{(0)}}\ki{\la^{(0)}}{\bar\si^{(0)}}$. We get
\begin{align*}
\mathring C_{\mu\nu}
=\sum\Big(&\lr{\si^{(0)}}{\phi^{(0)},{\psi^{(0)}}'}\lr{\bar\si^{(0)}}{\phi^{(0)},\psi^{(0)}}\lr{\mu^{(0)}}{\si^{(0)},{\psi^{(1)}}',{\om^{(0)}}'}\lr{\nu^{(0)}}{\bar\si^{(0)},{\chi^{(1)}}',{\om^{(0)}}'}
\\
&\times\prod_{i=2}^\ell\lr{\mu^{(i-1)}}{\phi^{(i-1)},\chi^{(i-1)},{\psi^{(i)}}',{\om^{(i-1)}}'}\lr{\nu^{(i-1)}}{\phi^{(i-1)},\psi^{(i-1)},{\chi^{(i)}}',{\om^{(i-1)}}'}\Big),
\end{align*}
where we have eliminated the variable $\la^{(0)}$ and introduced two new variables $\phi^{(0)},\psi^{(0)}\in\calp$. Now we eliminate $\si^{(0)},\bar\si^{(0)}$ and obtain
\begin{align*}
\mathring C_{\mu\nu}
=\sum\Big(&\lr{\mu^{(0)}}{\phi^{(0)},{\psi^{(0)}}',{\psi^{(1)}}',{\om^{(0)}}'}\lr{\nu^{(0)}}{\phi^{(0)},\psi^{(0)},{\chi^{(1)}}',{\om^{(0)}}'}
\\
&\times\prod_{i=1}^{\ell-1}\lr{\mu^{(i)}}{\phi^{(i)},\chi^{(i)},{\psi^{(i+1)}}',{\om^{(i)}}'}\lr{\nu^{(i)}}{\phi^{(i)},\psi^{(i)},{\chi^{(i+1)}}',{\om^{(i)}}'}\Big).
\end{align*}
Replacing $\om^{(i)}$ with ${\om^{(i)}}'$ for each $i$ gives the required result.
\end{pf}

\subsection{Proof of the main theorem.}\label{SSMainTProof}
Let $d<p$ and $\rho$ be a $d$-Rouquier $p$-bar-core. Take $\la\in\stp\rho d$ and $\mu\in\rpstp\rho d$. Our main theorem asserts that the decomposition number $[\Sp\la:\Dm\mu]$ equals the integer $\dd_{\la\mu}$ defined in \cref{cartansec}. We have seen in (\ref{adjd}) that $[\Sp\la:\Dm\mu]=\sum_\nu A_{\nu\mu}\dd_{\la\nu}$, so our task is to show that the adjustment matrix $A$ is the identity matrix. 

Recall from \cref{cartansec} that for genuine  super-Cartan matrix entries we have 
\[
C_{\nu\mu}=\sum_{\xi\pi\in\rpstp\rho d}A_{\xi\mu}\mathring C_{\xi\pi}A_{\pi\nu},
\]
and the unadjusted super-Cartan matrix entries $\mathring C_{\xi\pi}$ are given by \cref{unadjc}.
The matrix $A$ is triangular with non-negative integer entries, which implies that $C_{\nu\mu}\gs\mathring C_{\nu\mu}$ for all $\mu,\nu$, with equality for all $\mu,\nu$ \iff $A$ is the identity matrix. More simply, $A$ is the identity matrix \iff $\sum_{\mu,\nu}C_{\nu\mu}=\sum_{\mu,\nu}\mathring C_{\nu\mu}$.

Assume first that $\blk\rho d$ is of type $\Mtype$. Then simple modules are the same as simple supermodules, and indecomposable projective modules are the same as indecomposable projective supermodules, so the entries of the usual  Cartan matrix are given by $C_{\nu\mu}$.

Assume next that $\blk\rho d$ is of type $\Qtype$. Then when we look at modules rather than supermodules, each $\prj\mu$ splits as a sum $\prj{\mu,+}\oplus\prj{\mu,-}$ and each simple module $\Dm\mu$ splits as a direct sum $\Dm{\mu,+}\oplus\Dm{\mu,-}$. If we restrict to the double cover $\halt_n$ of the alternating group, then $\prj{\mu,+}$ and $\prj{\mu,-}$ both restrict to the same indecomposable projective character $\prj{\mu,0}$, and the simple modules $\Dm{\mu,+}$ and $\Dm{\mu,-}$ both restrict to the same simple module $\Em{\mu,0}$. So $\Res_{\halt_n}\prj\mu=2\prj{\mu,0}$ and $\Res_{\halt_n}\Dm\mu\cong\Em{\mu,0}^{\oplus 2}$, and it follows that the entries of the usual Cartan matrix of the block $\blk\rho d_\0$ of $\halt_n$ are given by $C_{\nu\mu}$.

Now consider the wreath product algebra $\W_d$ from \cref{wreathsec}. 
By \cref{mainkl}, 
this algebra is Morita equivalent to $\blk\rho d$ if $\blk\rho d$ is of type $\Mtype$ or to $\blk\rho d_\0$ if $\blk\rho d$ is of type $\Qtype$. In either case the Cartan matrix of $\W_d$ and the matrix $(C_{\nu,\mu})$ are the same 
up to row and column permutations; 
that is, there is a bijection $\theta:\rpstp\rho d\to\Par^J(d)$ such that $C_{\nu\mu}=[P(\theta(\nu)):L(\theta(\mu))]$ for all $\mu,\nu$. Summing over $\mu,\nu$, we obtain
\[
\sum_{\mu,\nu\in\rpstp\rho d}C_{\nu\mu}=\sum_{\bla,\bmu\in\Par^J(d)}[P(\bla):L(\bmu)].
\]
But a comparison of \cref{unadjc,wrcartan} shows that if we define a bijection
\begin{align*}
\iota:\rpstp\rho d\longrightarrow\Par^J(d), 
\ 
\mu\longmapsto(\la^1,\dots,\la^\ell)
\end{align*}
where
\[
\la^i=
\begin{cases}
\mu^{(i-1)}&\text{$i$ even}
\\
{\mu^{(i-1)}}'&\text{$i$ odd}
\end{cases}
\]
then $\mathring C_{\nu\mu}=[P(\iota(\nu)),L(\iota(\mu))]$ for all $\mu,\nu$. Summing, we obtain
\[
\sum_{\mu,\nu\in\rpstp\rho d}\mathring C_{\nu\mu}=\sum_{\bla,\bmu\in\Par^J(d)}[P(\bla):L(\bmu)]=\sum_{\mu,\nu\in\rpstp\rho d}C_{\nu\mu}
\]
and the result follows.


\begin{thebibliography}{ABC}

\backrefparscanfalse

\bibitem[BK${}_1$]{BKHeCl} J.~Brundan \& A.~Kleshchev, Hecke-Clifford superalgebras, crystals of type $A^{(2)}_{2\ell}$ and modular branching rules for $\tilde{S}_n$, {\em Represent. Theory} {\bf5} (2001), 317--403.\bkp

\bibitem[BK${}_2$]{bk1}J. Brundan \& A. Kleshchev, Projective representations of symmetric groups via Sergeev duality, {\em Math.\ Z.} 239 (2002), 27--68.\bkp

\bibitem[BK${}_3$]{bk2}J. Brundan \& A. Kleshchev, James' regularization theorem for double covers of symmetric groups, \textit{J.\ Algebra} \textbf{306} (2006), 128--137.\bkp

\bibitem[BK${}_4$]{bk3}J.\ Brundan \& A.\ Kleshchev, Odd Grassmannian bimodules and derived equivalences for spin symmetric groups, \texttt{arXiv:2203.14149}.\bkp

\bibitem[CK]{CK}J.~Chuang and R.~Kessar, Symmetric groups, wreath products, Morita equivalences, and Brou{\'e}'s abelian defect group conjecture, \emph{Bull.~London Math.~Soc.}~\textbf{34} (2002), 174--184.\bkp

\bibitem[CR]{CR}J. Chuang and R. Rouquier,
Derived equivalences for symmetric groups and $\mathfrak{sl}_2$-categorification, {\em Ann.\ of Math.} {\bf 167} (2008), 245--298.\bkp

\bibitem[CT${}_1$]{ctq}J.\ Chuang \& K.\ M.\ Tan, Some canonical basis vectors in the basic $U_q(\widehat{\mathfrak{sl}}_n)$-module, \textit{J.\ Algebra} \textbf{248} (2002), 765--779.\bkp

\bibitem[CT${}_2$]{CT1}J. Chuang \& K. M. Tan, Representations of wreath products of algebras, {\em Math.\ Proc.\ Cambridge Philos.\ Soc.} {\bf 135} (2003), 395--411.\bkp

\bibitem[ELV]{elv}M.\ Ebert, A.\ Lauda \& L.\ Vera, Derived superequivalences for spin symmetric groups and odd $\mathfrak{sl}(2)$-categorifications, \texttt{arXiv:2203.14153}.\bkp

\bibitem[Fa${}_1$]{mfreg}M.\ Fayers, $q$-analogues of regularisation theorems for linear and projective representations of the symmetric group, \textit{J.\ Algebra} \textbf{316} (2007), 346--367.\bkp

\bibitem[Fa${}_2$]{mfspinalt2}M.\ Fayers, Irreducible projective representations of the alternating group which remain irreducible in characteristic~$2$, \textit{Adv.\ Math.} \textbf{374} (2020) 107340.\bkp

\bibitem[Fa${}_3$]{mattconj}M.\ Fayers, Comparing Fock spaces in types $A^{(1)}$ and $A^{(2)}$, \textit{Algebr.\ Comb.}, to appear; \texttt{arXiv:2207.01879}.\bkp

\bibitem[Fu]{fulton}W.\ Fulton, \textit{Young tableaux}, London Mathematical Society Student Texts \textbf{35}, Cambridge University Press, Cambridge, 1997.\bkp

\bibitem[H]{humph} J. Humphreys, Blocks of projective representations of the symmetric groups, \textit{J.\ London Math.\ Soc.} 33 (1986), 441--452.\bkp

\bibitem[J${}_1$]{j1}
G.\ D.\ James, On the decomposition matrices of the symmetric groups II, \textit{J.\ Algebra} \textbf{43} (1976), 45--54.\bkp

\bibitem[J${}_2$]{JamesBook}G.\ D. James, {\em The Representation Theory of the Symmetric Groups}, Lecture Notes in Mathematics, Vol. {\bf 682}, Springer, New York/Heidelberg/Berlin, 1978.\bkp

\bibitem[JK]{JK}G.\ D. James and A. Kerber, {\em The Representation Theory of the Symmetric Group}, Encyclopedia of Mathematics and Its Applications, Vol. {\bf 16}, Addison-Wesley, Reading, MA, 1981.\bkp

\bibitem[Ke]{Kessar}
R. Kessar, 
Blocks and source algebras for the double covers of the symmetric and alternating groups, \textit{J. Algebra} \textbf{186}  (1996), 872--933.\bkp

\bibitem[Kl]{KBook}A.\ Kleshchev, {\em Linear and Projective Representations of Symmetric Groups}, Cambridge University Press, 2005.\bkp

\bibitem[KL]{kl}A.\ Kleshchev \& M.\ Livesey, RoCK blocks for double covers of symmetric groups and quiver Hecke superalgebras, \textit{Mem.\ Amer.\ Math.\ Soc.}, to appear; \texttt{arXiv:2201.06870}.\bkp

\bibitem[KS]{ks}A.\ Kleshchev \& V.\ Shchigolev, Modular branching rules for projective representations of symmetric groups and lowering operators for the supergroup $Q(n)$, \textit{Mem.\ Amer.\ Math.\ Soc.} \textbf{220} (2012), no.\ 1034, xviii+123 pp.\bkp

\bibitem[LM]{LM}B.\ Leclerc \& H.\ Miyachi, Some closed formulas for canonical bases of Fock spaces, \textit{Represent.\ Theory} \textbf{6} (2002), 290--312.\bkp

\bibitem[LT]{lt}B.\ Leclerc \& J.--Y. Thibon, $q$-deformed Fock spaces and modular representations of spin symmetric groups, \textit{J.\ Physics A} \textbf{30} (1997), 6163--6176.\bkp

\bibitem[M${}_1$]{MacD}I. G. Macdonald, Polynomial functors and wreath products, {\em J.\ Pure Appl.\ Algebra} {\bf 18} (1980), 173--204.\bkp

\bibitem[M${}_2$]{macdbook}I.~Macdonald, Symmetric functions and Hall polynomials, 2nd Edition, Oxford University Press, 1995.\bkp

\bibitem[MY${}_1$]{my} A.~Morris \& A.~Yaseen, Some combinatorial results involving shifted Young diagrams, \textit{Math.\ Proc.\ Cambridge Philos.\ Soc.} 99 (1986), 23--31.\bkp

\bibitem[MY${}_2$]{my2} A.~Morris \& A.~Yaseen, Decomposition matrices for spin characters of symmetric groups, \textit{Proc.\ Roy. Soc. Edinburgh} {\bf 108A} (1988), 145--164.\bkp

\bibitem[N]{Na} Nazarov, M. L. Young's orthogonal form of irreducible projective representations of the symmetric group, {\em J. London Math.\ Soc.} {\bf 42} (1990) 437--451.\bkp

\bibitem[R]{Ro}R. Rouquier, Repr\'esentations et cat\'egories d\'eriv\'ees, Rapport  d'habilitation, Universit\'e de Paris VII, 1998.\bkp

\bibitem[S]{stem}J.~Stembridge, Shifted tableaux and the projective representations of symmetric groups, \textit{Adv. Math.} \textbf{74} (1989), 87--134.\bkp

\end{thebibliography}
\end{document}